\documentclass[11pt,reqno]{amsart}
\usepackage{style}
\usepackage[foot]{amsaddr}
\usepackage[margin=1in]{geometry}

\title{The Power of Two Matrices in Spectral Algorithms for Community Recovery}
\author{Souvik Dhara{$^\star$}, Julia Gaudio{$^\dagger$}, Elchanan Mossel{$^\ddagger$}, Colin Sandon{$^\mid$}}
\address[{$^\star$}]{School of Industrial Engineering, Purdue University}
\address[{$^\ddagger$}]{Department of Industrial Engineering and Management Sciences, Northwestern University}
\address[{$^\dagger$}]{Department of Mathematics, Massachusetts Institute of Technology}
\address[{$\mid$}]{Department of Mathematics, {\'E}cole Polytechnique F{\'e}d{\'e}rale de Lausanne}
\email{\href{sdhara@purdue.edu}{sdhara@purdue.edu}, \href{mailto:julia.gaudio@northwestern.edu}{julia.gaudio@northwestern.edu}} 
\email{\href{mailto:elmos@mit.edu}{elmos@mit.edu}, \href{mailto:colin.sandon@epfl.ch} {colin.sandon@epfl.ch}}
\thanks{\emph{Acknowledgement:} 
The authors sincerely thank the anonymous reviewers for their helpful feedback. S.D., E.M and C.S. were partially supported by Vannevar Bush Faculty Fellowship ONR-N00014-20-1-2826. S.D.~was supported by Simons-Berkeley Research Fellowship and Vannevar Bush Faculty Fellowship ONR-N0014-21-1-2887. E.M. and C.S. were partially supported by NSF award DMS-1737944. E.M. was partially supported by Simons Investigator award (622132) and by ARO MURI W911NF1910217. J.G. was partially supported by NSF award CCF-2154100. Part of this work was completed while S.D. and C.S. were at the MIT Mathematics Department, and also while S.D. was at The Simons Institute for the Theory of Computing. J.G. thanks Nirmit Joshi for helpful discussions, and Charlie Guan for catching some errors in the proof of Lemma \ref{lem:high-disc}.
}

\begin{document}
\maketitle
\begin{abstract}
Spectral algorithms are some of the main tools in optimization and inference problems on graphs. Typically, the graph is encoded as a matrix and eigenvectors and eigenvalues of the matrix are then used to solve the given graph problem. Spectral algorithms have been successfully used for graph partitioning, hidden clique recovery and graph coloring.  In this paper, we study the power of spectral algorithms using {\em two matrices} in a graph partitioning problem. We use two different matrices resulting from two different encodings of the same graph and then combine the spectral information coming from these two matrices. 
    
We analyze a two-matrix spectral algorithm for the problem of identifying latent community structure in large random graphs. In particular, we consider the problem of recovering community assignments \emph{exactly} in the censored stochastic block model, where each edge status is revealed independently with some probability. We show that spectral algorithms based on two matrices are optimal and succeed in recovering communities up to the information theoretic threshold. Further, we show that for most choices of the parameters, any spectral algorithm based on one matrix is suboptimal. The latter observation is in contrast to our prior works (2022a, 2022b) which showed that for the {\em symmetric} Stochastic Block Model and the Planted Dense Subgraph problem, a spectral algorithm based on one matrix achieves the information theoretic threshold. We additionally provide more general geometric conditions for the (sub)-optimality of spectral algorithms.
\end{abstract}

\section{Introduction}
Spectral algorithms are some of the main tools in graph algorithms and combinatorial optimization. 
Some famous and classical examples include spectral algorithms for the hidden clique problem~\cite{AKS98}, graph bisection~\cite{Bop87}, and graph coloring~\cite{AK97,BS95}.
These algorithms encode the graph into a matrix by recording the 
status of each present/absent edge of the graph as an entry of the matrix. The most natural encoding is the adjacency matrix representation, where edges are encoded by the value $1$ and non-edges are encoded by the value $0$. Given the encoding matrix, a small number of eigenvectors for this matrix are used to solve the given graph problem.
\begin{quote}
    Our interest in this work lies in graph problems for which using multiple matrix representations gives an advantage over using a single matrix.
\end{quote}

In particular, we are interested in the power of spectral algorithms in such a scenario in the context of finding clusters in a planted partition model called the {\em Censored Stochastic Block Model (CSBM)}. 
In this model, there are two clusters of approximate sizes $n\rho$ and $n(1-\rho)$, and the edges inside each of the clusters appear independently with probabilities $p_1,p_2$ respectively, while edges between the two clusters appear with probability $q$. Moreover, each edge status is revealed with probability $t \log n/n$ for some fixed $t>0$. Thus the statuses of most edges are unknown. 
The censored model was introduced to model the fact that in many social networks, not all of the connections between individual nodes are known.

Given an instance of a censored graph with no vertex labels, the problem is to recover the partitions \emph{exactly} with high probability. This is often referred to as the \emph{exact recovery problem}.
We note that some applications of spectral algorithms to the exact recovery problem use an additional combinatorial clean-up stage (see e.g.~\cite{CO16,Vu18,YP14}), but we follow \cite{Abbe2020,Dhara2022a,Dhara2022b} in studying spectral algorithms that do not look at the graph after the top eigenvectors have been found. This is partially motivated by the fact that most real applications of spectral algorithms do not include a combinatorial clean-up stage.

The classical case in the literature considers exact recovery in the Stochastic Block Model where there is no censoring and $p_1,p_2,q = \Theta(\log n/n)$. 
In order to achieve exact recovery up to the information theoretic boundary, prior works used some trimming and post-processing steps together with the spectral algorithm~\cite{CO16,Vu18,YP14}. However, the question of whether a direct spectral algorithm based on the top two eigenvectors of the of the adjacency matrix would be optimal remained open until the recent resolution by Abbe, Fan, Wang, and Zhong~\cite{Abbe2020} for $p_1 = p_2$. 
In the \emph{censored} SBM, there are three possible observations (present, absent, or censored), so spectral recovery using a binary-valued adjacency matrix is suboptimal. Instead, one can use a ternary-valued encoding matrix.
It was recently shown in \cite{Dhara2022a,Dhara2022b} that, for some special cases of the planted partition model such as the planted dense subgraph problem $(p_2 = q)$ and the symmetric stochastic block model ($p_1= p_2, \rho = 1/2$), a spectral algorithm based on the top two eigenvectors of a signed adjacency matrix  is optimal. This raises the question:
\begin{quote}
    \emph{Are spectral algorithms based on the top eigenvectors of a signed adjacency matrix optimal for all censored stochastic block models?}
\end{quote}
The main contributions of this article are as follows: 
\begin{enumerate}
    \item In contrast with the success stories in \cite{Dhara2022a,Dhara2022b}, whenever $p_1,p_2,q$ are distinct, a spectral algorithm based on the top two eigenvectors of a signed adjacency matrix is always suboptimal (Theorem~\ref{thm:one-matrix-two-vec}~Part~\ref{thm:one-matrix-two-vec-2}).
    \item We propose spectral algorithms with two encoding matrices, where we take an appropriate linear combination of the corresponding top eigenvectors. We show that these algorithms are always optimal (Theorem~\ref{thm:two-matrix-two-vec}). The optimality of spectral algorithms with two matrices is also shown in the more general setting with $k \geq 2$ communities (Theorem~\ref{thm:two-matrix-k-communities}).
\end{enumerate}
Thus, these results exhibit a strict separation between spectral algorithm classes with one versus multiple encoding matrices, and this separation can be realized for even elementary planted partition models. To our knowledge, this  general phenomenon was not observed in the substantial prior literature for recovery problems in the planted partition problems.

\subsection{Model and Objective.}
We start by defining the Censored Stochastic Block Model. 
\begin{definition}[Censored Stochastic Block Model (CSBM)] \label{defn:CSBM}
Let $\rho\in (0,1)^k$ be such that $\sum_{i=1}^k \rho_i = 1$ and let $P\in (0,1)^{k\times k}$ be a symmetric matrix. 
Suppose we have $n$ vertices and each vertex $v\in [n]$ is assigned a community assignment $\true(v) \in [k]$ according to the distribution $\rho$ independently, i.e., $\PR(\true(v) = i) = \rho_i$ for $i\in [k]$. 
\begin{itemize}
    \item[$\triangleright$] For $u,v\in [n]$ and $u\neq v$, the edge $\{u,v\}$ exists independently with probability $P_{\true(u)\true(v)}$. Self-loops do not occur.
    \item[$\triangleright$] For every pair of vertices $\{u,v\}$, its connectivity status is {\em revealed} independently with probability $\frac{t\log n}{n}$, and is {\em censored} otherwise for some fixed $t>0$. 
\end{itemize}
The output is a random graph with edge statuses given by $\{\texttt{present},\texttt{absent}, \texttt{censored}\}$. The distribution of this random graph is called the Censored Stochastic Block Model.
We write $G \sim \CSBM_n^k(\rho, P, t)$ to denote a graph generated from the above model, with vertex labels removed (i.e., $\true$ is unknown). 
\end{definition}

\begin{definition}[Exact recovery] \label{defn:exact-recover}
Consider the $n\times k$ membership matrix $S_0$, 
where $(S_0)_{u i} = \mathbbm{1}\{\true(u) = i\}$, i.e., the $u$-th row indicates the community membership of $u$. 
Given an estimator $\hsig$, construct $\hat{S}$ similarly as $\hat{S}_{u i} = \mathbbm{1}\{\hsig(u) = i\}$. We say that an estimator achieves \emph{exact recovery} if there exists a $k\times k$ permutation matrix $J$ such that $\hat{S}J=S_0$. 
\end{definition}

\subsection{Information Theoretic Boundary.}
We start by discussing the information theoretic threshold.
The result will be stated in terms of a Chernoff--Hellinger divergence, introduced by Abbe and Sandon~\cite{Abbe2015}.  
\begin{definition}[Chernoff--Hellinger divergence]\label{def:CH-threshold}
Given two vectors $\mu, \nu \in (\R_+\setminus\{0\})^{l}$,
define
\[\CH_\xi(\mu, \nu) = \sum_{i \in[l]}  \big[\xi \mu_i  + (1-\xi) \nu_i - \mu_i^\xi \nu_i^{1-\xi}\big]\]
for $\xi \in [0,1]$. The Chernoff--Hellinger divergence of $\mu$ and $\nu$ is defined as
\begin{eq} \label{eq:CH-defn}
   \Delta_+(\mu, \nu) = \max_{\xi \in [0,1]} \CH_\xi(\mu, \nu). 
\end{eq}
\end{definition}
Define 
\begin{align}
&t_c := \Big(\min_{i\neq j} \Delta_+(\theta_i, \theta_j)\Big)^{-1} \nonumber\\
\text{where }&\theta_i = (\rho_rP_{ri}, \rho_r(1 - P_{ri}))_{r\in [k]} \in \R^{2k}. \label{eq:t_c}
\end{align} 
\begin{theorem}[Information theoretic threshold]\label{theorem:impossibility} Let $G \sim \CSBM_n^k(\rho, P, t)$. 
If $t<t_c$, then for any estimator $\hsig$, 
\begin{align*}
    \lim_{n\to\infty}\PR(\hsig \text{ achieves exact recovery}) = 0.
\end{align*}
\end{theorem}
\subsection{Spectral Algorithms.} 
For comparing the performance of spectral algorithms with one matrix versus spectral algorithms  with more than one matrix, we first specialize to the case of two communities. 

To define spectral algorithms formally, we first define the threshold procedures we allow to apply on vectors. These are the procedures that will be applied to the leading eigenvectors of the encoding matrices. 

\begin{breakablealgorithm}
\caption{\classify}\label{alg:classify}
\begin{algorithmic}[1]
\Require{Censored graph $G$ on $n$ vertices, vectors $(u_i)_{i=1}^m \subset \R^n$,
and scalars, $a_1,\dots,a_m,T \in \R$.
}
\Ensure{Community classification.}
\vspace{0.2cm}
   \State Compute possible score vectors 
   \[U  = \left\{\sum_{i = 1}^m s_i a_i u_i \text{ for all } s_1,...,s_m\in \{\pm 1\}\right\}.\]
    \State Compute possible assignments $\hat{S}(U) = \{ \hat{\sigma} = \sgn(u - T) : u \in U  \}$  and output a community assignment  
    $1+ (1+\hat{\sigma})/2$ that maximizes the posterior probability $\mathbb{P}( \hat{\sigma} \mid G)$ over $\hat{\sigma} \in \hat{S}(U)$. 
\end{algorithmic}
\end{breakablealgorithm}

Since eigenvectors are determined up to a sign flip, Step 2 above is required in order to resolve this sign ambiguity. This will be explained in more detail in Remark~\ref{rem:sign-amb}.
\begin{definition}[Signed adjacency matrix]\label{def:signed-adj}
Given $y>0$ and a graph $G$ with edge statuses $\{\texttt{present}$, $\texttt{absent}$, $\texttt{censored}\}$, define the signed adjacency matrix $A(G,y)$ as the $n\times n$ matrix with
\begin{align*}
A_{ij} &= \begin{cases}
1 &\text{if } \{i,j\} \text{ is present}\\
-y &\text{if } \{i,j\} \text{ is absent}\\
0 &\text{if } \{i,j\} \text{ is censored}.
\end{cases}
\end{align*}
\end{definition}

Let us define the class of algorithms \specone~that use a single encoding matrix.
\begin{definition}[\specone]
An algorithm $\mathcal{A}(G,y,a_1,a_2, T)$ in the \specone~class takes a censored graph $G$ as input, an encoding parameter $y\in \R_+$, and scalars $a_1,a_2,T \in \R$.
The algorithm then computes the top two eigenvectors $u_1,u_2$ of $A = A(G,y)$, and gives the output of $\classify((u_i)_{i=1}^2, (a_i)_{i=1}^2, T)$.
We denote the output of algorithm $\cA$ in this class as $\hsig_{\sss \cA}$.
\end{definition}

For the two community case, we will always consider the parameters:
\begin{equation} \label{eq:two-com-params}
P = \begin{pmatrix} p_1 & q \\q & p_2 \end{pmatrix}, \bar{\rho} = (\rho,1-\rho), \text{ and }  \rho,p_1,p_2,q \in (0,1).
\end{equation}

\begin{theorem}[Failure of \specone~in most cases]\label{thm:one-matrix-two-vec}
Let $G \sim \CSBM_n^2(\bar{\rho}, P, t)$ with $\bar{\rho}, P$ given by \eqref{eq:two-com-params}. 
\begin{enumerate}[(1)]
    \item \label{thm:one-matrix-two-vec-1} 
    Suppose that $p_1,p_2,q$ are not distinct. If $p_1 = p_2 = p$, 
    then assume $p + q \neq 1$. \footnote{The case $p_1 = p_2 = p$, $\rho = \frac{1}{2}$ is covered in \cite{Dhara2022a} without the assumption $p+q \neq 1$. In this case, spectral algorithms succeed for $t > t_c$.} There exist explicitly computable constants $y\in \R_+$ and $\gamma_1,\gamma_2\in \R$ such that the algorithm  $\cA= \mathcal{A}(G,y,\gamma_1,\gamma_2, 0)$ from the class $\specone$ satisfies 
    $$\lim_{n\to\infty}\PR(\hsig_{\sss \cA} \text{ achieves exact recovery}) = 1,$$ 
    for any $t > t_c$.
    In particular, Algorithm \ref{alg:one-matrix} produces such an estimator.
    \item \label{thm:one-matrix-two-vec-2} Suppose that $p_1,p_2,q$ are distinct. There exists $\delta_0>0$ such that, if $t < t_c + \delta_0$, then for any $\cA\in \specone$, 
    $$\lim_{n\to\infty}\PR(\hsig_{\sss \cA} \text{ achieves exact recovery}) =  0.$$
\end{enumerate}
\end{theorem}
For the case $p_1= p_2$,  Theorem~\ref{thm:one-matrix-two-vec}~Part~\ref{thm:one-matrix-two-vec-1} generalizes the result of \cite[Theorem 2.2]{Dhara2022a} to the case $\rho\neq 1/2$. Part~\ref{thm:one-matrix-two-vec-2} of the result is in sharp contrast with the results in \cite{Dhara2022a,Dhara2022b}; together, these results essentially say that the censored planted dense subgraph problem ($p_2 = q$) and the symmetric censored stochastic block models ($p_1 = p_2$) are remarkably the only cases where an algorithm from \specone~is successful\footnote{For the edge-case $p_1 = p_2 = p$ and $p + q = 1$, the rank of $\E[A] $ is 1 for the value of $y$ that we would want to use. This is why it is ruled out in Theorem~\ref{thm:one-matrix-two-vec}~Part~\ref{thm:one-matrix-two-vec-1}.
}. 
The possible limitation of \specone~was shown in \cite[Theorem 2.6]{Dhara2022b} for the special case of $q = 1/2$, $p_1 = 1-p_2$ and $\rho = 1/2$.
\begin{remark} \normalfont It is worthwhile to note that the choice of encoding parameters $\{1,-y,0\}$ is completely general and one does not get a more powerful class of algorithms by allowing an arbitrary ternary encoding. 
In fact, as our proof shows, if $p_1,p_2,q$ are distinct, then even if one allows arbitrary encodings, the \specone~algorithms still fail sufficiently near the threshold (see Remark~\ref{rem:arbitrary-encoding}). 
\end{remark}

Next, we will show that spectral algorithms with two matrices are always optimal for the recovery of two communities. 
Let us define the class of algorithms \spectwo~that uses two encoding matrices instead of one. 

\begin{definition}[\spectwo]
An algorithm $\mathcal{A}(G,y_1,y_2,(a_i)_{i=1}^4, T)$ in the \spectwo~class takes as input a censored graph $G$, two encoding parameter $y_1,y_2\in \R_+$ with $y_1\neq y_2$ and $(a_i)_{i=1}^4 \subset \R$,  $T \in \R$.
The algorithm considers two signed adjacency matrices $A_1 = A(G,y_1)$ and $A_2 = A(G,y_2)$, and computes their top two eigenvectors $u_1^r,u_2^r$, for $r = 1,2$. 
Then the algorithm outputs $\classify((u_i^r)_{i, r = 1,2}, (a_i)_{i=1}^{4}, T)$.
As before, we denote the output of algorithm $\cA$ from this class as $\hsig_{\sss \cA}$.
\end{definition}

\begin{theorem}[\spectwo~always succeeds in recovering two communities]\label{thm:two-matrix-two-vec}
Let $G \sim \CSBM_n^2(\bar{\rho}, P, t)$ with $\bar{\rho}, P$ given by~\eqref{eq:two-com-params}.
There exists a set $\mathcal{Y} \subset \R_+$ with $|\mathcal{Y}|\leq 3$ such that for any $y_1\neq y_2$ and $y_1,y_2 \notin \mathcal{Y}$, there exist explicit $(a_i)_{i = 1}^4 \subset \R^4$ such that the algorithm  $\mathcal{A}(G,y_1,y_2,(a_i)_{i=1}^4, 0)$ from the class $\spectwo$ satisfies
 $$\lim_{n\to\infty}\PR(\hsig_{\sss \cA} \text{ achieves exact recovery}) = 1,$$
 for any $t > t_c$.
In particular, Algorithm \ref{alg:two-matrix-a} produces such an estimator.
\end{theorem}
Theorem~\ref{thm:two-matrix-two-vec} not only shows that \spectwo~algorithms are always successful, but also shows that the choice of the encoding parameters $y_1,y_2$ does not matter too much as long as $y_1\neq y_2$ and they both lie outside a finite exception set. For example, we can choose $y_1,y_2\sim \text{Uniform}[0,1]$ independently.
Avoiding the finite exception set helps us ensure that $A_1$ and $A_2$ both have two eigenvectors with large, distinct eigenvalues. In contrast, 
the choice of the encoding is quite important for \specone~algorithms in Theorem~\ref{thm:one-matrix-two-vec}~\ref{thm:one-matrix-two-vec-1}. In fact, for $p_1 = p_2 = p$ or $p_1 = p$ and $p_2 = q$, the only choice of $y$ that yields an optimal algorithm is $\log (\frac{1-q}{1-p}) / \log \frac{p}{q}$. Thus, \spectwo~algorithms leads to a much broader and flexible class of algorithms as compared to \specone.

Finally, we show that \spectwo~succeeds for the recovery of $k \geq 3$ communities, as long as the parameters $P, \rho$ satisfy certain conditions. To this end, let us define \spectwo~for general $k$.

\begin{breakablealgorithm}
\caption{\textsc{Classify-Multiple}}\label{alg:classify-multiple}
\begin{algorithmic}[1]
\Require{Censored graph $G$ on $n$ vertices, vectors $(u_i)_{i=1}^m \subset \R^n$, and weight vectors $(a_i)_{i=1}^k \subset \R^m$.
}
\Ensure{Community classification.}
\vspace{0.2cm}
\State Let $U$ be the $n \times m$ matrix whose $i$-th column is $u_i$.
    \State 
For $s \in \{\pm 1\}^{m}$, let $D^{\sss (s)}:= \diag(s)$. Compute the set of possible assignments $\hat{S}$ consisting of $\hat{\sigma}(\cdot; s)$ with $s \in \{\pm 1\}^m$ such that
\begin{align*}
    \hat{\sigma}(v; s) = \argmax_{i \in [k]} \Big\{\big(U  D^{\sss (s)}  a_i\big)_v  \Big\} \text{ for each }v\in [n].
\end{align*}
\State Output $\hat{\sigma}(\cdot; s)$  that maximizes the posterior probability  $\mathbb{P}( \hat{\sigma} \mid G)$ over $\hat{\sigma} \in \hat{S}$.
\end{algorithmic}
\end{breakablealgorithm}
We will use this algorithm with $m=2k$, and the top $k$ eigenvectors from each of two signed adjacency matrices.

\begin{definition}[\spectwo~for $k\geq 3$ communities]
An algorithm $\mathcal{A}(G,y_1,y_2,(a_i)_{i=1}^{k}, 
T)$ in this class takes as input a censored graph $G$, two encoding parameters $y_1,y_2\in \R_+$ with $y_1\neq y_2$ and $(a_i)_{i=1}^{k} \subset \R^{2k}$.
The algorithm considers two signed adjacency matrices $A_1 = A(G,y_1)$ and $A_2 = A(G,y_2)$, and computes their top $k$ eigenvectors $(u_i^1)_{i\in [k]},(u_i^2)_{i\in [k]}$.  
Then the algorithm outputs $\textsc{Classify-Multiple}((u_i^r)_{i \in [k], r = 1,2}, (a_i)_{i=1}^{k})$.
As before, we denote the output of algorithm $\cA$ from this class as~$\hsig_{\sss \cA}$.
\end{definition}

\begin{theorem}[Success of \spectwo~ for $k\geq3$ communities]\label{thm:two-matrix-k-communities}
Let $G \sim \CSBM_n^k(\rho, P, t)$ where $\rho \in (0,1)^k$ is such that $\sum_i\rho_i = 1$, and $P \in (0,1)^{k \times k}$ is a symmetric matrix. 
Further, suppose that $P\cdot \diag(\rho)$ has exactly $k$ distinct non-zero eigenvalues.
Then there exists a finite set $\mathcal{Y} \subset \R_+$ such that for any $y_1\neq y_2$ and $y_1,y_2 \notin \mathcal{Y}$, the following holds: 
there exist explicit vectors $(a_i)_{i = 1}^{k} \subset \R^{2k}$ such that the algorithm  $\mathcal{A}(G,y_1,y_2,(a_i)_{i = 1}^{k})$ from the class $\spectwo$ satisfies
$$\lim_{n\to\infty}\PR(\hsig_{\sss \cA} \text{ achieves exact recovery}) = 1,$$
for any $t > t_c$.
In particular, Algorithm \ref{alg:two-matrix-b} produces such an estimator.
\end{theorem}
\begin{remark} \label{rem:sign-amb}
    The fact that the encoding parameters $y_1,y_2$ lie outside a finite set in Theorems~\ref{thm:two-matrix-two-vec}~and~\ref{thm:two-matrix-k-communities} is required to ensure that $\E[A(G,y_1)]$, $\E[A(G,y_2)]$ have $k$ distinct and non-zero eigenvalues. 
    The requirement of having $k$ non-zero eigenvalues is intuitive as we seek to recover an underlying rank~$k$ structure. Moreover, the eigenvectors of  $A(G,y)$  can only be approximated up to an unknown orthogonal transformation.
    This causes an ambiguity for defining the final estimator. 
    When the eigenvalues are distinct, this ambiguity can be resolved by going over all possible sign flips $s$ and choosing the best among them, as in Algorithm~\ref{alg:classify}~Step~2, or Algorithm~\ref{alg:classify-multiple}~Step~2. 
\end{remark}
\begin{remark}The condition in Theorem~\ref{thm:two-matrix-k-communities} that $P\cdot \diag(\rho)$ has distinct and non-zero values can be relaxed. In fact, if $P^{\sss (y)}$ is the matrix such that $P^{\sss (y)}_{ij}:= \rho_j(P_{ij} - y(1-P_{ij})) $, then by Lemma~\ref{edgeEigenvalue-general-k}, the same conclusions as Theorem~\ref{thm:two-matrix-k-communities} hold as long as there exists a $y$ such that $P^{\sss (y)}$ has $k$ distinct and non-zero eigenvalues. In fact, we can simply choose $y \sim \text{Uniform}((0,1))$. 
\end{remark}

\subsection{Proof Ideas}\label{sec:proof-ideas} 
We now give a brief outline of the proofs.
For a vertex $v$, we call 
$d(v) = (d_{+j},d_{-j})_{j\in [k]}\in \Z_+^{2k}$ the \emph{degree profile} of the vertex, where $d_{+j} = d_{+j}(v),d_{-j}=d_{-j}(v)$
respectively denote the number of present and absent edges from $v$ to community $j$ for $j \in [k]$.
Let us re-scale $\bar{d}(v) = d(v)/t \log n$.
The proof consists mainly of two steps:\\

\noindent \textbf{Step 1: Characterization of spectral algorithms using degree profiles.} 
Given any signed adjacency matrix $A = A(G,y)$, the starting point of our analysis is to find a good $\ell_\infty$-approximation for the eigenvectors. 
Using a recent general framework by Abbe,  Fan,  Wang and Zhong~\cite{Abbe2020}, we can show that 
the top $k$ eigenvalues $(u_i)_{i\in [k]}$ of $A$  satisfy (see Corollary~\ref{corollary:entrywise}): 
\begin{align*}
    \min_{s\in \{\pm 1\}}\bigg\|su_i -\frac{Au_i^\star}{\lambda_i^\star}\bigg\|_\infty = o\bigg(\frac{1}{\sqrt{n}} \bigg), \quad \text{ for }i \in [k], 
\end{align*}
with probability $1-o(1)$, where $(\lambda_i^\star,u_i^\star)$ is the $i$-th largest eigenvalue/eigenvector pair of $\E[A]$. Note that $\E[A]$ is a rank-two matrix with $u_i^\star$'s taking the same constant value corresponding to all vertices in the same community. 
The low rank of $\E[A]$ allows us to express $Au_i^\star$ as a linear combination of the degree profiles and thus drastically reduce the dimension of the problem. 
Using this representation, any linear combination of the $u_i's$ is also an expressible linear combination of degree profiles. 
Hence, we show that spectral algorithms essentially are asymptotically equivalent to classifying vertices depending on whether $\langle w_{\sss \mathrm{Spec}},\bar{d}(v)\rangle > (T+o(1))$ or $\langle w_{\sss \mathrm{Spec}},\bar{d}(v)\rangle<(T-o(1))$ for some $w_{\sss \mathrm{Spec}}\in \R^{2k}, T\in \R$. 
\\

\noindent \textbf{Step 2: Geometry of degree profiles.} At this point, the problem reduces to understanding whether, for a given vector $w$, a hyperplane orthogonal to $w$ can separate re-scaled degree profiles. 
To this end, for each community $i$, we define a measure of \emph{dissonance} $\eta_i$ for rescaled degree profiles, and define the $\delta$-\emph{dissonance range} as $\DR_{\delta} (i):=\{\bar{d}: \eta_i(\bar{d})\leq \delta\}$. 
We show that the $\DR_{\delta} (i)$'s are closed and convex sets. Moreover, (1) if $ 1/t < \delta$, then all the re-scaled degree profiles from community $i$ lie in $\DR_{\delta} (i)$ and (2) if $  \delta < 1/t$, then the re-scaled degree profiles from community $i$ are asymptotically \emph{dense} in $\DR_{\delta} (i)$ (see Lemma~\ref{DRmeaning}). In a sense, one can think of $\DR_{1/t} (i)$ as the cloud of re-scaled degree profiles arising from community~$i$. 
\begin{figure}
    \centering
    \begin{tikzpicture}
    \node[inner sep=0pt] (clouds) at (0,0)
    {\includegraphics[width=.45\textwidth]{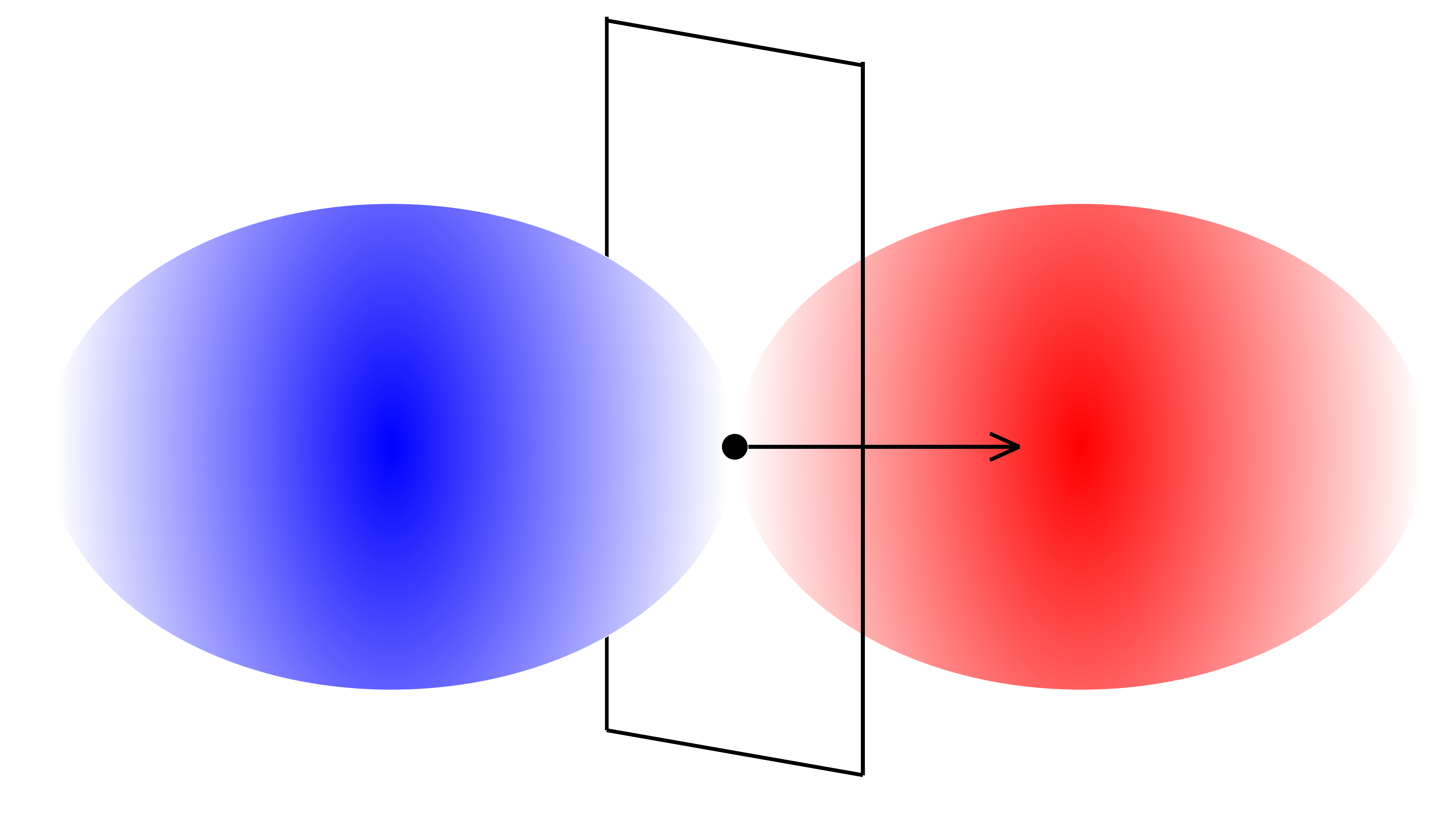}};
    \node[inner sep=0pt] (x) at (.1,-.5) {$x^\star$};
    \node[inner sep=0pt] (w) at (1.5,.1) {$w^\star$};
    \node[inner sep=0pt] (DR1) at (-1.8,-1.9) {{$\DR_\delta(1)$}};
    \node[inner sep=0pt] (DR2) at (1.9,-1.9) {{$\DR_\delta(2)$}};
    \node[inner sep=0pt] (opt) at (1.6,2.2) {{Optimally separating hyperplane}};
\end{tikzpicture}
    \caption{Visualizing dissonance ranges of two communities near $t_c$.}
    \label{fig:visual}
\end{figure}
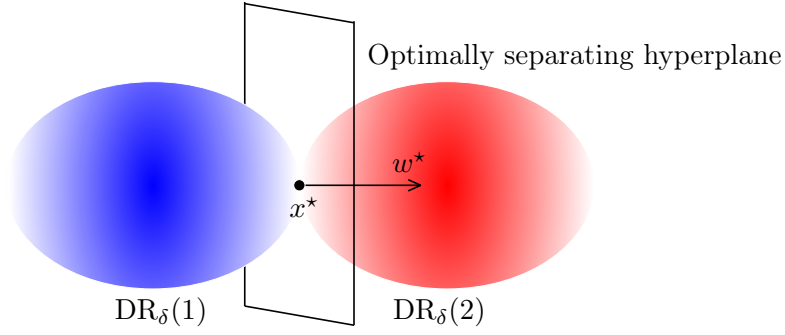

Next, consider the ``hardest'' scenario when $t = t_c$. In that case, we show that the clouds $\DR_{1/t_c} (i)$ and $\DR_{1/t_c} (j)$ corresponding to communities $i$ and $j$ intersect only at a single point~$x^\star$ (see Lemma~\ref{DRoverlapThreshold}), and as $t$ increases away from $t_c$, the two clouds gradually separate. 
Due to convexity, $\DR_{1/t_c} (i)$ and $\DR_{1/t_c} (j)$ lie on two opposite sides of the tangent hyperplane at~$x^\star$. Let $w^\star$ be such that this tangent hyperplane is given by $H ^\star= \{x:\langle w^\star,x-x^\star\rangle = 0\}$. 
Then $H ^\star$ is the only hyperplane that separates the clouds of degree profiles near $t_c$; see Figure~\ref{fig:visual}. 
Thus, as long as we are trying to separate clouds of degree profiles using this $H ^\star$, we will succeed for any $t>t_c$. However, if we try to separate the clouds with a different hyperplane $\{x:\langle w,x-x^\star\rangle = 0\}$ for some $w\notin \spn(w^\star)$, then we will fail sufficiently close to $t_c$.

Combining this with the asymptotic characterization of spectral algorithms, it thus remains to be seen whether we can choose the parameters of the spectral algorithm in such a way that $w_{\sss \mathrm{Spec}} \in \spn(w^\star)$. 
For \specone~algorithms in the two community case, we show that $w_{\sss \mathrm{Spec}}$ takes values in a restricted set $\{w\in \R^4: \frac{w_1}{w_2} = \frac{w_3}{w_4} = y\}$, no matter the choice of the parameters. 
For \specone~algorithms, generally $w_{\sss \mathrm{Spec}} \notin \spn(w^\star)$ except for the specific cases in Theorem~\ref{thm:one-matrix-two-vec}~\ref{thm:one-matrix-two-vec-1}. 
However, for \spectwo~algorithms, there always exists a way to choose the linear combinations in such a way that $w_{\sss \mathrm{Spec}} \in \spn(w^\star)$, which ensures  their optimality. \\

\noindent \textbf{Information Theoretic Threshold.}
There is an alternate way of characterizing the information theoretic boundary by observing that even the ``best'' estimator will separate communities using the hyperplane $H^\star$ above.
Consider the problem of classifying a single vertex $v$ given
$G$ and $(\true(u))_{u\in [n]\setminus \{v\}}$. The MAP estimator for the community assignment of $v$ is called the genie-based estimator. 
This is an optimal estimator (even though it is not computable given $G$). 
Now, a direct computation shows that the genie-based estimator classifies a vertex in one of the two communities based on whether $\langle w^{\star},\bar{d}(v)\rangle > 0$ or  $\langle w^{\star},\bar{d}(v)\rangle < 0$, with the same $w^\star$ as above (see \cite[Proposition~6.1]{Dhara2022a}). Thus, in a sense, separating degree profiles based on hyperplanes orthogonal to $w^\star$ is the optimal decision rule. 
When $t < t_c$, the degree profile clouds of the two communities overlap significantly, and therefore even the optimal estimator misclassifies a growing number of vertices. 
This gives rise to the information theoretic impossibility region for exact recovery when $t < t_c$.

\subsection{Discussion} 
Theorems~\ref{thm:two-matrix-two-vec}~and~\ref{thm:two-matrix-k-communities} prove optimality of spectral algorithms using two matrices. The use of two matrices hinges on the fact that there are three types of edge information:
present, absent, and censored, and the information about a vertex's community coming from present and absent edges are of the same order.
We believe that our results generalize in a straightforward manner to the scenario of labeled edges, where the possible edge statuses $\{\text{present,absent}\}$ are replaced by $L$ different types. Indeed, this is the setting considered by \cite{Heimlicher2012,Ryu2016,YP16}. In particular, \cite{YP16} determined the information-theoretic threshold for exact recovery and proposed an efficient, iterative spectral method. We believe that optimal (vanilla) spectral algorithms in the general $L$-labeled edge scenario must use $L$ different encoding matrices.

We also believe that the framework of this paper can be extended beyond graphs to other important machine learning problems with censoring on top of an underlying low-rank structure. This may include non-square matrices (e.g. items vs features matrix in recommender systems). We leave these as interesting future research questions.

\subsection{Notation.}
Let $[n] = \{1, ,2, \dots, n\}$. We often use the Bachmann--Landau asymptotic notation $o(1), O(1)$ etc. For two sequences $(a_n)_{n\geq 1}$ and $(b_n)_{n\geq 1}$, we write $a_n \asymp b_n$ as a shorthand for $\lim_{n\to \infty}\frac{a_n}{b_n} =1$. 
Given a sequence of probability measures $(\PR_n)_{n\geq 1}$, a sequence of events $(\cE_n)_{n\geq 1}$ is said to hold \emph{with high probability} if $\lim_{n\to\infty} \PR_n(\cE_n) = 1$

For a vector $x \in \mathbb{R}^d$, we define $\Vert x \Vert_2 = (\sum_{i} x_i^2)^{1/2}$ and $\Vert x \Vert_{\infty} = \max_i |x_i|$. 
For $x \in \R^d$ and $r>0$, we denote the open $\ell_2$-ball of radius $r$ around $x$ by $B_2(x,r)$.
Similarly, for $X \subset \R^d$ and $r>0$, we denote the open $\ell_2$-ball of radius $r$ around $X$ by $B_2(X,r)$. For a collection of vectors $(x_i)_{i} \subset \R^d$, we denote their linear span by $\spn((x_i)_{i})$. Also, given a subspace $\mathcal{Z} \subset \R^d$, the projection of $x$ onto $\mathcal{Z}$ will be denoted by $\proj_{\sss \mathcal{Z}} (x)$.

For a matrix $M \in \mathbb{R}^{n \times d}$, we use $M_{i \cdot}$ to refer to its $i$-th row, represented as a row vector. 
Given a matrix $M$, $\Vert M \Vert_2 = \max_{\Vert x \Vert_2 = 1} \Vert M x \Vert_2$ is the spectral norm, 
 $\Vert M \Vert_{2 \to \infty} = \max_i \Vert M_{i \cdot}\Vert_2$ is the matrix $2 \to \infty$ norm, 
 and $\|M\|_{F} = (\sum_{i,j}M_{ij}^2)^{1/2}$ is the Frobenius norm. Whenever we apply a real-value function to a vector, it should be interpreted as a coordinatewise operation.
 
 Throughout, we condition on the event that the random community assignments given by $\sigma_0$ are close to their expected sizes. Specifically, note that, since $n_j:= \{v: \true(v) = j\}$ are marginally distributed as $\text{Bin}\left(n, \rho_j\right)$, and therefore, for all $\ve \in (0,1)$, 
\begin{eq} \label{eq:concentration}
\left|n_j - n\rho_j \right| \leq \ve n
\end{eq}
with probability at least $1 - 2\exp(-\ve^2 n/2 )$ by applying the McDiarmid inequality. Throughout, the notation $\PR(\cdot), \E[\cdot]$ conditions 
on a fixed value of $\sigma_0$ satisfying \eqref{eq:concentration} with $\varepsilon = n^{-1/3}$.

\subsection{Organization}
We start analyzing the geometric properties of the degree profile clouds in Section~\ref{sec:geometric}, which lies in the heart of all the proofs. 
Subsequently, in Section~\ref{sec:achieve-impossible}, we prove the impossibility result and also prove that the Maximum a Posteriori (MAP) Estimator always succeeds up to the information theoretic threshold. The entrywise bounds for the top eigenvectors are provided in Section~\ref{sec:entrywise}. Finally, we complete the proofs of Theorems~\ref{thm:one-matrix-two-vec},~\ref{thm:two-matrix-two-vec} in Section~\ref{sec:spectral}.

\section{Geometry of  Degree Profiles}\label{sec:geometric}
In this section, we develop the technical tools for Step 2 in Section~\ref{sec:proof-ideas}.
We will develop these tools for general $k$-community CSBMs. Throughout, we fix  $\rho\in (0,1)^k$ such that $\sum_{i=1}^k \rho_i = 1$ and let $P\in (0,1)^{k\times k}$ be a symmetric matrix.
Let us define degree profiles, which will be the main object of analysis in this section.
\begin{definition}[Degree profile] Suppose that $G \sim \CSBM_n^k(P, \rho, t)$. 
For a vertex $v$, we define $d(v) = (d_{+r},d_{-r})_{r\in [k]}\in \Z_+^{2k}$ to be the \emph{degree profile} of $v$, where $d_{+r} = d_{+r}(v)$ and $d_{-r}=d_{-r}(v)$ respectively denote the number of present and absent edges from $v$ to community $r$ for $r\in [k]$. 
\end{definition}

As discussed in Section~\ref{sec:proof-ideas}, the $\ell_\infty$ approximation guarantee for the eigenvectors gives us an alternative characterization of spectral algorithms in terms of separating degree profiles of different communities using certain hyperplanes. The next proposition allows us to determine when separation using hyperplanes is impossible. Before the statement we need a couple of definitions. Let $V_i$ denote the vertices in community $i$.
\begin{definition}[Separates communities]\label{defn:separate}
We say that $w \in \mathbb{R}^{2k}$ \emph{separates communities $(i,j)$ with margin $\beta > 0$} if
\[\min_{v \in V_i} w^T d(v) \geq \beta/2 \quad \text{and} \quad  \max_{v \in V_j} w^T d(v) \leq -\beta/2.\]
or vice versa.
\end{definition}
If $w$ separates communities $(i,j)$ with margin $\beta > 0$, then computing the weighted degree profile $w^T d(v)$ for each $v \in V_i \cup V_j$ allows us to distinguish these two communities. Note that if $w$ separates communities $(i,j)$ with margin $\beta$, then $-w$ also separates communities $(i,j)$ with margin $\beta$. 
Next we define the scenario where a finite number of hyperplanes cannot separate the two communities. 

\begin{definition}[Confuses communities] \label{defn:confuse}
Let $(w_r)_{r = 1}^m \subset \mathbb{R}^{2k}$ and let $(\gamma_r)_{r=1}^m \subset \mathbb{R}$. We say that $[(w_r)_{r = 1}^m, (\gamma_r)_{r=1}^m]$ \emph{confuses communities $(i,j)$ at level $\beta$} if there exist $u \in V_i$, $v \in V_j$, and $s\in\{-1,1\}^m$ such that $s_r(w_r^T d(u)-\gamma_r)>\beta $ and $s_r(w_r^T d(v)-\gamma_r)>\beta $ for all $1\le r\le m$.
\end{definition}
In other words, there are representatives from communities $i$ and $j$, such that both of their degree profiles appear on the same sides of all the hyperplanes $\{x : w_r^T x = \gamma_r\}$. A larger value of $\beta$ means that the pair of degree profiles is farther from the hyperplanes.
Note that the notion of confusion also rules out the possibility of separation with multiple hyperplanes. We claim that there is a unique best vector for separating community $i$ and community $j$ in the following sense.

\begin{proposition}\label{prop:seperation-verification}
Let $G \sim \CSBM_n^k(\rho, P, t)$, $1\le i< j\le k$, and let $w^\star$ be the $2k$-dimensional vector such that
\begin{eq} \label{defn:genie-vec}
w^\star = \bigg(\log \frac{P_{ri}}{P_{rj}}, \log \frac{1-P_{ri}}{1-P_{rj}} \bigg)_{r\in [k]}.
\end{eq}
\begin{enumerate}[(1)]
    \item \label{prop:seperation-verification-1} If $t>1/\Delta_+(\theta_i,\theta_j)$, then there exists $\varepsilon>0$ such that $w^{\star}$ separates communities $i$ and $j$ with margin $\varepsilon\ln(n)$ with probability $1-o(1)$. 
    \item \label{prop:seperation-verification-2} Let $\mathcal{Z} \subset \R^{2k}$ be a linear subspace and  $w^{\star}\notin \mathcal{Z}$. There exists 
    $\mu>0$ such that if $t\Delta_+(\theta_i,\theta_j) < 1 + \mu$, then for every $m>0$ there exists $\varepsilon>0$ such that the following holds with probability $1-o(1)$: For every $z_1,...,z_{m}\in \mathcal{Z}$ and $\gamma_1,...,\gamma_{m}\in\mathbb{R}$, $[(z_r)_{r=1}^{m},(\gamma_r)_{r=1}^{m}]$ confuses communities $i$ and $j$ at level $\varepsilon\ln(n)$. 
\end{enumerate}

\end{proposition}

The above result yields the following corollary which is useful in designing our classification algorithm for $k \geq 3$ communities (Algorithm \ref{alg:two-matrix-b}).
\begin{corollary}\label{corollary:max}
If $t>1/\Delta_+(\theta_i,\theta_j)$, then there exists $\varepsilon>0$ such that with probability $1-o(1)$
\begin{align*}
&\big(\log(P_{ri}), \log(1-P_{ri}) \big)_{r\in [k]}\cdot d(v) >\max_{j\ne i}\big(\log(P_{rj}), \log (1-P_{rj}) \big)_{r\in [k]}\cdot d(v)+\varepsilon\log(n)  
\end{align*}
for all $i \in [k]$ and $v \in V_i$.
\end{corollary}
\begin{proof}
Proposition \ref{prop:seperation-verification} implies that with probability $1-o(1)$,
\begin{align*}
&\big(\log(P_{ri}), \log(1-P_{ri}) \big)_{r\in [k]}\cdot d(v)> \big(\log(P_{rj}), \log (1-P_{rj}) \big)_{r\in [k]}\cdot d(v)+\varepsilon\log(n)   
\end{align*}
for every $i, j \in [k], i \neq j$ and $v\in V_i$ . The claim follows.
\end{proof}

The rest of this section is devoted to the proof of Proposition~\ref{prop:seperation-verification}.
In Section~\ref{sec:dissonance-def}, we define the \emph{dissonance range} relative to a community as the set of $2k$-tuples that are sufficiently close to the average normalized degree profile for that community, and prove some of their analytic properties. In Section~\ref{sec:dissonance-CSBM}, we prove that, with high probability, the normalized degree profile of every vertex is within $o(1)$ of the dissonance range corresponding to its community. Moreover, we also show that the dissonance ranges are asymptotically dense in the sense that for every point in a dissonance range there is a vertex in the corresponding community whose normalized degree profile is within $o(1)$ of that point. 
Next, in Section~\ref{sec:seperate-hyperlane}, we prove that if the projections of two dissonance ranges onto the space spanned by a set of vectors overlap, then any set of hyperplanes perpendicular to those vectors confuses the corresponding communities. We prove this by showing that there are points in the interiors of the two dissonance ranges that are on the same sides of all such hyperplanes. In Section~\ref{sec:dis-overlap}, we show that for any two communities and the appropriate choice of $t$, their dissonance ranges intersect at a single point, the hyperplane perpendicular to $w^\star$ through that point separates the rest of the dissonance ranges, and the boundaries are also smooth in the vicinity of that point. 
Finally, in Section~\ref{sec:dis-optimal-recovery}, we prove Proposition~\ref{prop:seperation-verification} by observing that the hyperplane through the origin perpendicular to $w^\star$ separates the dissonance ranges corresponding to the underlying communities whenever $t$ is greater than the critical value, while the projections of the dissonance ranges onto any subspace of $\mathbb{R}^k$ not containing $w^\star$ will overlap for any value of $t$ sufficiently close to the critical value.

\subsection{Dissonance Range and its Properties.} \label{sec:dissonance-def}
Let us start by defining the dissonance range and obtaining some basic analytic properties.
\begin{definition}[Dissonance range]
Given $i\in[k]$ and $x \in \R_+^{2k}$, the \emph{dissonance} of $x$ relative to community $i$ is given by
\begin{align}
    \eta_i(x)=\sum_{r=1}^k &\bigg[ x_{1,r}\ln\bigg(\frac{x_{1,r}}{\e\rho_r P_{ri}} \bigg) +x_{2,r}\ln\bigg(\frac{x_{2,r}}{\e\rho_r (1-P_{ri})}\bigg)\bigg] + 1,\label{def:eta}
\end{align}
where we regard the terms in these expressions as being $0$ if the corresponding entry of $x$ is $0$. 
We also define the $\delta$-\emph{dissonance range} relative to community $i$ by $\DR_\delta(i):= \{x: \eta_i(x) \leq \delta\}$.
\end{definition}

\begin{lemma}\label{DRproperties}
Fix $i\in[k]$ and $\delta>0$. Then $\DR_\delta(i)$ is a bounded, closed and convex subset of $\R_+^{2k}$. In addition, for any $\delta'>\delta$, there exists $ \varepsilon>0$ such that 
\begin{align*}
    B_2(\DR_{\delta}(i),\varepsilon) \cap \mathbb{R}_+^{2k} \subset \DR_{\delta'}(i).
\end{align*}
\end{lemma}

\begin{proof}
We first show that $\DR_\delta(i)$ is bounded. Note that $z\ln(z) \to \infty$ if and only if $z\to\infty$. Thus, if $\DR_\delta(i)$ were unbounded, then we could find a subsequence $(x_k)_{k\geq 1}\subset \DR_\delta(i)$ such that $\eta_i(x_k) \to \infty$. However, $\eta_i(x_k) \leq \delta$ by definition of $\DR_\delta(i)$. This leads to a contradiction and hence $\DR_\delta(i)$ is bounded. 

Next, since $\eta_i$ is continuous, we have that $\DR_\delta(i)$ is closed. 
Further, $\eta_i$ is a sum of convex function and hence it is convex. Therefore, its sublevel set $\DR_\delta(i)$ is convex.

To show the last claim, note that $\eta_i$ is uniformly continuous on $[0,b]^{2k}$ for any $b>0$.
Thus, there exists $\varepsilon>0$ such that for any $x,x'\in[0,b]^{2k}$ with $\|x-x'\|_2 \leq \varepsilon$, we have  $|\eta_i(x)-\eta_i(x')|\le \delta'-\delta$. 
This proves $ B_2(\DR_{\delta}(i),\varepsilon) \cap \mathbb{R}_+^{2k} \subset \DR_{\delta'}(i)$, and completes the proof of the lemma.
\end{proof}

\subsection{Relating Dissonance Range with Degree Profiles of CSBMs} \label{sec:dissonance-CSBM}
Our next goal will be to identify which degree profiles are likely to occur in CSBMs. 
\begin{lemma} \label{DRmeaning}
Fix $0 < \delta < \delta'$.
Let $t\in (1/\delta',1/\delta)$ and $G \sim \CSBM_n^k(P, \rho, t)$. The following holds with probability $1-o(1)$:
\begin{enumerate}
    \item \label{DRmeaning-1} There exists $c > 0$ such that for every $i\in [k]$ and $d\in\mathbb{Z}_+^{2k}$ satisfying $d/(t\log(n))\in \DR_{\delta}(i)$, there are at least $n^c$ vertices in community~$i$ with degree profile $d$. 
    \item \label{DRmeaning-2} For each $i\in [k]$ and for every vertex $v\in G$ in community $i$, the degree profile of $v$ is of the form $x t\ln(n)$ for some $x\in \DR_{\delta'}(i)$.
\end{enumerate} 
\end{lemma}
In order to prove this lemma, we need the Poisson approximation result stated below. The proof of this follows from a straightforward application of Stirling's approximation and will therefore be provided in Appendix~\ref{sec:appendix-poisson}.
\begin{lemma}
\label{fact:stirling}
Let $(S_r)_{r\in [k]}$ be a partition of $[n]$ such that $|S_r| = n \rho_r (1+O(\log^{-2} n ))$ for all $r\in [k]$, where $\rho \in (0,1)^k$. 
Suppose that $\{W_v\}_{v=1}^{n}$ is i.i.d.~from a distribution taking values in $\{a,b,c\}$ and, if $v\in S_r$, $\mathbb{P}(W_v = a) = \alpha \psi_r$, $\mathbb{P}(W_v = b) = \alpha (1-\psi_r)$, and $\mathbb{P}(W_v = c) = 1- \alpha$. 
Fix $i \in [k]$. Also, let $V \subset S_i$ be such that $|V| = O(n/\log^2n) $. 

For $x \in\{ a,b\}$, let $D_{x,r}:= \#\{v\in S_r: W_v =x\}$ for $r \in [k] \setminus \{i\}$ and $D_{x,i}:= \#\{v\in S_i\cap V^c: W_v =x\}$. 
Let $D = (D_{a,r},D_{b,r})_{r\in [k]}$ and also let $d = (d_{1,r},d_{2,r})_{r\in [k]}\in \Z_+^{2k}$ be such that $\|d\|_1 = o(\log^{3/2} n)$ and $\alpha = t \log n /n$. Then 
\begin{align*}
\PR\left(D= d\right) \asymp \prod_{r = 1}^k &P\left(\rho_r\psi_r t\log n ; d_{1,r}\right) P\left(\rho_r(1-\psi_r) t\log n ; d_{2,r}\right),
\end{align*}
where $P(\lambda; m)$ is the probability that a $\mathrm{Poisson}(\lambda)$ random variable takes value $m$.
\end{lemma}

\begin{proof}[Proof of Lemma~\ref{DRmeaning}]
To prove the first part, fix $i \in [k]$ and let $d\in\mathbb{Z}_+^{2k}$ be such that $d/(t\log(n))\in \DR_{\delta}(i)$. Recall that $n_j$ is the number of vertices in community $j$ for every $j\in [k]$. 
By \eqref{eq:concentration}, $\left|n_j - n\rho_j \right| \leq n^{\frac{2}{3}}$ for all $j \in [k]$, with probability $1-o(1)$. 
In the subsequent proof, we always condition on this event, even if it is not mentioned explicitly. 

In order to prove bounds on how many vertices have a given degree profile, we will want a large set of vertices whose degree profiles are independent. As such, let $S_i$ be a random set of $2n/\log^{2}(n)$ vertices in community $i$, {chosen independently from $G,\true$}. Next, let $S'_i$ be the subset of $S_i$ consisting of all vertices $v$ such that all the connections between $v$ and $S_i$ are censored. Note that the degree profiles of the vertices in $S'_i$ are independent conditioned on the number of vertices in each community.
To lower-bound the size of $S'_i$, let $X$ be the number of revealed connections among vertices in $S$. By a counting argument, $|S'_i| \geq |S_i| - 2X$. Observe that $\mathbb{E}[X] = \binom{2n/\log^{2}n}{2} \frac{t\log n}{n} = O(n/\log^{3}(n))$. The Markov inequality then implies that $X = o(|S_i|)$ with high probability, which implies $|S'_i| \geq \frac{1}{2}|S_i| = \frac{n}{\log^2(n)}$.

Let $\cF'$ denote the sigma-algebra with respect to which $S_i$ and $(n_j)_{j\in [k]}$ are measurable, and let 
\begin{align*}
\cF := &\cF' \cap \bigg\{|n_j -  n \rho_j| \leq n^{\frac{2}{3}},
\forall j\in [k]\bigg\} \\
&\cap \bigg\{|S_i'| \geq \frac{n}{\log^2(n)}\bigg\}. \end{align*}
Fix $v \in [n]$. Since $\DR_\delta(i)$ is bounded, we have that $\|d\|_1 = O(\log n)$.
Thus, by Lemma~\ref{fact:stirling},  
\begin{eq}\label{cond-prob-simpl}
    &\PR(d(v) = d \mid \cF \cap \{v \in S_i'\}) \\
    &\asymp \e^{-t\log n} \prod_{j=1}^k\frac{\big[\rho_j P_{i,j} t\log(n)\big]^{d_{+j}}}{d_{+j}!}\frac{\big[\rho_j(1-P_{i,j}) t\log(n)\big]^{d_{-j}}}{d_{-j}!}\\
    & \asymp n^{-t} \prod_{j=1}^k\frac{(\rho_j P_{i,j} t\log(n))^{d_{+j}}}{\sqrt{2\pi d_{+j}}(d_{+j}/\e)^{d_{+j}}}\frac{(\rho_j (1-P_{i,j}) t\log(n))^{d_{-j}}}{\sqrt{2\pi d_{-j}}(d_{-j}/\e)^{d_{-j}}}\\
    &= n^{-t}  \prod_{j=1}^k\frac{1}{2\pi\sqrt{d_{+j}d_{-j}}} \bigg(\frac{\e \rho_jP_{i,j} t\log(n)}{d_{+j}}\bigg)^{d_{+j}} \bigg(\frac{\e \rho_j(1-P_{i,j}) t\log(n)}{d_{-j}}\bigg)^{d_{-j}} \\
    & = \bigg(\prod_{j=1}^k\frac{1}{2\pi\sqrt{d_{+j}d_{-j}}} \bigg) n^{-t \eta_i(d/t\log n)},
\end{eq}
where 
in the final step, we have used the definition of $\eta_i$ from \eqref{def:eta}. 
Next, since $d/(t\log(n))\in \DR_{\delta}(i)$, we have that $\eta_i(d/t\log (n)) \leq \delta$, and thus \eqref{cond-prob-simpl} yields, for all sufficiently large $n$, 
\begin{align*}
p_n&:= \PR(d(v) = d \mid \cF \cap \{v \in S_i'\}) \\
&\geq C \frac{n^{-t\delta}}{\log^{k}n} +o(n^{-1}) \\
&\geq n^{-t\delta (1+o(1))},
\end{align*}
for some $C>0$.
Next, if $d'(v)$ denotes the degree profile of vertex $v$ discarding all the present and absent edges in~$S_i$, then $d(v) = d'(v)$ for all $v\in S_i'$. 
Moreover, conditionally on $\cF'$, $\{d'(v)\}_{v\in S_i'}$ are independent. 
Thus,  conditionally on $\cF$, $|\{v \in S_i' : d(v) = d\}|$ is distributed as a 
$\text{Bin}(|S_i'|, p_n)$ random variable.
Note also that, conditionally on $\cF$, $|S_i'| p_n \geq 2n^{c}$ for some $c>0$. Thus,  using concentration of binomial random variables, we conclude that 
\[|\{v \in S_i' : d(v) = d\}| \geq \frac{1}{2}~|S_i'|~p_n \geq n^c\]
with probability at least $1 - \exp(-c'n^c)$ for some $c' > 0$. Observing that $|\{d \in \mathbb{Z}_+^{2k} : d/( \log n) \in \DR_{\delta}(i)\}| = O\left(\text{polylog}(n)\right) = o(\exp(c'n^c))$, the claim follows by a union bound.

In order to prove the second part, we again use \eqref{cond-prob-simpl}. 
By the union bound and \cite[Corollary 2.4]{JLR00}, there exists a sufficiently large constant $C>0$ such that 
\begin{eq} 
\PR( \exists v\in [n]: \|d(v)\|_1 > C\log n ) = o(n^{-1}). 
\end{eq}
Now, for any $d$ such that $d/(t\log n) \notin \DR_{\delta'}(i)$ and $\|d\|_1 \leq C\log n$, we can use \eqref{cond-prob-simpl} to show that, for all sufficiently large $n$, and fixed $v \in [n]$
\begin{align*}
    \PR(d(v) = d) &\leq (1+o(1))  \bigg(\prod_{j=1}^k\frac{1}{2\pi\sqrt{d_{+,j}d_{-,j}}} \bigg) n^{-t \eta_i(d/t\log n)} \leq n^{-t\delta'} .
\end{align*}
Now, 
\begin{align*}
    &\PR(\exists v \text{ with }\true(v) = i: d(v)/(t\log n)\notin \DR_{\delta'}(i) ) \\
    &\leq n \PR(d(v) = d \text{ for some } d/(t\log n) \notin \DR_{\delta'}(i) \text{ and } \|d\|_1\leq C\log n)+ o(1)  \\
    & \leq n(C\log n)^{2k} n^{-t\delta'} + o(1) = o(1),
\end{align*}
where in the last step we have used that $t\delta'>1$. Hence the proof is complete.
\end{proof}

\subsection{Separating Degree Profiles using Hyperplanes} \label{sec:seperate-hyperlane}  
Now that we have connected the degree profiles that occur in a community with dissonance ranges relative to that community, we can start showing that the behavior of the dissonance ranges implies results on our ability to separate the communities with hyperplanes, starting with the following proposition:

\begin{proposition}\label{NonoverlapLemma}\label{NonoverlapLemma-item-2}
Let $G \sim \CSBM_n^k(\rho, P, t)$, $\mathcal{Z}\subset \R^{2k}$ be a linear subspace,
and let $\delta > 0$ be such that $t\delta<1$. 
Suppose further that there are communities $i$ and $j$ such that the projections of $\DR_{\delta}(i)$ and $\DR_{\delta}(j)$ onto $\mathcal{Z}$ overlap. 
Then, for any $m \in \mathbb{N}$, 
there exists $\varepsilon>0$ such that for any unit vectors $w_1,...,w_{m}\in \mathcal{Z}$ and $\gamma_1,...,\gamma_{m}\in\mathbb{R}$,  with probability $1-o(1)$, $[(w_r)_{r=1}^{m},(\gamma_r \log n)_{r=1}^{m}]$ confuses communities $i$ and $j$ at level $\varepsilon\ln(n)$.
\end{proposition}
\begin{proof}
Let $z_0 \in \text{Proj}_{\mathcal{Z}}(\DR_{\delta}(i)) \cap \text{Proj}_{\mathcal{Z}}(\DR_{\delta}(j))$. There must exist $z_i\in \DR_\delta(i)$ and $z_j\in \DR_\delta(j)$ such that $\text{Proj}_{\mathcal{Z}}(z_i)=\text{Proj}_{\mathcal{Z}}(z_j)=z_0$. Now, let $\delta' = \frac{1}{2}\left(\delta + \frac{1}{t}\right)$, so that $\delta < \delta' < \frac{1}{t}$. By Lemma~\ref{DRproperties}, there exists $\mu > 0$ such that $B_2(\DR_{\delta}(i), \mu)\cap \mathbb{R}_+^{2k} \subseteq \DR_{\delta'}(i)$ and $B_2(\DR_{\delta}(j), \mu)\cap \mathbb{R}_+^{2k} \subseteq \DR_{\delta'}(j)$. So, if we let $\mu_0=\mu/3k$ and $z'_0=z_0+\text{Proj}_{\mathcal{Z}}(\mu_0,\mu_0,...,\mu_0)$ then 
\begin{align*}
B_2(z'_0,\mu_0) \cap \mathcal{Z}&=\text{Proj}_{\mathcal{Z}}(B_2(z_i+(\mu_0,...,\mu_0),\mu_0))\\
&\subseteq \text{Proj}_{\mathcal{Z}}(B_2(z_i,\mu)\cap\mathbb{R}_+^{2k})\\
&\subseteq \text{Proj}_{\mathcal{Z}}(\DR_{\delta'}(i))
\end{align*}
By the same logic, $B_2(z'_0,\mu_0) \cap \mathcal{Z}\subseteq \text{Proj}_{\mathcal{Z}}(\DR_{\delta'}(j))$.

Now, let $S_{\overline{d}}$ be the volume of a unit ball in $\overline{d}$ dimensions, and $d$ be the dimensionality of~$\mathcal{Z}$. Note that $B_2(z'_0,\mu_0) \cap \mathcal{Z}$ is a ball of radius $\mu_0$ in dimension $d$, and therefore it has volume $S_d\mu_0^d$.
For any hyperplane and any $\varepsilon_0>0$, the region of $B_2(z'_0,\mu_0) \cap \mathcal{Z}$ that is within $\varepsilon_0$ of the hyperplane has a volume strictly less than $2\varepsilon_0 S_{d-1}\mu_0^{d-1}$. Fix $m \in \mathbb{N}$ and set $\varepsilon_0 =  \mu_0 S_d/(2 m S_{d-1})$. It follows that for any unit vectors $w_1, \dots, w_m \in \mathcal{Z}$, and any $\gamma_1, \dots, \gamma_m \in \mathbb{R}$, the region of $B_2(z'_0,\mu_0) \cap \mathcal{Z}$ that is within $\varepsilon_0$ of \emph{any} hyperplane of the form $\{x : w_r^T x = \frac{\gamma_r}{t}\}$ has volume that is strictly less than $m \cdot 2\varepsilon_0 S_{d-1}\mu_0^{d-1} = \text{Vol}\left(B_2(z'_0,\mu_0) \cap \mathcal{Z} \right)$. 
Therefore, there exists a point $x_0\in B_2(z'_0,\mu_0) \cap \mathcal{Z}$ such that \[ \left|w_r^T x_0 - \frac{\gamma_r}{t}\right| > \varepsilon_0\] for all $1 \leq r \leq m$. So, there exists an open ball $B \subseteq B_2(z'_0,\mu_0)\cap\mathcal{Z}$, and $s \in \{-1,1\}^m$ such that for all $x \in B$, 
\[s_r \left(w_r^T x - \frac{\gamma_r}{t}\right) > \frac{\varepsilon_0}{2}\] for all $1 \leq r \leq m$. In other words, the open ball $B$ is separated from all hyperplanes defined by $(w_r, \frac{\gamma_r}{t})_{r=1}^m$.

Now, observe that for all sufficiently large $n$ there exist $x_i\in \DR_{\delta'}(i)\cap \text{Proj}^{-1}_{\mathcal{Z}}(B)$ and $x_j\in \DR_{\delta'}(j)\cap \text{Proj}^{-1}_{\mathcal{Z}}(B)$ such that $x_it\ln(n),x_jt\ln(n)\subseteq\mathcal{Z}_+^{2k}$. Since $t \delta' < 1$, by Lemma \ref{DRmeaning}, with probability $1-o(1)$, there exist vertices $u \in V_i$ and $v \in V_j$ such that $d(u)/(t\log(n))=x_i$ and $d(v)/(t\log(n)) = x_j$. By the above, we have
\[s_r \left(w_r^T  \frac{d(u)}{t\log n} - \frac{\gamma_r}{t} \right) = s_r \left(w_r^T \frac{d(v)}{t \log n} - \frac{\gamma_r}{t}\right) > \frac{\varepsilon_0}{2}\]
for all $1 \leq r \leq m$. Multiplying through by $t\log n$ and taking $\varepsilon = \frac{t\varepsilon_0}{2}$, we conclude that $[(w_r)_{r=1}^m, (\gamma_r \log n)_{r=1}^m]$ confuses communities $i$ and $j$ at level $\varepsilon \log n$.
\end{proof}

\subsection{When Dissonance Ranges Barely Overlap.} \label{sec:dis-overlap}
At this point, the key question is what hyperplanes can separate the rescaled degree profiles from different communities. 
In order to answer that, we consider the ``hardest'' case where $t =t_0= 1/\Delta_+(\theta_i,\theta_j)$, with  $\theta_i$ defined by \eqref{eq:t_c}. 
Recall that $w^\star$ is defined so that
\begin{eq}
w^\star = \bigg(\log \frac{P_{ri}}{P_{rj}}, \log \frac{1-P_{ri}}{1-P_{rj}} \bigg)_{r\in [k]}.
\end{eq}
Below, we show that the hyperplane orthogonal to $w^\star$ almost separates the dissonance ranges even for $t = t_0$. We also set up additional properties that will help us to show that a hyperplane orthogonal to $w\neq w^\star$ cannot separate the dissonance ranges just above $t_0$, and also to establish the impossibility of exact recovery (Theorem~\ref{theorem:impossibility}) below $t_0$.

\begin{lemma}\label{DRoverlapThreshold}
Suppose that $1\le i<j\le k$ and let $t_0=1/\Delta_+(\theta_i,\theta_j)$, where $\theta_i$ is defined by \eqref{eq:t_c}. Then $\DR_{1/t_0}(i)$ and $\DR_{1/t_0}(j)$ intersect at a single point. 
Let $x^\star$ be this intersection point of $\DR_{1/t_0}(i)$ and $\DR_{1/t_0}(j)$.
Let $H :=\{x: \langle w^\star, x-x^*\rangle \geq 0\}$ be the half-space created by the hyperplane through $x^\star$ perpendicular to~$w^\star$. 
Then $\DR_{1/t_0}(i) \cap H = \DR_{1/t_0}(i)$ and $\DR_{1/t_0}(j) \cap H =  \{x^\star\}$,
i.e., the hyperplane $\{x: \langle w^\star, x-x^*\rangle = 0\}$  separates $\DR_{1/t_0}(i) \setminus\{x^\star\}$ and  $\DR_{1/t_0}(j) \setminus\{x^\star\}$. 
Also, there exists $r>0$ such that $B_2(x^\star +  r w^\star,r \|w^\star\|_2) \subset \DR_{1/t_0}(i)$ and $B_2(x^\star -  r w^\star,r\|w^\star\|_2) \subset \DR_{1/t_0}(j)$. 
For $t < t_0$, the intersection $\DR_{1/t}(i) \cap \DR_{1/t}(j)$ has a non-empty interior.
\end{lemma}

\begin{proof}Recall the definition of $\Delta_+$ and $\CH_\xi$ from Definition~\ref{def:CH-threshold}. Let $\xi^\star$ be the maximizer of \eqref{eq:CH-defn}. 
We claim that $0<\xi^\star<1$. Indeed,
\begin{align*}
    &\Delta_+(\theta_i,\theta_j) = 1 - \min_{\xi\in [0,1]} \sum_{r\in [k]} \rho_r \Big(P_{ri}^{\xi}P_{rj}^{1-\xi} + (1-P_{ri})^\xi(1-P_{rj})^{1-\xi}\Big)\\
    &=: 1-\min_{\xi\in [0,1]} f(\xi). 
\end{align*}
Now, $f(0)= f(1) = 1$, and $f(1/2) < 1$ by the inequality of arithmetic and geometric means. Therefore, the minimum of $f$ is not attained at $\{0,1\}$, which proves $0<\xi^\star<1$. 

Next, define the $2k$-dimensional vector
\begin{align*}
    x^{\star} = \left(\rho_r P_{ri}^{\xi^\star} P_{rj}^{1-\xi^\star} , \rho_r (1-P_{ri})^{\xi^\star} (1-P_{rj})^{1-\xi^\star}\right)_{r\in [k]}.
\end{align*}
Setting $\frac{\dif}{\dif \xi} \CH_\xi(\theta_i,\theta_j) \big\vert_{\xi = \xi^\star} = 0$ yields
\begin{align} 
&\sum_{r\in [k]} \bigg[\rho_rP_{ri}^{\xi^\star} P_{rj}^{1-\xi^\star} \log \frac{P_{ri}}{P_{rj}} + \rho_r (1-P_{ri})^{\xi^\star}(1-P_{rj})^{1-\xi^\star} \log \frac{1- P_{ri}}{1 - P_{rj}} \bigg] \nonumber \\
&= \langle w^\star,x^\star\rangle = 0. \label{inner-prod-zero}
\end{align}
Also recall $\eta_i$ from \eqref{def:eta}. Then, 
\begin{align*}
\eta_i(x^\star)-\eta_j(x^\star) &= \sum_{r=1}^k\bigg[ x_{1,r}^\star\ln\bigg(\frac{x_{1,r}^\star}{\e\rho_r P_{ri}} \bigg)+x_{2,r}^\star\ln\bigg(\frac{x_{2,r}^\star}{\e\rho_r (1-P_{ri})}\bigg)\bigg] \\
&- \sum_{r=1}^k\bigg[ x_{1,r}^\star\ln\bigg(\frac{x_{1,r}^\star}{\e\rho_r P_{rj}} \bigg)+x_{2,r}^\star\ln\bigg(\frac{x_{2,r}^\star}{\e\rho_r (1-P_{rj})}\bigg)\bigg]\\
& = \sum_{r=1}^k\bigg[ x_{1,r}^\star\ln\bigg(\frac{P_{rj}}{P_{ri}} \bigg)+x_{2,r}^\star\ln\bigg(\frac{1-P_{ri}}{1-P_{rj}}\bigg)\bigg]\\
&=\langle w^\star,x^\star \rangle = 0. 
\end{align*}
Therefore, 
\begin{align*}
\eta_i(x^\star)=\eta_j(x^\star)&=\xi^\star\eta_i(x^\star)+(1-\xi^\star)\eta_j(x^\star)\\
&=\sum_{r=1}^k\bigg[ x_{1,r}^\star\ln\bigg(\frac{x_{1,r}^\star}{\e\rho_{r} P_{ri}^{\xi^\star}P_{rj}^{1-\xi^\star}}\bigg) \\
& \quad \quad \quad +x_{2,r}^\star\ln\bigg(\frac{x_{2,r}^\star}{\e\rho_{r} (1-P_{ri})^{\xi^\star}(1-P_{rj})^{1-\xi^\star}}\bigg)\bigg]+1\\
&=\sum_{r=1}^k\left[ x_{1,r}^\star\ln(1/\e)+x_{2,r}^\star\ln(1/\e)\right]+1\\
&=1-\sum_{r=1}^k \big[x_{1,r}^\star+x_{2,r}^\star\big]\\
&=\sum_{r=1}^k \left[\rho_{r}-\rho_{r} P_{ri}^{\xi^\star}P_{rj}^{1-\xi^\star}-\rho_{r} (1-P_{ri})^{\xi^\star}(1-P_{rj})^{1-\xi^\star}\right]\\
&=\Delta_+(\theta_i,\theta_j).
\end{align*}
Therefore, $x^\star\in \DR_{1/t_0}(i)\cap \DR_{1/t_0}(j)$. Next, observe that 
\begin{align*}
    \nabla \eta_i(x^\star) &= \bigg(\ln\bigg(\frac{x_{1,r}^\star}{\e\rho_r P_{ri}} \bigg) + 1, \ln\bigg(\frac{x_{2,r}^\star}{\e\rho_r (1-P_{ri})}\bigg) + 1 \bigg)_{r\in [k]} \\
    & = \bigg(\ln\bigg(\frac{P_{ri}}{P_{rj}} \bigg)^{\xi^\star - 1} , \ln\bigg(\frac{1-P_{ri}}{1-P_{rj}} \bigg)^{\xi^\star - 1}  \bigg)_{r\in [k]} \\
    &= (\xi^\star - 1) w^\star. 
\end{align*}
Similarly, we also have that 
\begin{align*}
    \nabla \eta_j(x^\star) &= \bigg(\ln\bigg(\frac{x_{1,r}^\star}{\e\rho_r P_{rj}} \bigg) + 1, \ln\bigg(\frac{x_{2,r}^\star}{\e\rho_r (1-P_{rj})}\bigg) + 1 \bigg)_{r\in [k]} \\
    & = \bigg(\ln\bigg(\frac{P_{ri}}{P_{rj}} \bigg)^{\xi^\star } , \ln\bigg(\frac{1-P_{ri}}{1-P_{rj}} \bigg)^{\xi^\star }  \bigg)_{r\in [k]} = \xi^\star  w^\star. 
\end{align*}
By convexity of $\eta_i$ and $\eta_j$, for any $x\in [0,\infty)^{2k}$, we have that 
\begin{align} 
&\eta_i(x)\ge \eta_i(x^\star)+\langle x-x^\star, \nabla \eta_i(x^\star)\rangle \nonumber \\ 
&=\frac{1}{t_0}+\langle x-x^\star,  (\xi^\star-1)w^\star\rangle \label{eq:eta_i-ineq}
\end{align}
and 
\begin{align}
&\eta_j(x)\ge \eta_j(x^\star)+\langle x-x^\star, \nabla \eta_j(x^\star)\rangle \nonumber \\
&=\frac{1}{t_0}+\langle x-x^\star,  \xi^\star w^\star\rangle.\label{eq:eta_j-ineq}
\end{align}
In \eqref{eq:eta_i-ineq} and \eqref{eq:eta_j-ineq}, equality can hold only at $x=x^\star$ since $\eta_i$ and $\eta_j$ are strictly convex. 
Now, for any $x \in \DR_{1/t_0} (i)$, we have $\eta_i(x) \leq 1/t_0$. Thus, \eqref{eq:eta_i-ineq} implies that $ (\xi^\star - 1) \langle x - x^\star, w^\star\rangle \leq 0$, in which case, we must have $ \langle x - x^\star, w^\star\rangle \geq 0$ for all $x \in \DR_{1/t_0} (i)$, and therefore $\DR_{1/t_0} (i) \subset H$.
Moreover,~\eqref{eq:eta_j-ineq} implies that $ \xi^\star \langle x - x^\star, w^\star\rangle \leq 0$, and since $0<\xi^\star<1$, the equality holds if and only if $x = x^\star$. Therefore, $\DR_{1/t_0} (j) \cap H = \{x^\star\}$, which proves the first part of the claim.

Next, observe that by continuity of the second derivatives of $\eta_i$ and $\eta_j$, there must exist $r_0, c>0$ such that for all $x$ with $||x-x^\star||_2\le r_0$,
\begin{align*} 
\eta_i(x)&\le \eta_i(x^\star)+\langle x-x^\star, \nabla \eta_i(x^\star)\rangle+c\|x-x^\star\|_2^2\\
&=\frac{1}{t_0}+\langle x-x^\star,  (\xi^\star-1)w^\star\rangle +c\|x-x^\star\|_2^2\\
&=\frac{1}{t_0}+c\|x-x^\star+(\xi^\star-1)w^\star/2c\|_2^2\\
& \quad \quad -\|(\xi^\star-1)w^\star\|_2^2/4c\\
&\le \frac{1}{t_0},
\end{align*}
for  $\|x-x^\star+(\xi^\star-1)w^\star/2c\|_2\le \|(\xi^\star-1)w^\star\|_2/2c$, and 
\begin{align*} 
\eta_j(x)&\le \eta_j(x^\star)+\langle x-x^\star, \nabla \eta_j(x^\star)\rangle+c\|x-x^\star\|_2^2\\
&=\frac{1}{t_0}+\langle x-x^\star,  \xi^\star w^\star\rangle+c\|x-x^\star\|_2^2\\
&=\frac{1}{t_0}+c||x-x^\star+\xi^\star w^\star/2c||_2^2-||\xi^\star w^\star||_2^2/4c\\
&\le \frac{1}{t_0}, 
\end{align*}
for  $\|x-x^\star+\xi^\star w^\star/2c\|_2\le \|\xi^\star w^\star\|_2/2c$.
In order to ensure that $\eta_i(x), \eta_j(x) \leq 1/t_0$, set $r=\min(r_0/\|w^\star\|_2,\xi^\star/c,(1-\xi^\star)/c)/2$. The ball of radius $r||w^\star||_2$ centered on $x^\star-r w^\star$ is completely contained in $\DR_{1/t_0}(j)$ and the ball of radius $r||w^\star||_2$ centered on $x^\star+r w^\star$ is completely contained in $\DR_{1/t_0}(i)$, as desired.

Finally, for $t < t_0$, observe that   $B_2(\DR_{1/t_0}(i),\tilde{\varepsilon}) \subset \DR_{1/t}(i)$ for some $\tilde{\varepsilon}>0$, and thus $x^\star$ is in the interior of $\DR_{1/t}(i)$. 
Similarly,  $x^\star$ is in the interior of $\DR_{1/t}(j)$.
Therefore, the intersection $\DR_{1/t}(i) \cap \DR_{1/t}(j)$ has a non-empty interior. 
\end{proof}

\subsection{A Necessary and Sufficient Condition for Optimal Recovery.} \label{sec:dis-optimal-recovery}
Finally, we combine the results of the above sections to prove Proposition \ref{prop:seperation-verification}. Recall the notions of separating communities and confusing communities from Definitions~\ref{defn:separate} and ~\ref{defn:confuse}. 

\begin{proof}[Proof of Proposition~\ref{prop:seperation-verification}]
To prove the first part, define $t_0=1/\Delta_+(\theta_i,\theta_j)$, so that $t_0 < t$. By Lemma~\ref{DRoverlapThreshold}, there exists~$x^\star$ such that $\DR_{1/t_0}(i)\cap \DR_{1/t_0}(j)=\{x^\star\}$. Additionally, the hyperplane $\{x : \langle w^{\star} ,x - x^{\star}\rangle = 0\}$ separates $\DR_{1/t_0}(i)$ and $\DR_{1/t_0}(j)$. Note that by \eqref{inner-prod-zero}, the hyperplane is equivalently written as $\{x : \langle w^{\star}, x \rangle = 0\}$. 
Thus, for all $x \in \DR_{1/t_0}(i)$, we have $\langle w^{\star} ,x \rangle {\geq} 0$, while for all $x \in \DR_{1/t_0}(j)$, we have $\langle w^{\star} ,x \rangle {\leq} 0$.

Since $\DR_{1/t_0}(i)$ and $\DR_{1/t_0}(j)$ are both closed, convex sets, $x^\star$ is neither in the interior of $\DR_{1/t_0}(i)$ nor in the interior of $\DR_{1/t_0}(j)$. 
Fix some $\delta \in(\frac{1}{t}, \frac{1}{t_0})$.
By Lemma \ref{DRproperties}, there exists $\varepsilon'>0$ such that $B_2(\DR_{\delta}(i),\varepsilon') \subset \DR_{1/t_0}(i)$. Therefore, we can conclude that $x^\star \notin \DR_{\delta}(i)$. Similarly  $x^\star \notin \DR_{\delta}(j)$.
Hence,  $\DR_{\delta}(i) \cap \DR_{\delta}(j) = \varnothing$.
Also, since $\DR_{1/t_0}(i)\backslash\{x^\star\} \subset \{x:\langle w^{\star} ,x \rangle > 0\}$, and $\DR_{1/t_0}(j)\backslash\{x^\star\} \subset \{x:\langle w^{\star} ,x \rangle < 0\}$, we can conclude that the hyperplane $\{x:\langle w^{\star} ,x \rangle= 0\}$ separates $\DR_{\delta}(i)$ and $\DR_{\delta}(j)$. Since dissonance ranges are closed by Lemma~\ref{DRproperties}, there exists $\varepsilon > 0$ such that for any $x^{\sss (i)} \in \DR_{\delta}(i)$ and  $x^{\sss (j)} \in \DR_{\delta}(j)$, we have
\begin{align*}
\langle w^{\star}, x^{\sss (i)} \rangle> \frac{\varepsilon}{2t} \quad \text{and}\quad 
\langle w^{\star}, x^{\sss (j)}\rangle < -\frac{\varepsilon}{2t}.
\end{align*}
By Lemma \ref{DRmeaning}, $d(u)/(t\ln(n)) \in \DR_{\delta}(i)$ for every $u \in V_i$ with probability $1-o(1)$. Similarly, $d(v)/(t\ln(n)) \in \DR_{\delta}(j)$ for every $v \in V_j$ with probability $1-o(1)$. Therefore, with probability $1-o(1)$, we have that for all $u \in V_i$ and $v \in V_j$,
\begin{align*}
\langle w^{\star}, d(u) \rangle > \frac{\varepsilon}{2} \log(n) \quad \text{and} \quad 
\langle w^{\star}, d(v)\rangle < -\frac{\varepsilon}{2} \log(n).
\end{align*}
We conclude that $w^{\star}$ separates communities $i$ and $j$ with margin $\varepsilon \log n$ with high probability.

Next, suppose that $w^{\star} \not\in \mathcal{Z}$.
By Lemma \ref{DRoverlapThreshold}, there exists $r>0$ such that 
$B_2(x^\star+rw^\star, r\|w^\star\|_2) \subset \DR_{1/{t_0}}(i)$ and $B_2(x^\star-rw^\star, r\|w^\star\|_2) \subset \DR_{1/{t_0}}(j)$. 
Next, let $w'$ be the projection of $w^\star$ onto $\mathcal{Z}$. The fact that $w^\star\notin \mathcal{Z}$ implies that $w^\star-w'\neq 0$ and $\|w'\|_2<\|w^\star\|_2$.
Let $x^{\sss (i)} = x^\star + r(w^{\star}-w')$ and $x^{\sss (j)}= x^\star-r(w^{\star}-w')$.
We claim that there exists a sufficiently small $r'>0$ such that 
\begin{align}\label{ball-inclusion-DR}
    B_2(x^{\sss (i)}, r') \subset \DR_{1/t_0}(i) \text{ and } B_2(x^{\sss (j)}, r') \subset \DR_{1/t_0}(j). 
\end{align}
Indeed, take $y \in B_2(x^{\sss (i)}, r')$. Then, 
\begin{align*}
    \|y - (x^\star+rw^\star) \|_2 \leq r' + r\|w'\|_2.
\end{align*}
Since $\|w'\|_2<\|w^\star\|_2$, we can pick $r'$ such that $ \|y - (x^\star+rw^\star) \|_2 \leq r\|w^\star\|_2$, and therefore $B_2(x^{\sss (i)}, r') \subset B_2(x^\star+rw^\star, r\|w^\star\|_2) \subset \DR_{1/t_0}(i)$. The second conclusion of \eqref{ball-inclusion-DR} follows similarly.

By \eqref{ball-inclusion-DR}, since $x^{(i)}$ and $x^{(j)}$ lie in interiors of $\DR_{1/t_0}(i)$ and $\DR_{1/t_0}(j)$ respectively, there exists $\mu>0$ such that, for any $t<t_0+\mu$, $x^{\sss (i)}$ and $x^{\sss (j)}$ also lie in interiors of $\DR_{1/t}(i)$ and $\DR_{1/t}(j)$ respectively. Note that $\proj_{\mathcal{Z}}(x^{\sss (i)}) = \proj_{\mathcal{Z}}(x^{\sss (j)})$, therefore the projections $\DR_{1/t}(i)$ and $\DR_{1/t}(j)$ onto $\mathcal{Z}$ overlap. The desired conclusion follows by Proposition \ref{NonoverlapLemma}.
\end{proof}

\section{Achievability and Impossibility} \label{sec:achieve-impossible}
Let us define the Maximum A Posteriori (\textsc{MAP}) estimator, which is the optimal estimator of $\true$. \footnote{{Here, the MAP estimator is optimal because it minimizes the 0-1 loss; that is, it minimizes $\mathbb{P}(\hat{\sigma} \neq \sigma_0)$ over all estimators $\hat{\sigma}$.}} Given a realization $G$ of the censored graph, the \textsc{MAP} estimator outputs $\map \in \argmax_{\sigma} \mathbb{P}(\true = \sigma \mid  G )$, choosing uniformly at random from the argmax set.
In this section, we start by proving Theorem~\ref{theorem:impossibility}, which is essentially equivalent to showing that $\map$ does not succeed in exact recovery for $t<t_c$. Next we prove that the estimator $\map$ always succeeds for $t>t_c$. This shows the statistical achievability for the exact recovery problem. 

{\begin{remark} 
Following the original posting of this paper to arXiv, we came to know of an earlier work of Yun and Proutiere~\cite{YP16}, which establishes the information-theoretic threshold for a general class of labeled stochastic block models. Theorems~\ref{theorem:impossibility} and \ref{theorem:achievability-general} can be obtained by verifying the conditions of \cite[Theorem 3]{YP16}. In more detail, the positive direction of \cite[Theorem 3]{YP16} shows that an iterative spectral algorithm recovers the communities with high probability above the threshold. The negative direction rules out the existence of an algorithm that succeeds with high probability, which is weaker than Theorem \ref{theorem:impossibility}. However, the negative result can be strengthened, using some intermediate results found in the proof of \cite[Theorem 3]{YP16}, to say that below the threshold, any algorithm fails to recover the communities with high probability. The results of Yun and Proutiere are stated in terms of another divergence quantity, which is asymptotically related to the CH divergence as per \cite[Claim 4]{YP16}.\\
We include our original proofs for completeness. Theorem~\ref{theorem:impossibility} is a straightforward consequence of the machinery developed in Section~\ref{sec:geometric}, and the proof of Theorem~\ref{theorem:achievability-general} is non-algorithmic, instead directly analyzing the MLE.
\end{remark}}

\subsection{Impossibility}
\begin{proof}[Proof of Theorem \ref{theorem:impossibility}]
Recall that we have  $t<t_c$ in this case where $t_c$ is given by \eqref{eq:t_c}. Fix $i<j$ such that $t<t_0 = 1/\Delta_+(\theta_i,\theta_j)$. Using the final conclusion of Lemma~\ref{DRoverlapThreshold}, we have that $\DR_{1/t} (i) \cap \DR_{1/t} (j)$ contains an open ball. By Lemma~\ref{DRmeaning}~\eqref{DRmeaning-1}, there exists $d\in \Z_+^{2k}$ and $L_n=n^{\Omega(1)}$ such that $d/(t\log n) \in \DR_{1/t} (i) \cap \DR_{1/t} (j)$ and with probability $1-o(1)$ there are $L_n$ vertices from Community $i$ and $L_n$ vertices from Community $j$, all with degree profile $d$. Let $S_i$ and $S_j$ respectively denote these two sets. Note that, if $u \in S_i$ and $v \in S_j$ is such that $\{u,v\}$ is censored, then swapping the labels of $u$ and $v$ leads to an equiprobable community labeling. 

Next we will show that with high probability, there exists $\overline{S}_i \subset S_i$ with $|\overline{S}_i| \geq \frac{L_n}{2}$ and an injective mapping $f : \overline{S}_i \to S_j$ such that for every $u \in \overline{S}_i$, the observation between $u$ and $f(u)$ is censored. In other words, we will show that there are at least $\frac{L_n}{2}$ non-intersecting pairs of vertices, one in~$S_i$ and the other in $S_j$, with censored observations. In order to construct the pairings, let $f$ be a uniformly chosen bijection between $S_i$ and $S_j$. Enumerating the edges in the matching as $\{1, 2, \dots, L_n\}$, let $X_e$ be the indicator that the $e\text{-th}$ such edge corresponds to a non-censored observation. We will bound $\sum_{e=1}^{L_n} X_e$ by applying Markov's inequality. 

In order to bound $\mathbb{E}[X_e]$, consider the background Erd\H{o}s--R\'enyi graph with parameter $\frac{t\log n}{n}$, where edges in the background graph correspond to edges and non-edges in $G$, and where non-edges in the background graph correspond to censored observations in $G$. 
Applying \cite[Corollary 2.4]{JLR00}, with high probability, the maximal degree in the background graph is $7t \log n$.
Let $E_{\text{deg}}$ denote this event. It follows that $\mathbb{E}[X_e \mid E_{\text{deg}}] \leq \frac{7t \log n}{L_n}$, since the $e{\text{-th}}$ edge is chosen uniformly at random from $S_i \times S_j$. Thus, Markov's inequality implies that
\begin{align*}
\mathbb{P}\left(\sum_{e=1}^{L_n} X_e \geq \frac{L_n}{2}~\bigg\vert~ E_{\text{deg}}\right) &\leq \frac{7t \log n}{L_n/2} = o(1).
\end{align*}
Since $\mathbb{P}(E_{\text{deg}}) = 1-o(1)$, it follows that $\mathbb{P}\left(\sum_{e=1}^{L_n} X_e \geq \frac{L_n}{2} \right) = o(1)$. Hence, the desired $\overline{S}_i$ and~$f$ exist.

Recall that $\hat{\sigma}_{\text{MAP}}$ makes a uniformly random selection from among the set of labelings with greatest posterior probability. For fixed $\sigma$, let $Q(\sigma; G)$ denote its posterior probability (so that $Q(\sigma_0; G)$ is a random variable). In order to lower-bound the probability that $\map$ returns the incorrect labeling, we introduce the set $\Sigma = \{\sigma : Q(\sigma; G) = Q(\sigma_0; G)\}$. The above pairings certify that $|\Sigma| \geq \frac{L_n}{2} = \omega(1)$ with high probability. It follows that $\PR(\map \text{ fails} \mid G) \geq 1- \frac{2}{|\Sigma|}$. Taking expectations on both sides shows that $\map$ fails with high probability. Consequently, any other estimator also fails in exact recovery with high probability, completing the proof.
\end{proof}
\subsection{Statistical Achievability}
\begin{theorem}\label{theorem:achievability-general}
Let $G \sim \CSBM_n^k(\rho, P, t)$.
If $t > t_c$, then $$\lim_{n \to \infty} \mathbb{P}(\map \text{ achieves exact recovery}) = 1.$$
\end{theorem}
In order to prove Theorem \ref{theorem:achievability-general}, we require two concentration results. Given a graph $\mathcal{G} = (V, E)$ and $W \subseteq V$, let $e(W)$ be the number of edges with both endpoints in $W$. 
\begin{lemma}[{\cite[Corollary 2.3]{Krivelevich2006}}]\label{lemma:pseudorandom}
Let $0\leq p_n \leq 0.99$ and let $\mathcal{G}$ be a sample from an Erd\H{o}s-R\'enyi random graph on vertex set $[n]$ and with edge probability $p_n$. Then, with probability $1-o(1)$,
\[\left|e(W) -\binom{|W|}{2} p_n\right| \leq O(\sqrt{n p_n}) |W| \quad \text{for all } W \subseteq [n].\]
\end{lemma}
\begin{lemma}\label{lemma:Schudy-Sviridenko}
Let $X_1, X_2, \dots, X_n$ be a sequence of independent discrete random variables, whose support is a finite set $S$. Let $X = \sum_{i=1}^n X_i$ and $Y = \sum_{i=1}^n |X_i|$. Let $L = \max\{|x| : x \in S\}$. Then for any $\delta \in (0,1)$,
\[\mathbb{P}\left(\left|X - \mathbb{E}\left[X \right] \right| \geq \delta |\mathbb{E}\left[X \right]| \right) \leq \exp \bigg(2-C\delta^2 \frac{\left(\mathbb{E}[X]\right)^2}{L \mathbb{E}[Y]} \bigg),\]
where $C>0$ is a universal constant.
\end{lemma}
The proof of Lemma~\ref{lemma:Schudy-Sviridenko} follows directly from \cite[Theorem 1.3]{Schudy2012}. See Appendix \ref{sec:appendix-sum-concentration} for details.
We will also need the following definitions in the proof of Theorem~\ref{theorem:achievability-general}. 

\begin{definition}[Permissible relabeling]
A permutation $\pi : [k] \to [k]$ is called a \emph{permissible relabeling} if $\rho(i) = \rho(\pi(i))$ for all $i \in [k]$ and $P_{ij} = P_{\pi(i), \pi(j)}$ for all $i, j \in [k]$. Let $\mathcal{P}(\rho, P)$ denote the set of permissible relabelings.
\end{definition}
\begin{definition}[Discrepancy] \label{def:disc}
Given two assignments $\sigma, \sigma' : [n] \to [k]$, their \emph{discrepancy} $\discrepancy(\sigma, \sigma')$ is defined as 
\[\min_{\pi \in \mathcal{P}(\rho, P)}\left\{\ham((\pi \circ \sigma), \sigma') \right\},\]
where $\ham(\cdot, \cdot)$ denotes the Hamming distance.
\end{definition}
Note that, if an estimator $\hsig$ satisfies $\discrepancy(\hsig,\true) = 0$ with high probability, then $\hsig$ achieves exact recovery. Next, let $E_{+}$ and $E_-$ respectively denote the sets of present and absent edges of~$G$. 
For a community assignment $\sigma$, communities $i,j \in [k]$ 
and $\Box\in \{+,-\}$, define
\begin{align*}
&S^{ij}_{\sss \Box}(G,\sigma) = \{e = \{u,v\}\in E_{\sss \Box}: \{\sigma(u), \sigma(v)\}  = \{i,j\}\}\\\
&\text{and}\quad  s^{ij}_{\sss \Box}(G,\sigma) = |S^{ij}_{\sss \Box}(G,\sigma)|.
\end{align*}
For example, $s_-^{11}(G, \sigma)$ is the number of absent edges with both endpoints in community 1 according to $\sigma$. 
Define 
\begin{align} 
    &z(G,\sigma) = 2\sum_{i,j \in [k]:j \geq i}  \Big[s^{ij}_+(G,\sigma) \log P_{ij}+  s^{ij}_-(G,\sigma)\log (1-P_{ij})\Big].\label{defn:z}
\end{align}
{Note that $z(G,\sigma)$ is twice the log-likelihood of $G$ under $\sigma$.}
The idea is to show that the maximizer of $z(G,\sigma)$ yields a  configuration~$\sigma$ with zero discrepancy. 
We state this in the following two lemmas which deal with low and high values of discrepancies separately.
 
\begin{lemma}\label{lem:low-disc}
    There exists $c\in (0,1)$ such that with high probability
    \begin{eq} \label{eq:z-bound}
    &z(G, \sigma) < z(G, \true) \text{ for all } \sigma \text{ such that } 0<\discrepancy(\sigma, \sigma_0) \leq cn.
    \end{eq}
\end{lemma}
For the high discrepancy case, we need to restrict the range of $\sigma$. To that end, for any $\eta>0$, define
\begin{align}
    &\Sigma_0(\eta) := \Big\{\sigma: [n] \mapsto [k] : |\{v : \sigma(v) = i\}|   \in ((\rho_i-\eta) n, (\rho_i+\eta) n),\ \forall i\in [k] \Big\}.\label{defn-sigma-range}
\end{align}

\begin{lemma}\label{lem:high-disc}
  Fix any $c\in (0,1]$. There exists an $\eta>0$ such that with high probability
    \begin{align} 
    &z(G, \sigma) < z(G, \true) \text{ for all } \sigma\in \Sigma_0(\eta) \text{ such that } \discrepancy(\sigma, \sigma_0) \geq cn.\label{eq:z-bound-2}
    \end{align}
\end{lemma}

\begin{proof}[Proof of Theorem \ref{theorem:achievability-general}]
Fix $c$ such that both the conclusions of Lemmas~\ref{lem:low-disc}~and~\ref{lem:high-disc} hold. Let $\eta$ be picked according to Lemma~\ref{lem:high-disc}. 
Rather than analyzing the MAP estimator, we will analyze the estimator 
\[\overline{\sigma} = \argmax_{\sigma \in \Sigma_0(\eta)}\{z(G, \sigma)\}.\]
Lemmas~\ref{lem:low-disc}~and~\ref{lem:high-disc} yield  $\discrepancy(\overline{\sigma},\true) = 0$, and therefore $\overline{\sigma}$ succeeds in exact recovery, with high probability. Since the MAP estimator is optimal, this also implies that the MAP estimator succeeds in exact recovery with high probability. 
\end{proof}

\begin{proof}[Proof of Lemma~\ref{lem:low-disc}]
Let $\discrepancy(\sigma,\true) = \delta n$ for some $\delta>0$ (to be chosen later). Let $\pi\in \mathcal{P}(\rho, P)$ be such that $\ham(\sigma\circ \pi, \true) = \delta n$. 
However, since $z(G,\sigma) = z(G,\sigma\circ \pi)$ for any $\pi \in \mathcal{P}(\rho, P)$, we can without loss of generality assume that $\ham(\sigma,\true) = \delta n$.
Let us fix $\Box \in \{+,-\}$.
To prove \eqref{eq:z-bound}, we start by analyzing $s_{\sss \Box}^{rj}(G,\sigma) -s_{\sss \Box}^{rj}(G,\true)$ with $r,j\in [k]$. 
Fix $r\neq j$. We decompose 
\begin{eq} \label{decomp-difference}
    s_{\sss \Box}^{rj}(G,\sigma) -s_{\sss \Box}^{rj}(G,\true) &= \sum_{\{u,v\}\in E_{\sss \Box}} \1_{\{\{\sigma(u), \sigma(v)\} = \{r,j\}, \{\true(u),\true(v)\} =  \{r,i\}, i\neq j\}} \\
    &+ \sum_{\{u,v\}\in E_{\sss \Box}} \1_{\{\{\sigma(u), \sigma(v)\} = \{r,j\}, \{\true(u), \true(v)\} = \{i,j\}, i\neq r\}} \\
    &+ \sum_{\{u,v\}\in E_{\sss \Box}} \1_{\{\{\sigma(u), \sigma(v)\} = \{r,j\}, \{\sigma(u), \sigma(v)\}\cap\{\true(u), \true(v)\} = \varnothing\}} \\
    &-  \sum_{\{u,v\}\in E_{\sss \Box}} \1_{\{\{\sigma(u), \sigma(v)\} = \{r,i\},i \neq j, \{\true(u), \true(v)\} = \{r,j\}\}}  \\ 
    &- \sum_{\{u,v\}\in E_{\sss \Box}} \1_{\{\{\sigma(u), \sigma(v)\} = \{i,j\}, i\neq r, \{\true(u), \true(v)\} = \{r,j\}\}} \\
    &- \sum_{\{u,v\}\in E_{\sss \Box}} \1_{\{\{\true(u), \true(v)\} = \{r,j\}, \{\sigma(u), \sigma(v)\}\cap\{\true(u), \true(v)\} = \varnothing\}} .
\end{eq}
To analyze~\eqref{decomp-difference}, denote the six terms above by (I), (II), \dots, (VI) respectively.

Let $H_{\sss \Box} (\sigma)$ be the graph on $\{v:\sigma(v)\neq\true(v)\}$ where $\{u,v\}$ is an edge of $H_{\sss \Box} (\sigma)$ if and only if $\{u,v\} \in E_{\sss \Box}$. Let $e(H_{\sss \Box} (\sigma))$ denote the number of edges in $H_{\sss \Box} (\sigma)$. We will show that
\begin{eq} \label{equality-1-dif}
    \bigg|\text{(I)} - \sum_{i\in [k] \setminus \{j\}} \sum_{v:\sigma(v) = j, \true(v) = i} d_{{\sss \Box}r}(v)\bigg| \leq 3k\cdot e(H_{\sss \Box} (\sigma)).
\end{eq}
To compute (I), fix $i,r,j$, $r\neq j$ $i\neq j$, and consider two cases: \\ 

\noindent\textbf{Case I: $i,j,r$ are distinct.} Denote this contribution as (Ia). There are two subcases. Suppose that the $r$-labeled vertex under $\sigma,\true$ is the same vertex. 
Think of $u$ being such that $\sigma(u) = \true(u) = r$. 
The number of such edges is $\sum_{v:\sigma(v) = j, \true(v) = i} d_{{\sss \Box}r}(v) - \text{Err}_{\sss \text{(I)}}$, where $\text{Err}_{\sss \text{(I)}}$ is the number of $\{u,v\}\in E_{\sss \Box}$ such that $\sigma(u) \neq r,\true(u) = r, \sigma(v) = j, \true(v) = i$. To see this, note that the summation $\sum_{v : \sigma(v) = j, \true(v) = i} d_{{\sss \Box}r}(v)$ counts all edges (present or absent depending on $\Box = +$ or $\Box = -$) from $\{v : \sigma(v) = j, \true(v) = i\}$ to $\{u:\true(u) = r\}$. However, this causes an over-counting because these edges may be incident to $u$'s with $\sigma(u) \neq r$, resulting in the substraction of $\text{Err}_{\sss \text{(I)}}$. Note that $\text{Err}_{\sss \text{(I)}}$ is at most $e(H_{\sss \Box} (\sigma))$. Next, consider the second subcase, where the $r$-labeled vertex under $\sigma,\true$ are different. Since $i,j,r$ are distinct, such edges will have both endpoints in $\{v:\sigma(v)\neq\true(v)\}$. Therefore, 
\begin{eq}\label{distinct-contributions}
    \bigg|\text{(Ia)} - \sum_{v:\sigma(v) = j, \true(v) = i} d_{{\sss \Box}r}(v) \bigg| \leq 2 e(H_{\sss \Box} (\sigma)).
\end{eq}
\noindent\textbf{Case II: $i = r$.} Denote this contribution as (Ib). Since $r\neq j$, we only need to consider the case where one of the endpoints is labeled $r$ by both $\sigma,\true$. An argument identical to the first part of Case I shows 
\begin{eq}\label{identical-contributions}
    \bigg|\text{(Ib)} - \sum_{v:\sigma(v) = j, \true(v) = i} d_{{\sss \Box}r}(v) \bigg| \leq e(H_{\sss \Box} (\sigma)).
\end{eq}
Combining \eqref{distinct-contributions} and \eqref{identical-contributions}, \eqref{equality-1-dif} follows immediately. Bounds similar to \eqref{equality-1-dif} also hold for Terms (II), (IV), and (V). Term (III) is easily bounded by $e(H_{\sss \Box} (\sigma))$. Finally, we simply drop Term (VI) for upper bounding \eqref{decomp-difference}.

For $r = j$, we get a similar decomposition as~\eqref{decomp-difference}, except that the second and fifth terms would be omitted. For each of the terms, we can also prove \eqref{equality-1-dif}. In particular, 
\begin{align*}
\left|\text{(I)} -  \sum_{i\in [k] \setminus \{j\}} \sum_{v:\sigma(v) = j, \true(v) = i} d_{{\sss \Box}r}(v) \right| &\leq (k-1) \cdot e(H_{\sss \Box} (\sigma)) \leq 3k\cdot e(H_{\sss \Box} (\sigma)).
\end{align*}

Next, we need to bound $e(H_{\sss \Box} (\sigma))$. Note that the number of vertices in $H_{\sss \Box} (\sigma)$ is $\ham(\sigma,\sigma_0)$, where $\ham(\cdot, \cdot)$ denotes the Hamming distance.
Letting $\tau = \max_{a,b\in [k]} \max \{P_{ab}, 1-P_{ab}\}$, we see that there is a coupling such that, with probability 1, $H_{\sss \Box}(\sigma)$ is a subgraph of an Erd\H{o}s-R\'enyi random graph on vertex set $[n]$ and edge probability $\frac{\tau t \log n}{n}$. 
Applying Lemma \ref{lemma:pseudorandom}, we obtain that with probability $1-o(1)$
\begin{eq}\label{bound-Err-I}
    &e(H_{\sss \Box}(\sigma) )\leq \frac{\tau t \log n}{2n} \ham(\sigma, \true)^2 + O(\sqrt{\log n}) \ham(\sigma, \true) \text{ for all }\sigma \in [k]^n.
\end{eq}
Combining \eqref{equality-1-dif} and \eqref{bound-Err-I}, we get an estimate for (I) in \eqref{decomp-difference}. 
Similar estimates for (II), (IV), (V) can be deduced using an identical argument. The term (III) can be directly bounded by $e(H_{\sss\Box} (\sigma))$ as well and (VI) can be dropped. 
Therefore, \eqref{decomp-difference} implies that with probability $1-o(1)$
\begin{eq} \label{diff-srj}
    s_{\sss\Box}^{rj}(G,\sigma) -s_{\sss\Box}^{rj}(G,\true) &\leq \sum_{i:i\neq j} \sum_{v:\sigma(v) = j, \true(v) = i} d_{{\sss\Box} r}(v) \\
    &\quad + \sum_{i:i\neq r} \sum_{u:\sigma(u) = r, \true(u) = i} d_{{\sss\Box} j}(u)\\
    &\quad - \sum_{i:i\neq j} \sum_{v: \sigma(v) = i, \true(v) = j} d_{{\sss\Box} r}(v)  \\
    &\quad- \sum_{i:i\neq r} \sum_{u: \sigma(u) = i, \true(u) = r} d_{{\sss\Box} j}(u) \\
    &\quad + \frac{8k\tau t \log n}{n} \ham(\sigma, \true)^2 + O(\sqrt{\log n} ) \ham(\sigma, \true).
\end{eq}
For $r = j$, a bound identical to \eqref{diff-srj} holds after omitting the second and the fourth terms.
Next, by Proposition \ref{prop:seperation-verification} \ref{prop:seperation-verification-1}, there exists $\varepsilon > 0$ such that for all $i,j\in [k]$ and $i>j$, with high probability, 
\begin{eq} \label{inner-prod-ij-sign}
\langle w^\star_{ij}, d(v) \rangle &= \sum_{r\in [k]} d_{+r}(v)\log \frac{P_{ri}}{P_{rj}} + d_{-r}(v)\log \frac{1-P_{ri}}{1-P_{rj}}\\ 
&\quad \begin{cases}
\geq \varepsilon \log n, \quad &\forall v : \true(v) = j\\
\leq - \varepsilon \log n, \quad &\forall v : \true(v) = i
\end{cases}
\end{eq}
Let $L = \sum_{i,j\in [k]} \sum_{r\in [k]}\big|\log \frac{P_{ri}}{P_{rj}}\big| + \big|\log \frac{1-P_{ri}}{1-P_{rj}}\big| $. Thus, \eqref{diff-srj} yields, with high probability, 
\begin{align*}
    z(G,\sigma) - z(G,\true) &\leq 2\sum_{\substack{i,j\in [k]\\i> j}} \sum_{r\in [k]} \bigg[\bigg(\sum_{\substack{v:\sigma(v) = j,\\ \true(v) = i}} d_{+r}(v) - \sum_{\substack{v: \sigma(v) = i,\\ \true(v) = j}} d_{+r}(v) \bigg)\log \frac{P_{ri}}{P_{rj}}\\
    &\qquad + \bigg(\sum_{\substack{v:\sigma(v) = j,\\ \true(v) = i}} d_{-r}(v) - \sum_{\substack{v: \sigma(v) = i,\\ \true(v) = j}} d_{-r}(v) \bigg)\log \frac{1-P_{ri}}{1-P_{rj}}\bigg] \\
    & \qquad +  \frac{8k^3L\tau t \log n}{n} \ham(\sigma, \true)^2 + O(\sqrt{\log n})  \ham(\sigma, \true) \\
    & = 2\sum_{\substack{i,j\in [k]:\\i> j}} \bigg[\sum_{\substack{v:\sigma(v) = j,\\ \true(v) = i}} \langle w^\star_{ij}, d(v) \rangle - \sum_{\substack{v: \sigma(v) = i,\\ \true(v) = j}} \langle w^\star_{ij}, d(v) \rangle \bigg]\\
    & \qquad +  \frac{8k^3L\tau t \log n}{n} \ham(\sigma, \true)^2 + O(\sqrt{\log n} ) \ham(\sigma, \true)\\
    & \leq -2 \ham(\sigma, \true) \varepsilon \log n+ \frac{8k^3L\tau t \log n}{n} \ham(\sigma, \true)^2 \\
    & \qquad+ O(\sqrt{\log n}) \ham(\sigma, \true). 
\end{align*}
Thus, for any 
$\delta \leq \frac{\varepsilon}{3k^3L\tau t}$, we can ensure that $z(G,\sigma) - z(G,\true)<0$ for all $\sigma$ with $\ham(\sigma, \true) = \delta n$  with high probability. Thus the proof follows by taking $c = \frac{\varepsilon}{2k^3 L\tau t}$.
\end{proof}

\begin{proof}[Proof of Lemma~\ref{lem:high-disc}]\footnote{Thank you to Charlie Guan for noting several errors present in an earlier version of this proof.} Fix $c\in (0,1]$. Define 
\begin{eq}\label{defn:eta}
    \eta &= \frac{1}{3} \left(\min\{|\rho_i - \rho_j| : \rho_i \neq \rho_j, i, j \in [k]\} \land \min\{\rho_i : i \in [k]\}\right)\wedge \frac{c}{k}.
\end{eq}
Throughout, we condition on the event that $\true\in \Sigma_0(\eta)$, where $\Sigma_0(\eta)$ is defined in~\eqref{defn-sigma-range}. Due to~\eqref{eq:concentration}, this conditioning event holds with high probability. 
Fix an assignment $\sigma \in \Sigma_0(\eta)$ satisfying $\discrepancy(\sigma, \sigma_0) \geq cn$. The idea is to show that $\mathbb{E}\left[z(G, \sigma) - z(G, \sigma_0)\right] \leq -C n \log n$, and use the concentration bound in Lemma~\ref{lemma:Schudy-Sviridenko} to conclude that \eqref{eq:z-bound} holds.

We first compute the expected difference $\E[z(G, \sigma) - z(G, \sigma_0)]$. 
Let $V_{ij}:= \{v:\true(v) = i, \sigma(v) = j\}$ and $\nu_{ij} = |V_{ij}|$. 
Now, fix $i,j,a,b$ with $i\ne a$ or $j\ne b$. The expected number of present edges between $V_{ij}$ and $V_{ab}$ is $\nu_{ij}  \nu_{ab} \times\alpha P_{ia}$, where $\alpha = \frac{t\log n}{n}$. The contribution of these edges to $\E[z(G, \sigma) - z(G, \sigma_0)]$ is 
\begin{align*}
&2\nu_{ij} \nu_{ab} \times \alpha P_{ia} \times \left(\log (P_{jb}) - \log(P_{ia})\right) =  2\nu_{ij}  \nu_{ab} \times \alpha \times P_{ia} \log \frac{P_{jb}}{P_{ia}} .    
\end{align*}
Similarly, the contribution from absent edges is
\[2\nu_{ij}  \nu_{ab} \times \alpha \times (1-P_{ia}) \log \frac{1-P_{jb}}{1-P_{ia}} .\]
Meanwhile, the contribution of present or absent edges from $V_{ab}$ to itself is
\[\nu_{ab}(\nu_{ab}-1) \times \alpha \times P_{aa} \log \frac{P_{bb}}{P_{aa}}+\nu_{ab}(\nu_{ab}-1)\times \alpha \times (1-P_{aa}) \log \frac{1-P_{bb}}{1-P_{aa}}.\]

Summing over all contributions, 
and noting that the contribution for the terms with $i = j$  and $a = b$ is zero due to the presence of $\log$ terms, we obtain 
    \begin{align*}
    \mathbb{E}\left[z(G, \sigma) - z(G, \sigma_0)\right] &= \alpha \sum_{i,j,a,b\in [k]} (\nu_{ij}  \nu_{ab}-\1_{a=i,b=j}\nu_{ij}) \bigg(P_{ia} \log \frac{P_{jb}}{P_{ia}} + (1-P_{ia}) \log \frac{1-P_{jb}}{1-P_{ia}} \bigg)\\
    &= -\alpha \sum_{i,j,a,b\in [k]} (\nu_{ij}  \nu_{ab}-\1_{a=i,b=j}\nu_{ij}) \DKL\left(P_{ia}, P_{jb} \right)\\
    &\le -\alpha \sum_{i,j,a,b\in [k]} \nu_{ij}  (\nu_{ab}-1) \DKL\left(P_{ia}, P_{jb} \right),
    \end{align*}
where $\DKL (\cdot,\cdot)$ denotes the Kullback--Leibler divergence. 
Our goal is to upper-bound the expectation. 
Note that all terms are nonpositive, so it suffices to bound a subset of the terms.
We treat two disjoint cases separately.\\

\noindent \textbf{Case 1:} \textit{For all $i$, there is at most one $j \in [k]$ such that $\nu_{ij} \geq \frac{\eta n}{k}$.} Fixing $i$, the pigeonhole principle then implies that there is exactly one such $j$. But since $\sum_{l} \nu_{il} \geq (\rho_i - \eta) n$, we know that 
\[\nu_{ij} \geq \rho_i n - (k-1)\frac{\eta n}{k} - \eta n > \left( \rho_i -2 \eta \right)n.\]

Let $\pi: [k] \to [k]$ denote the mapping such that $\nu_{i,\pi(i)} > \frac{\eta n}{k}$. We will next show that $\pi$ is a permutation. Supposing otherwise, there must exist $j \in [k]$ such that $\nu_{ij} < \frac{\eta n}{k}$ for all $i \in [k]$. But then
\[(\rho_i - \eta)n \leq \sum_{i} \nu_{ij} < \eta n,\]
which implies $\rho_i < 2\eta$, a contradiction.

Next, we argue that $\pi$ is not permissible. {Recall Definition~\ref{def:disc}}. Indeed, if $\pi$ were permissible, then 
\begin{align*}
    &\discrepancy(\sigma,\true) \leq \ham(\pi\circ \sigma,\true) \\
    &= \sum_{i} \sum_{j: j \neq \pi(i)} \nu_{ij} < k(k-1) \cdot \frac{\eta n}{k} < cn 
\end{align*}
where the second-to-last step uses the fact that by assumption $\nu_{ij}<\frac{\eta n}{k}$ whenever $j \neq \pi(i)$ 
and the last step follows from \eqref{defn:eta}. This leads to a contradiction and thus $\pi$ is not permissible.

Next, observe that
\begin{align*}
\sum_{i,a} \nu_{i\pi(i)}  (\nu_{a\pi(a)}-1)\DKL\left(P_{ia}, P_{\pi(i), 
 \pi(a)} \right)&\geq (1-o(1))n^2\sum_{i,a} (\rho_{i} - 2\eta) (\rho_{a} - 2\eta) \DKL\left(P_{ia}, P_{\pi(i), 
 \pi(a)} \right)\\
 &\geq \frac{(1-o(1))}{9} n^2\sum_{i,a} \rho_{i}\rho_{a} \cdot \DKL\left(P_{ia}, P_{\pi(i), 
 \pi(a)} \right).
\end{align*}
Since $\eta \leq \frac{1}{3} \min\{|\rho_i - \rho_j| :\rho_i \neq \rho_j, i,j \in [k]\}$, we claim that $\rho(\pi(i)) = \rho(i)$ for all $i \in [k]$. Under this claim, the fact that $\pi(\cdot)$ is not permissible implies that
there must exist $(i,a)$ for which $P_{ia} \neq P_{\pi(i), \pi(a)}$, and thus the above term is {at least} $Cn^2$ for some constant $C > 0$ and all sufficiently large $n$.
Hence, 
$\mathbb{E}\left[z(G, \sigma) - z(G, \sigma_0)\right] \leq -\frac{\alpha C n^2}{9}$.

To show that $\rho(\pi(i)) = \rho(i)$ for all $i \in [k]$, assume for the sake of contradiction that there exists $i \in [k]$ such that $\rho(\pi(i)) < \rho(i)$. But
\[(\rho(i) - 2\eta)n <  (\rho(i) - \eta)n - (k-1) \frac{\eta n}{k} \leq \nu_{i,\pi(i)} \leq \sum_{r} \nu_{r, \pi(i)} \leq (\rho(\pi(i)) + \eta)n\]
where the second inequality is due to the case assumption. But this implies 
\[\rho(i) - 2\eta < \rho(\pi(i)) + \eta\] and thus that \[  \rho(i) < \rho(\pi(i)) + 3\eta\]
leading to a contradiction. \\

\noindent \textbf{Case 2}: \textit{There exist $i, j, j'$ with $j \neq j'$ such that $\nu_{ij}, \nu_{ij'} \geq \frac{\eta n}{k}$.} Let $b \in [k]$ be such that $P_{jb} \neq P_{j'b}$ (which exists as we are implicitly assuming $t_c < \infty$). 
Let $a \in [k]$ be such that $\nu_{ab} \geq \frac{n(\rho_b - \eta)}{k}$, which is guaranteed to exist by the pigeonhole principle. 
Then either
\begin{equation*}
P_{ia} \neq P_{jb} \text{ or } P_{ia} \neq P_{j' b}.
\end{equation*}
{Recall that $\alpha = \nicefrac{t\log n}{n}$ is the censoring probability.}
Therefore,
\begin{align*}
-\mathbb{E}\left[z(G, \sigma) - z(G, \sigma_0)\right] &\geq \alpha \sum_{c,d,c',d'\in [k]} \nu_{cd}  (\nu_{c'd'}-1) \DKL\left(P_{cc'}, P_{dd'} \right)\\
&\geq (1-o(1))\alpha \frac{\eta n}{k} \cdot \frac{n(\rho_b - \eta)}{k} \left(\DKL\left(P_{ia}, P_{jb} \right) + \DKL\left(P_{ia}, P_{j'b} \right) \right) \\
&\geq (1-o(1))\alpha C' n^2.
\end{align*}
Summarizing both cases, we have shown that there exists a constant $C'' > 0$ 
such that
\begin{eq}
    \label{eq:difference-expectation-upper-bound}
    \mathbb{E}\left[z(G, \sigma) - z(G, \sigma_0)\right] \leq -\alpha C'' n^2 = - t C'' n \log n.
\end{eq}
for all sufficiently large $n$.
We next apply Lemma \ref{lemma:Schudy-Sviridenko} to establish concentration of the difference $z(G, \sigma) - z(G, \true)$. 
Letting $\mathscr{P}, \mathscr{A}$ denote the set of present and absent edges respectively, note that 
\begin{align*}
    X&:=\frac{1}{2} \E[z(G, \sigma) - z(G, \true)] \\
    &= \sum_{1\leq u<v\leq n}  \bigg[ \1_{\{\{u,v\} \in \mathscr{P}\}} \log \frac{P_{\sigma(u),\sigma(v)}}{P_{\true(u),\true(v)}} + \1_{\{\{u,v\} \in \mathscr{A}\}} \log \frac{1-P_{\sigma(u),\sigma(v)}}{1-P_{\true(u),\true(v)}}\bigg]. 
\end{align*}
Denote each term in the summation by $X_{uv}$. Then $X = \sum_{1\leq u<v \leq n}X_{uv}$ is a sum of independent random variables conditionally on $\true$, for any $\sigma\in [k]^n$. 
Let $Y = \sum_{1\leq u<v \leq n}|X_{uv}|$. Then for any $\delta \in (0,1)$,
\begin{align*}
&\mathbb{P}\left(z(G, \sigma) - z(G, \true) \geq (1-\delta) \mathbb{E}\left[z(G, \sigma) - z(G, \true)\right]\right) \leq \exp\bigg(2-C\delta^2 \frac{\left(\mathbb{E}[X]\right)^2}{L \mathbb{E}[Y]} \bigg),
\end{align*}
where $C$ is the universal constant from Lemma \ref{lemma:Schudy-Sviridenko}, and $L>0$ is a constant depending on $P,t$. 
To upper-bound $\mathbb{E}[Y]$, note that for any  $1 \leq u < v \leq n$, {we have $X_{uv} = 0$ whenever $\{\sigma(u),\sigma(v)\} = \{\true(u),\true(v)\}$, and}
\begin{align*}
\frac{\mathbb{E}\left[|X_{uv}| \right]}{|\mathbb{E}\left[X_{uv} \right]|}&= \Bigg(P_{\sigma_0(u), \sigma_0(v)}\left|  \log \frac{P_{\sigma(u),\sigma(v)}}{P_{\true(u),\true(v)}} \right| + \left(1 - P_{\sigma_0(u), \sigma_0(v)}\right) \left| \log \frac{1-P_{\sigma(u),\sigma(v)}}{1-P_{\true(u),\true(v)}} \right| \Bigg)\\
&\quad \quad \times \Bigg(\Big|P_{\sigma_0(u), \sigma_0(v)}\log \frac{P_{\sigma(u),\sigma(v)}}{P_{\true(u),\true(v)}} +  \left(1 - P_{\sigma_0(u), \sigma_0(v)}\right)  \log \frac{1-P_{\sigma(u),\sigma(v)}}{1-P_{\true(u),\true(v)}} \Big|\Bigg)^{-1},
\end{align*}
{whenever $\{\sigma(u),\sigma(v)\} \neq \{\true(u),\true(v)\}$. Since $0<P_{ij}<1$, taking a maximum on the right hand side of the above expression over $\sigma\in [k]^n$ with $\{\sigma(u),\sigma(v)\} \neq \{\true(u),\true(v)\}$ yields the following:} There exists a constant $C^{(1)} > 0$ depending on $P$ such that $\mathbb{E}\left[|X_{uv}| \right] \leq C^{(1)} |\mathbb{E}\left[X_{uv} \right]|$ for all $u,v$. It follows that 
\begin{align*}
\mathbb{E}[Y] &\leq C^{(1)} \sum_{1 \leq u < v \leq n} \left|\mathbb{E}[X_{uv}]\right| \leq C^{(2)} n\log n,
\end{align*} for some constant $C^{(2)}>0$.
Also, by \eqref{eq:difference-expectation-upper-bound}, $|\E[X]| \geq \frac{tC''}{2} n\log n$ for all sufficiently large $n$. Therefore there exists a constant $\tilde{C}>0$ such that
\begin{align*}
&\mathbb{P}\left(z(G, \sigma) - z(G, \true) \geq (1-\delta) \mathbb{E}\left[z(G, \sigma) - z(G, \true)\right]\right) \leq \exp\big(2-\tilde{C}\delta^2 n \log n \big).
\end{align*}
Taking $\delta = \frac{1}{2}$ and using \eqref{eq:difference-expectation-upper-bound}, we conclude that $z(G, \sigma) - z(G, \true) < 0$ with probability at least $1 - \exp(2-\frac{\tilde{C}}{4} n \log n )$. Finally, we take a union bound over the set $\{\sigma : \discrepancy(\sigma, \true) \geq cn\}$, whose cardinality is at most $k^n$. Since $\exp(2-\frac{\tilde{C}}{4} n \log n ) k^n = o(1)$, we conclude that \eqref{eq:z-bound-2} holds with high probability. 
\end{proof}

\section{Entrywise eigenvector bounds}\label{sec:entrywise} 
Our analysis of spectral algorithms relies on precise entrywise control of eigenvectors of adjacency matrices, which is guaranteed by the following result.
As before, {we work with} a fixed value of $\sigma_0$ satisfying \eqref{eq:concentration} with $\varepsilon = n^{-1/3}$.

\begin{theorem}\label{theorem:entrywise-bound-multiple}
Fix $k\in \N$, $\rho \in (0,1)^k$ such that $\sum_{i=1}^k \rho_i = 1$. {Fix} a symmetric matrix $P \in (0,1)^{k \times k}$, and let $G \sim \CSBM_n^k(\rho, P,t)$.
Define $A = A(G, y)$ for some constant $y > 0$, and let $A^{\star} = \mathbb{E}[A]$. Let $(\lambda_i, u_i)$ and $(\lambda_i^{\star}, u_i^{\star})$ denote  the $i$-th largest eigenpair of $A$ and $A^\star$ respectively {for $i\in [k]$}.
Let $r, s$ be integers satisfying $1 \leq r \leq k$ and $0 \leq s \leq k-r$.
Let $U = (u_{s+1}, \dots, u_{s+r}) \in \mathbb{R}^{n \times r}$, $U^{\star} = (u^{\star}_{s+1}, \dots, u^{\star}_{s+r}) \in \mathbb{R}^{n \times r}$, and $\Lambda^{\star} = \text{diag}(\lambda^{\star}_{s+1}, \dots, \lambda^{\star}_{s+r}) \in \mathbb{R}^{r \times r}$. Suppose that
\begin{eq} \label{delta-star-defn}
\Delta^{\star}:=(\lambda^{\star}_s - \lambda^{\star}_{s+1}) \land (\lambda^{\star}_{s+r} - \lambda^{\star}_{s + r +1}) \land \min_{i \in [r]} \left| \lambda^{\star}_{s+i}\right| > 0,    
\end{eq}
where $\lambda_0^{\star} = \infty$ and $\lambda_{k+1}^{\star} = -\infty$.
Then, with probability at least $1- O(n^{-3})$,
\begin{equation}
\inf_{O \in \mathcal{O}^{r \times r}} \left \Vert U O - A U^{\star} \left(\Lambda^{\star} \right)^{-1} \right\Vert_{2 \to \infty} \leq \frac{C}{\log \log (n) \sqrt{n}},
\end{equation}
for some $C= C(\rho, P, t,y) >0$, where $\mathcal{O}^{r \times r}$ denotes the set of $r\times r$ orthogonal matrices.
\end{theorem}
\begin{corollary}\label{corollary:entrywise}
Recall the notation from Theorem~\ref{theorem:entrywise-bound-multiple}.
If all eigenvalues of $A^{\star}$ are distinct and nonzero, then with probability $1-O(n^{-3})$, for all $i \in [k]$,
\[\min_{s \in \{\pm 1\}} \left\Vert s u_i - \frac{A u_i^{\star}}{\lambda_i^{\star}}\right \Vert_{\infty} \leq \frac{C}{\log \log(n) \sqrt{n}},\]
for some $C= C(\rho, P, t,y) >0$.
\end{corollary}
The proof of Theorem \ref{theorem:entrywise-bound-multiple} relies on an entrywise eigenvector perturbation bound derived in~\cite{Abbe2020}. 
We provide the  statement for a general random matrix $A$ here for completeness, and will verify these general conditions subsequently for $G\sim \CSBM_n^k(\rho,P,t)$.
Also, we reuse the notation from Theorem~\ref{theorem:entrywise-bound-multiple}. 
Let $H = U^T U^{\star}$, with singular value decomposition given by $H = \bar{W} \bar{\Sigma} \bar{V}^T$. Let $\text{sgn}(H) = \bar{W} \bar{V}^T \in \mathbb{R}^{r \times r}$, which is an orthonormal matrix, called the \emph{matrix sign function} \cite{Gross2011}.
Given this setup, \cite[Theorem 2.1 Part (2)]{Abbe2020} gives the following result.
\begin{theorem}[Theorem 2.1 Part (2), \cite{Abbe2020}]\label{theorem:abbe-entrywise}
Let $A$ be a random matrix as described above. Suppose that the following assumptions are satisfied, for some $\gamma > 0$ and $\varphi(x) : \mathbb{R}_+ \to \mathbb{R}_+$. 
\begin{enumerate}
    \item (Properties of $\varphi$) $\varphi(x)$ is continuous and non-decreasing in $\mathbb{R}_+$, $\varphi(0) = 0$, and $\frac{\varphi(x)}{x}$ is non-increasing in $\mathbb{R}_+$.
    \item (Incoherence) $\Vert A^{\star} \Vert_{2 \to \infty} \leq \gamma  \Delta^{\star}$, where $\Delta^\star$ is defined in \eqref{delta-star-defn}.
    \item (Row- and column-wise independence) For any $m \in [n]$, the entries in the $m$th row and column are independent with others, i.e. $\{A_{ij} : i = m \text{ or } j = m\}$ are independent of $\{A_{ij} : i \neq m, j \neq m\}$.
    \item (Spectral norm concentration) Define $\kappa = \frac{1}{\Delta^{\star}} \max_{i \in [r]}|\lambda^{\star}_{s+i}|$, and suppose $32 \kappa \max\{\gamma, \varphi(\gamma)\} \leq 1$. Then, for some $\delta_0 \in (0,1)$,
    \[\mathbb{P}\left(\Vert A - A^{\star} \Vert_2 \leq \gamma \Delta^{\star}\right) \geq 1 - \delta_0.\]
    \item (Row concentration)  There exists $\delta_1 \in (0,1)$ such that 
    for any $m \in [n]$ and $W \in \mathbb{R}^{n \times r}$,
    \begin{eq}
     &\mathbb{P}\Bigg(\left\Vert(A - A^{\star})_{m, \cdot} W \right\Vert_2  \leq \Delta^{\star} \Vert W \Vert_{2 \to \infty} \varphi\left(\frac{\Vert W \Vert_F}{\sqrt{n} \Vert W \Vert_{2 \to \infty}} \right) \Bigg) \geq 1 - \frac{\delta_1}{n}. \label{eq:row-concentration}   
    \end{eq}
\end{enumerate}
Then, with probability at least $1-\delta_0 - 2\delta_1$,
\begin{align*}
&\left\Vert U \mathrm{sgn}(H) - A U^{\star} (\Lambda^{\star})^{-1} \right \Vert_{2 \to \infty} \leq C_0\Big(\kappa (\kappa + \varphi(1)) (\gamma + \varphi(\gamma)) \Vert U^{\star} \Vert_{2 \to \infty} + \frac{\gamma}{\Delta^{\star}} \Vert A^{\star} \Vert_{2 \to \infty}\Big),    
\end{align*}
where $C_0>0$ is an absolute constant.
\end{theorem}
\begin{remark}\label{remark:single-eigenvector}
Theorem \ref{theorem:abbe-entrywise} can be applied to the recovery of a single eigenvector $u_l$ by setting $r = 1$ and $s = l-1$. In that case, the requirement \eqref{eq:row-concentration} simplifies to 
\begin{align*}
&\mathbb{P}\left(\left|(A - A^{\star})_{m, \cdot} \cdot w \right| \leq \Delta^{\star} \Vert w \Vert_{\infty} \varphi\left(\frac{\Vert w \Vert_2}{\sqrt{n} \Vert w \Vert_{\infty}} \right) \right) \geq 1 - \frac{\delta_1}{n}    
\end{align*}
for each $w \in \mathbb{R}^n$. The conclusion becomes 
\begin{align*}
&\min_{s \in \{\pm 1\}} \left\Vert s u_k - \frac{A u_k^{\star}}{\lambda_k^{\star}}\right \Vert_{\infty}\leq C_0\Big( \kappa (\kappa + \varphi(1)) (\gamma + \varphi(\gamma)) \Vert u_k^{\star} \Vert_{\infty} + \frac{\gamma}{\Delta^{\star}} \Vert A^{\star} \Vert_{2 \to \infty}\Big).    
\end{align*}
\end{remark}
In order to prove Theorem \ref{theorem:entrywise-bound-multiple}, we verify the five conditions of Theorem \ref{theorem:abbe-entrywise}. The following lemma states properties of the eigenspace of $A^\star$. 
\begin{lemma}\label{lemma:spectral-gap}
Let $G \sim \CSBM_n^k(\rho, P, t)$ where $\rho \in (0,1)^k$ is such that $\sum_i\rho_i = 1$, and $P \in (0,1)^{k \times k}$ is a symmetric matrix.
{For $y\in \R$}, define $A = A(G, y)$ as in Definition~\ref{def:signed-adj}. Denote $A^{\star} = \mathbb{E}[A]$ and let $(\lambda_l^\star,u_l^\star)_{l\in [k]}$ be the top $k$ eigenpairs. Then there {exist} constants $(\nu_l)_{l\in [k]}$ depending on $P$, $\rho$, $t$, and $y$ such that 
\begin{eq}\label{eq:lambda-star-asymp}
    \lambda_l^{\star} = (1+o(1)){\nu_l} \log(n)  \quad \text{for all } l\in [k]
\end{eq}
Moreover, {if the} $\nu_l$'s are distinct, then there exist constants $(\zeta_{lj})_{l,j\in [k]}$ depending on $P$, $\rho$, $t$, and $y$ such that
\begin{eq}\label{eq:u-star-asymp}
    u_{l{w}}^\star = (1+o(1)) \frac{\zeta_{lj}}{\sqrt{n}}  \quad \text{ for all }{w}\in\{v: \true(v) = j\}
\end{eq}
\end{lemma}  
\begin{proof}
{Let $B^{\star}$ be a block matrix, where the $i$th block has size equal to either $[\lfloor \rho_i n \rfloor$ or $ \lceil \rho_i n \rceil$, and the $(i,j)$ block takes value $\alpha(P_{ij} - y(1-P_{ij}))$. Then $B^{\star}$ is an approximation of $A^{\star}$ up to deviations in community sizes, permutation of community labels, and a diagonal correction (since $A^{\star}$ is a zero-diagonal matrix). Let $(\overline{\lambda}_l, \overline{u}_l)_{l \in [k]})$ be the top $k$ eigenpairs of $B^{\star}$. Since $B^{\star}_{ij} = \Theta\left(\frac{\log n}{n}\right)$ for all $i,j$, a straightforward adaptation of the proof of \cite[Lemma 3.2]{Dhara2022b} implies the existence of constants $(\nu_l)_{l\in [k]}$ depending on $P$, $\rho$, $t$, and $y$ such that
\begin{equation*}
    \overline{\lambda}_l^{\star} = (1+o(1)){\nu_l} \log(n)  \quad \text{for all } l\in [k].
\end{equation*}
By Weyl's theorem,
\[\left|\lambda_l - \overline{\lambda}_l \right| \leq \Vert A^{\star} - B^{\star} \Vert_2 \leq \Vert A^{\star} - B^{\star} \Vert_F\]
for all $l \in [k]$. Recall that $\sigma_0$ is assumed to satisfy \eqref{eq:concentration} with $\varepsilon = n^{-1/3}$. It follows that
\[\Vert A^{\star} - B^{\star} \Vert_F = O(\sqrt{n\cdot n^{2/3} \cdot \alpha^2}) = o(\log n),\]
which establishes \eqref{eq:lambda-star-asymp}.}

{
Regarding the eigenvectors, observe that due to the block nature of $B^{\star}$, the eigenvectors of $B^{\star}$ take on the form given by \eqref{eq:u-star-asymp}. Since $n_l = (1+o(1)) n\rho_l$ and $\lambda_l = (1+o(1)) \lambda_l^{\star}$ for all $l \in [k]$, it follows that the eigenvectors of $A^{\star}$ also take on the form given by \eqref{eq:u-star-asymp} (see \cite[Lemma 5.3]{Dhara2022a}). 
}
\end{proof}

Among the conditions in Theorem~\ref{theorem:abbe-entrywise}, only the fourth and the fifth are substantial. We verify them in the two lemmas below.
\begin{lemma}\label{lemma:entrywise-3}
Let $A$ be a symmetric and zero-diagonal random matrix. Suppose that the entries $\{A_{ij}: i < j \}$ are independent, $A_{ij} \in [a,b]$ for two constants $a < b$, and  $\mathbb{E}[|A_{ij}|] \leq p$ for all $i, j$, where $\frac{c_0 \log n}{n} \leq p \leq 1- c_1$ for constants $c_0, c_1 > 0$. Then, for any $c > 0$, there exists $c' > 0$ such that
\[\mathbb{P}\left(\Vert A - \mathbb{E}[A] \Vert_2 \leq c' \sqrt{np} \right) \geq 1 - 2n^{-c}.\]
\end{lemma}
\begin{proof}
Let $A = A^+ - A^-$, where $A^+_{ij} = \max\{A_{ij}, 0\}$ and $A^-_{ij} = - \min\{A_{ij}, 0\}$ for all $i,j$. Then
\begin{align}
\Vert A - \mathbb{E}[A] \Vert_2 
\leq \Vert A^+ - \mathbb{E}[A^+] \Vert_2  + \Vert A^- - \mathbb{E}[A_-] \Vert_2. \label{eq:spectral-norm-triangle}
\end{align}
Note that $A^+$ and $A^-$ are symmetric and zero-diagonal matrices with independent upper-triangular entries. Also, note that for all $i \neq j$,
\[\max\big\{\mathbb{E}[A_{ij}^+],\mathbb{E}[A_{ij}^-] \big\} \leq \mathbb{E}[|A_{ij}|] \leq p.\]
If $b \leq 0$, then $\Vert A^+ - \mathbb{E}[A^+] \Vert_2 = 0$. Otherwise, suppose $b > 0$. By \cite[Theorem 5]{Hajek2016}, for any $c > 0$, there exists $c_+ > 0$ such that
\begin{align}
&\mathbb{P}\left(\Vert A^+ - \mathbb{E}[A^+] \Vert_2 > c_+ \sqrt{b}\cdot  \sqrt{np}  \right) =  \mathbb{P}\bigg(\bigg\Vert \frac{1}{b}A^+ - \frac{1}{b}\mathbb{E}[A^+] \bigg\Vert_2 > c_+ \sqrt{\frac{np}{b}}   \bigg) \leq n^{-c}. \label{eq:spectral-norm-b}
\end{align}
Similarly, if $a \geq 0$, then $\Vert A^- - \mathbb{E}[A^-] \Vert_2 = 0$. Otherwise, suppose $a < 0$. By \cite[Theorem 5]{Hajek2016}, for any $c > 0$, there exists $c_- > 0$ such that 
\begin{align}
&\mathbb{P}\left(\Vert A^- - \mathbb{E}[A^-] \Vert_2 > c_- \sqrt{|a|} \sqrt{np}  \right) =  \mathbb{P}\bigg(\bigg\Vert \frac{1}{|a|}A^- - \frac{1}{|a|}\mathbb{E}[A^-] \bigg\Vert_2 > c_- \sqrt{\frac{np}{|a|}}   \bigg) \leq n^{-c}. \label{eq:spectral-norm-a}
\end{align}
Stated above in terms of upper bound on probabilities.
Take $c' = c_+ \sqrt{\max\{b,0\}} + c_- \sqrt{|\min\{a,0\}|}$. Combining \eqref{eq:spectral-norm-triangle}, \eqref{eq:spectral-norm-b}, and \eqref{eq:spectral-norm-a} along with a union bound, the proof is complete.
\end{proof}
\begin{lemma}\label{lemma:entrywise-4-multiple-eigenvectors}
Let $r \in \N$ be a constant, and $W \in \mathbb{R}^{n \times r}$ be a fixed matrix. Let $\{Z_i\}_{i=1}^n$ be independent random variables where $\mathbb{P}(Z_i = 1) = p_i$, $\mathbb{P}(Z_i  = -y) = q_i$, and $\mathbb{P}(Z_i = 0) = 1 - p_i - q_i$. Finally, let $\overline{Z} \in \mathbb{R}^n$, where $\overline{Z}_i = Z_i - \mathbb{E}[Z_i]$ for $i \in [n]$. Then for any $\beta \geq 0$,
\begin{align*}
&\mathbb{P}\bigg(\big\Vert \overline{Z}^T W\big\Vert_2 \geq r\frac{\max\{1,y\} (2+\beta) n}{1 \lor \log \big(\frac{\sqrt{n} \Vert W \Vert_{2 \to \infty}}{\Vert W \Vert_F} \big)} \Vert W \Vert_{2 \to \infty} \max_i\{p_i + q_i\}\bigg) \\
&\leq 2r \exp\big(-\beta n \max_i\{p_i + q_i\} \big).    
\end{align*}
\end{lemma}
\begin{proof} 
Let $w_j = W_{\cdot j}$ denote the $j$th column of $W$, for $j \in [r]$. We will show that
\begin{equation}
\frac{r\Vert W\Vert_{2 \to \infty}}{1 \lor \log \big(\frac{\sqrt{n} \Vert W \Vert_{2 \to \infty}}{\Vert W \Vert_F} \big)} \geq  \sum_{j=1}^r \frac{\| w_j\|_{\infty}}{1 \lor \log \big(\frac{\sqrt{n} \| w_j \|_{\infty}}{\| w_j \|_2} \big)} \label{eq:w-claim}.
\end{equation}
Given \eqref{eq:w-claim}, we then obtain
\begin{align}
&\mathbb{P}\bigg(\big\Vert \overline{Z}^T W\big\Vert_2 \geq  r\frac{\max\{1,y\} (2+\beta) n}{1 \lor \log \big(\frac{\sqrt{n} \Vert W \Vert_{2 \to \infty}}{\Vert W \Vert_F} \big)} \Vert W \Vert_{2 \to \infty} \max_i\{p_i + q_i\}\bigg) \nonumber\\
&\leq \mathbb{P}\bigg(\sum_{j=1}^r \bigg| \sum_{i=1}^n W_{ij} \overline{Z}_i \bigg| \geq  \max\{1,y\} (2+\beta)n  \cdot \max_i \{p_i + q_i\} \sum_{j=1}^r \frac{\Vert w_j\Vert_{\infty}}{1 \lor \log \big(\frac{\sqrt{n} \| w_j \|_{\infty}}{\| w_j \|_2} \big)}\bigg) \label{eq:triangle-inequality}\\
&\leq \sum_{i=1}^r \mathbb{P}\bigg(\bigg| \sum_{i=1}^n W_{ij} \overline{Z}_i \bigg| \geq   \frac{\max\{1,y\} (2+\beta)n \Vert w_j\Vert_{\infty}}{1 \lor \log \big(\frac{\sqrt{n} \Vert w_j \Vert_{\infty}}{\| w_j \|_2} \big)} \max_i \{p_i + q_i\}\bigg) \label{eq:union-bound}\\
&\leq 2r \exp\big(-\beta n \max_i\{p_i + q_i\} \big) \label{eq:r=1-case}.
\end{align}
Here \eqref{eq:triangle-inequality} follows from \eqref{eq:w-claim} and the fact that $\|x\|_2\leq \|x\|_1$ for any finite dimensional vector.
Next,~\eqref{eq:union-bound} follows by the union bound, and \eqref{eq:r=1-case} is an application of \cite[Lemma 5.2]{Dhara2022a}.

It remains to prove \eqref{eq:w-claim}. Since $\Vert w_j \Vert_2 \leq \Vert W \Vert_F$ for all $j \in [r]$, we obtain
\begin{align*}
\sum_{j=1}^r \frac{\| w_j\|_{\infty}}{1 \lor \log \big(\frac{\sqrt{n} \| w_j \|_{\infty}}{\| w_j \|_2} \big)} &\leq\sum_{j=1}^r \frac{\Vert w_j\Vert_{\infty}}{1 \lor \log \left(\frac{\sqrt{n} \Vert w_j \Vert_{\infty}}{\Vert W \Vert_F} \right)}.
\end{align*}
Let $g(c,x) := \frac{x}{1 \vee \log (c x) }$ for $c>0$. Then $\frac{\partial}{\partial x}g(c,x) = 1$ for $x< \e/c$, and $\frac{\partial}{\partial x}g(c,x) = \frac{\log (cx) -1 }{\log^2 (cx)}>0$ for $x> \e/c$.
Therefore, $g(c,\cdot)$ is increasing for any $c>0$. Since $\Vert w_j \Vert_{\infty} \leq \Vert W \Vert_{2 \to \infty}$ for all $j$, we obtain
\begin{align*}
&\sum_{j=1}^r \frac{\| w_j\|_{\infty}}{1 \lor \log \big(\frac{\sqrt{n} \| w_j \|_{\infty}}{\Vert W \Vert_F} \big)}  = \sum_{j=1}^r g\bigg(\frac{\sqrt{n}}{\|W\|_{F}}, \| w_j \|_{\infty}\bigg)\\
&\leq r g\bigg(\frac{\sqrt{n}}{\|W\|_{F}}, \Vert W \Vert_{2 \to \infty}\bigg) = \frac{r \Vert W\Vert_{2 \to \infty}}{1 \lor \log \left(\frac{\sqrt{n} \Vert W \Vert_{2 \to \infty}}{\Vert W \Vert_F} \right)},
\end{align*}
which completes the proof of \eqref{eq:w-claim}.
\end{proof}

\begin{proof}[Proof of Theorem \ref{theorem:entrywise-bound-multiple}]
We now verify the conditions of Theorem~\ref{theorem:abbe-entrywise} for the signed adjacency matrix $A = A(G,y)$ when $G\sim \CSBM_n^k(\rho,P,t)$.
Set 
\[\varphi(x) = r \frac{2 \log(n)}{\Delta^{\star} } \max\{1,y\}(t + 2) \bigg(1 \lor \log \Big(\frac{1}{x}\Big) \bigg)^{-1}.\]
Note that $\lim_{x \to 0^+}\varphi(0) = 0$ and $\frac{\varphi(x)}{x}$ is non-increasing on $\mathbb{R}_+$. Thus the first condition holds. 

To verify the second condition, we find that $\Vert A^{\star} \Vert_{2 \to \infty} = \Theta \big(\frac{ \log n}{\sqrt{n}} \big)$. 
Applying Lemma \ref{lemma:entrywise-3} with $c = 3$, and using the fact that $|A^{\star}_{ij}| \leq \frac{t \log(n)}{n}\max_{i,j\in [k]}P_{ij}$, there exists $c' > 0$ such that
\begin{equation}
\mathbb{P}\left(\Vert A - \mathbb{E}[A] \Vert_2 \leq c' \sqrt{\log(n)} \right) \geq 1 - n^{-3}. \label{eq:third-condition}   
\end{equation}
{By \eqref{delta-star-defn},} we have $\Delta^{\star} > 0$. Moreover, by Lemma \ref{lemma:spectral-gap}, we have $\Delta^{\star} = \Theta(\log(n))$.
Let $\gamma = c' \sqrt{\log(n) }/\Delta^{\star}$. Therefore, $\Vert A^{\star} \Vert_{2 \to \infty} \leq \gamma  \Delta^{\star}$ is satisfied for $n$ large enough. 

The third condition is immediate. 

The second part of the fourth condition holds with $\delta_0 = n^{-3}$ due to \eqref{eq:third-condition}. To verify the first part, note that $\kappa = \Theta(1)$ by Lemma \ref{lemma:spectral-gap} and $\gamma = o(1)$. 
Then $32 \kappa \max\{\gamma, \varphi(\gamma)\} \leq 1$ for all sufficiently large $n$.

To verify the fifth condition, fix $W \in \mathbb{R}^{n \times r}$ and $m \in [n]$. By Lemma \ref{lemma:entrywise-4-multiple-eigenvectors} with $p_i \in \{\frac{t\log n}{n} P_{ab}: a,b\in [k]\}$, $p_i+q_i = \frac{t\log n}{n}$
and $\beta = \frac{4}{t}$, we obtain 
\begin{align*}
&\mathbb{P}\bigg(  \big\|\big((A - A^{\star})_{m, \cdot}\big) \cdot W  \big\|_2  \geq r \frac{\max\{1,y\}(2+ 4/t)n}{1 \lor \log\left(\frac{\sqrt{n} \Vert W \Vert_{2 \to \infty}}{\Vert W \Vert_F} \right)} \Vert W \Vert_{2 \to \infty} \frac{t\log n}{n}  \bigg) \leq 2 r\exp\left(-\frac{4 n}{t} \times \frac{t\log n}{n}\right),    
\end{align*}
which can be re-written as
\begin{align*}
&\mathbb{P}\Big( \big\|\big((A - A^{\star})_{m, \cdot}\big) W  \big\|_2 \geq \Delta^{\star} \Vert W \Vert_{2 \to \infty} \varphi\left(\frac{\Vert W \Vert_F}{\sqrt{n} \Vert W \Vert_{2 \to \infty}} \right)  \Big) \leq 2r n^{-4}.    
\end{align*}
Therefore, the fifth condition is satisfied with $\delta_1 = 2r n^{-3}$. 

Applying Theorem~\ref{theorem:abbe-entrywise}, we conclude that with probability at least $1-(1 + 4r)n^{-3} \geq 1 - 5rn^{-3}$,
\begin{align*}
\inf_{O \in \mathcal{O}^{r \times r}} \left \Vert U O - A U^{\star} \left(\Lambda^{\star} \right)^{-1} \right\Vert_{2 \to \infty} &\leq C_0 \Big(\kappa (\kappa + \varphi(1)) (\gamma + \varphi(\gamma)) \Vert u_k^{\star} \Vert_{\infty} + \frac{\gamma}{\Delta^{\star}} \Vert A^{\star} \Vert_{2 \to \infty} \Big)\\
&\leq \frac{C(P,\rho,t,y)}{\log \log (n) \sqrt{n}}.
\end{align*}
\end{proof}
\section{Performance of Spectral Algorithms}\label{sec:spectral}
Throughout this section, we use the notation
\begin{align}\label{defn:y-value}
    y(p,q) = \frac{\log \frac{1-q}{1-p}}{ \log \frac{p}{q}} \quad \text{for}~ p\neq q.
\end{align}
This will be the choice of $y$ value for which \specone~algorithms are optimal in the cases stated in Theorem~\ref{thm:one-matrix-two-vec}~\ref{thm:one-matrix-two-vec-1}. Also, as before, we condition on a fixed value of $\sigma_0$ satisfying \eqref{eq:concentration} with $\varepsilon = n^{-1/3}$.

Recall that our spectral algorithms use the top two eigenvectors of the signed adjacency matrix/matrices. In general, the signed adjacency matrix should have two main eigenvectors which correspond (up to a potential sign flip) to the main eigenvectors of the expected adjacency matrix. However, this could run into complications if both eigenvalues are the same or one of the eigenvalues is $0$. In order to address this, we have the following eigenvalue characterization. The proof is provided in Appendix~\ref{sec:eigenvalue-properties}.
\begin{lemma} \label{edgeEigenvalue}
Let $0<p_1,p_2,q<1$ be not all the same, $\rho\in (0,1)$ and define
\begin{eq}\label{defn:2-2-matrix}
    &A' = A'(y)\\
    &:=\begin{pmatrix} p_1-y(1-p_1) & q-y(1-q) \\q-y(1-q) & p_2-y(1-p_2)\end{pmatrix} \begin{pmatrix} \rho & 0 \\0 & 1-\rho\end{pmatrix}
\end{eq}
for each $y>0$.
Then all of the following hold.\begin{enumerate}
    \item For any fixed $p_1,p_2,q,\rho\in (0,1)$,
    there exists a set $\cY$ with $|\cY| \leq 3$ such that the eigenvalues of $A'$ are distinct and nonzero for all 
    $y\notin \cY$. \label{edgeEigenvalue-1}
    
    \item If $p_1 = p$, $p_2 = q$, $p\neq q$, and $y=y(p,q)$ then the eigenvalues of $A'$ are distinct and nonzero. \label{edgeEigenvalue-2}    
    \item If $p_1=p_2 = p$, $p\neq q$, and $y=y(p,q)$ then the eigenvalues of $A'$ are distinct and nonzero if and only if $p+q\ne 1$. 
    \label{edgeEigenvalue-3}
\end{enumerate}
\end{lemma}

\subsection{One Matrix}
In order to prove Theorem \ref{thm:one-matrix-two-vec}~\ref{thm:one-matrix-two-vec-1}, we provide an algorithm which is an instance of \specone, that will succeed up to the information theoretic threshold when 
\begin{eq} \label{spec-one-success-regime}
    &\text{either }p_1 = p_2 = p , p\neq q \text{ and }p+q \neq 1 \\ &\text{or} \quad p_1= p \text{ and } p_2 = q \neq p.
\end{eq}
{To design the algorithm, we crucialy use the entrywise eigenvector bounds. Remark \ref{remark:single-eigenvector} tells us that for any pair of constants $a_1, a_2$, we have
\begin{equation}
a_1 u_1 + a_2 u_2 \approx A\left(\frac{a_1 u_1^{\star}}{\lambda_1^{\star}} + \frac{a_2 u_2^{\star}}{\lambda_2^{\star}}\right), \label{eq:algorithm-intuition}   
\end{equation}
where the approximation is in $\ell_{\infty}$, and we have ignored the sign ambiguity in $u_1$ and $u_2$ for clarity of exposition. Observe that the vector $a_1\frac{ u_1^{\star}}{\lambda_1^{\star}} + a_2\frac{ u_2^{\star}}{\lambda_2^{\star}}$ is a block vector; that is, it is of the form
\begin{equation}\label{eq:system}
\begin{cases}
\sqrt{n} \log (n)  \left(a_1\frac{ u_1^{\star}}{\lambda_1^{\star}} + a_2\frac{ u_2^{\star}}{\lambda_2^{\star}}\right)_i =  \alpha_1 & \sigma_0(i) = 1\\
\sqrt{n} \log (n) \left(a_1\frac{ u_1^{\star}}{\lambda_1^{\star}} + a_2\frac{ u_2^{\star}}{\lambda_2^{\star}}\right)_i = \alpha_2 & \sigma_0(i) = 2,\\
\end{cases}
\end{equation}
where $\alpha_1, \alpha_2$ depend on $a_1, a_2$. We see that the $v^{\text{th}}$ entry of \eqref{eq:algorithm-intuition} is equal to 
\begin{eq}\label{eq:alg-one-vertex}
&A_{v,\cdot} \cdot \left(\frac{a_1 u_1^{\star}}{\lambda_1^{\star}} + \frac{a_2 u_2^{\star}}{\lambda_2^{\star}}\right) \\
&= \alpha_1 d_{+1}(v)  - y\alpha_1 d_{-1}(v) + \alpha_2 d_{+2}(v)  - y\alpha_2 d_{-2}(v).    
\end{eq}
Since we will ultimately threshold \eqref{eq:algorithm-intuition} at $0$, we set $\alpha_1, \alpha_2$ so that \eqref{eq:alg-one-vertex} is proportional to $w^{\star T} d(v)$, where $w^{\star}$ is defined in \eqref{defn:genie-vec}. In this way, the spectral algorithm would give rise to the optimal hyperplane for separating the two communities.
It remains to find $a_1, a_2$ to satisfy~\eqref{eq:system} with the desired values of $\alpha_1, \alpha_2$. Since we do not have access to $(\lambda_1^{\star}, u_1^{\star}), (\lambda_2^{\star}, u_2^{\star})$, we form an auxiliary matrix $B$, which is essentially equivalent to $\mathbb{E}[A \mid \sigma_0]$, up to a permutation of the rows and columns. Letting $(\gamma_1, v_1), (\gamma_2, v_2)$ be the eigenpairs of $B$, we solve \eqref{eq:system} with $(\gamma_1, v_1), (\gamma_2, v_2)$ in place of $(\lambda_1^{\star}, u_1^{\star}), (\lambda_2^{\star}, u_2^{\star})$, thereby obtaining $a_1, a_2$. This strategy of solving for the weights is captured in Algorithm \ref{alg:weights-1}, with the classification algorithm given in Algorithm \ref{alg:one-matrix} below.
}

\begin{breakablealgorithm}
\caption{One-matrix community detection}\label{alg:one-matrix}
\begin{algorithmic}[1]
\Require{Parameters $t>0$, $\rho\in (0,1)$, $p_1,p_2, q\in (0,1)$ satisfying \eqref{spec-one-success-regime} and  $G \sim \CSBM_n^2(\bar{\rho}, P, t)$.
}
\Ensure{Community classification $\hsig \in \{1,2\}^n$.} 
\vspace{0.2cm}
    \State Construct an $n \times n$ matrix $A = A(G,y)$ as defined in Definition~\ref{def:signed-adj}.
    \State Find the top two eigenpairs $(\lambda_1, u_1) \text{ and }  (\lambda_2, u_2)$ of $A$. 
    \State \label{step:compute-weights}Compute $(a_1, a_2)$, the weights produced by Algorithm \ref{alg:weights-1}.
    \State Let $U = \{ s_1 a_1 u_1 + s_2 a_2 u_2 : s_1, s_2 \in \{\pm1\}\}$. For each $u \in U$, let $\hat{\sigma}( \cdot ;u) = 1+ (1+\sgn(u))/2$.
    \State Return $\hsig = \argmax_{u\in U} \PR(G \mid \hsig(\cdot;u))$.
\end{algorithmic}
\end{breakablealgorithm}

\begin{breakablealgorithm}
\caption{Find weights (one matrix)}\label{alg:weights-1}
\begin{algorithmic}[1]
\Require{Parameters $t>0$, $\rho\in (0,1)$, $p_1,p_2, q\in (0,1)$ satisfying \eqref{spec-one-success-regime}.
}
\Ensure{Weights $(a_1, a_2)$}
\vspace{0.2cm}
    \State Let $\cV_1:= \{i:i\leq \rho n\}$ and define $B$ to be the symmetric block matrix where $B_{ij}$ is $ \frac{t \log n}{n}[p_1 - y(p_1,q)(1-p_1)]$ if $i,j\in \cV_1$, $\frac{t \log n}{n}[p_2 - y(p_2,q)(1-p_2)]$ if $i,j\notin \cV_1$, and $\frac{t\log n}{n}[q - y(p_1,q)(1-q)]$ if $i\in \cV_1,j\notin \cV_1$ or $i\notin \cV_1,j\in \cV_1$. 
    Let the eigenpairs of $B$ be denoted $(\gamma_1, v_1), (\gamma_2, v_2)$.
    \State Set $\alpha_1 = \log \frac{p}{q}$. 
    If $p_1 = p_2 = p$, set $\alpha_2 = -\alpha_1$. Otherwise ($p_2 = q$), set $\alpha_2 = 0$. Let $z$ be a block vector with $z_i = \alpha_1$ if $i \in \cV_1$ and $z_i = \alpha_2$ if $i \not \in \cV_1$.
    \State Return $(a_1, a_2)$ satisfying 
    \begin{eq} \label{constant-choosing-identity}
        \sqrt{n} \log n \left(a_1 \frac{v_1}{\gamma_1} + a_2 \frac{v_2}{\gamma_2}\right) = z.
    \end{eq}
\end{algorithmic}
\end{breakablealgorithm}
It is worthwhile to note that finding weights in Algorithm~\ref{alg:weights-1} does not require any information about~$\true$.

\begin{proof}[Proof of Theorem \ref{thm:one-matrix-two-vec}]
Let $n_i$ be the number of vertices in community $i$ for $i = 1,2$.
Throughout the proof, we will condition on $\sigma_0$ satisfying $\left| n_i - \rho_i n \right| \leq n^{2/3}$. This event has probability $1-o(1)$ as shown earlier in \eqref{eq:concentration}. 

First, suppose that \eqref{spec-one-success-regime} holds. We will first prove Theorem~\ref{thm:one-matrix-two-vec}~\ref{thm:one-matrix-two-vec-1} by showing Algorithm~\ref{alg:one-matrix} succeeds up to the information theoretic limit. 
Let $A = A(G, y)$ with $y = y(p_1,q)$, and define  $A^{\star} = \mathbb{E}[A]$. Let $(\lambda_i,u_i)$ and $(\lambda_i^\star,u_i^\star)$ denote the $i$-th largest eigenpair of $A$ and $A^\star$ respectively. 
We claim that 
\begin{eq}\label{eval-sep-A-star}
    &\lambda_1^\star = (1+o(1)) \nu_1\log n, \quad \lambda_2^\star = (1+o(1)) \nu_2\log n \\
    &\text{with } \nu_1\neq \nu_2, \nu_1,\nu_2\neq 0.
\end{eq}
Indeed, consider the matrix $B$ defined in Step 1 of Algorithm~\ref{alg:weights-1}, whose eigenvalues are $t \log n$ times the corresponding eigenvalues the matrix $A'$ defined in \eqref{defn:2-2-matrix}. Under the conditions of Theorem~\ref{thm:one-matrix-two-vec}~\ref{thm:one-matrix-two-vec-1}, Lemma \ref{edgeEigenvalue} (Parts \ref{edgeEigenvalue-2} and \ref{edgeEigenvalue-3}) shows that the non-zero eigenvalues of $B$ are $\nu_1\log n$ and $\nu_2 \log n$ with $\nu_1\neq \nu_2$. Next, suppose $O$ is the permutation matrix such that, in $OA^\star O^T$, the rows and columns corresponding to vertices in community 1 appear before those in community 2. By Weyl's theorem, the top two eigenvalues of $OA^\star O^T$ are within $1+o(1)$ multiplicative factor of those of $B$. Since $O$ is an orthogonal matrix, \eqref{eval-sep-A-star} follows immediately. 

Using \eqref{eval-sep-A-star}, we can apply Corollary~\ref{corollary:entrywise} and conclude that, with probability $1-O(n^{-3})$,
\begin{align*}
&\left\Vert s_1 u_1 -  \frac{A u_1^\star}{\lambda_1^\star} \right\Vert_{\infty} \leq \frac{C}{\sqrt{n} \log \log n}\\
\text{and}~~~&\left \Vert s_2 u_2 -  \frac{A u_2^\star}{\lambda_2^\star} \right \Vert_{\infty} \leq \frac{C}{\sqrt{n} \log \log n},    
\end{align*}
for some $s_1, s_2 \in \{-1,1\}$ and some constant $C>0$. Consequently, for any $a_1, a_2\in \R$, with probability $1-o(1)$,
\begin{eq}\label{eq:vec-approximation-in-proof}
    &\left \Vert s_1a_1u_1+s_2a_2u_2 - A\left(\frac{ a_1}{\lambda_1^\star} u_1^\star+\frac{ a_2}{\lambda_2^\star} u_2^\star\right) \right \Vert_{\infty} \\
    &\leq \frac{C(|a_1| + |a_2|)}{\sqrt{n} \log \log n}.
\end{eq}
In Step \ref{step:compute-weights} of Algorithm \ref{alg:one-matrix}, we
pick $(a_1,a_2)$ according to Algorithm~\ref{alg:weights-1}. 
Let $\cV_1':=\{i:i\leq n_1(\true)\}$ and define $B',v_1',v_2'$ similarly as $B,v_1,v_2$ in Algorithm~\ref{alg:weights-1} by replacing $\cV_1$ by $\cV_1'$. 
For $l=1,2$, note that $v_l$ takes some value $\frac{\zeta_{l1}}{\sqrt{n}}$ on $\cV_1$ and $\frac{\zeta_{l2}}{\sqrt{n}}$ on $\cV_1^c$ for constants $\zeta_{l1},\zeta_{l2}$. 
Using identical steps as \cite[Lemma 5.3]{Dhara2022a}, we can argue that $v_l'$ also takes value $(1+o(1))\frac{\zeta_{l1}}{\sqrt{n}}$ on $\cV_1'$ and $(1+o(1))\frac{\zeta_{l2}}{\sqrt{n}}$ on~$(\cV_1')^c$. Therefore, 
\begin{align*}
    \sqrt{n} \log n \left(a_1 \frac{v_1'}{\gamma_1} + a_2 \frac{v_2'}{\gamma_2}\right) = \tilde{z},
\end{align*}
where $\tilde{z}$ is a block vector taking values $(1+o(1))\alpha_1$ on $\cV_1'$ and $(1+o(1)) \alpha_2$ on~$(\cV_1')^c$.
Next, note that the matrix $A^\star$ can be obtained from $B'$ by jointly permuting the row and column labels and then setting the diagonal entries to be zero. 
Also, noting that $\lambda_l^\star = (1+o(1)) \gamma_l$, we can ensure that 
\begin{eq}\label{form-lin-comb}
    \sqrt{n} \log(n) \left(a_1 \frac{u_1^{\star}}{\lambda_1^{\star}} + a_2 \frac{u_2^{\star}}{\lambda_2^{\star}}\right) = z^{\star},
\end{eq}
where $z^{\star}$ is a block vector taking value $(1+o(1))\alpha_1$ on $V_1:=\{v:\true(v) = +1\}$ and $(1+o(1)) \alpha_2$ on~$V_2:=\{v:\true(v) = -1\}$.
Let 
\[\tau = A \left(a_1 \frac{u_1^{\star}}{\lambda_1^{\star}} + a_2 \frac{u_2^{\star}}{\lambda_2^{\star}}  \right).\]
 Then, $\sqrt{n} \log(n) \tau =(1+o(1)) A z^{\star}$. 
By \eqref{form-lin-comb}, with probability $1-o(1)$, for each $v \in [n]$, 
\begin{eq}\label{eq:tau}
\sqrt{n} \log(n) \tau_v &= \alpha_1 d_{+1}(v) - y \alpha_1 d_{-1}(v) \\
& \quad +  \alpha_2 d_{+2}(v)-   y \alpha_2 d_{-2}(v) + o(\log n), 
\end{eq}
where $(d_{+1}(v),d_{-1}(v),d_{+2}(v),d_{-2}(v))$ denotes the degree profile of $v$.
Also, in this case, note that $w^\star$ in \eqref{defn:genie-vec} simplifies to
\begin{align} \label{w-star-two-com}
    w^\star = \bigg(\log \frac{p_1}{q}, \log \frac{1-p_1}{1-q}, \log \frac{q}{p_2}, \log \frac{1-q}{1-p_2}\bigg).
\end{align}
In order to apply Proposition \ref{prop:seperation-verification}, we need to ensure that the coefficients of \eqref{eq:tau} coincide with $w^{\star}$ up to a scalar factor. There are two cases to consider. First, suppose $p_1 = p_2 = p$, and $p \neq q$ (where we rule out the case $\{p+q = 1, \rho \neq 1/2\}$). Recalling that $y = y(p,q)$, we obtain \begin{align} \label{sym-w-star}
    w^\star =  (1,-y,-1,y)\log \left(\frac{p}{q}\right). 
\end{align}
Comparing \eqref{eq:tau} and \eqref{sym-w-star}, we see that the choice $\alpha_1 = \log\frac{p}{q}$ and $\alpha_2 = -\alpha_1$ equates the coefficients of the leading terms of \eqref{eq:tau} with the entries of \eqref{sym-w-star}. These are the values of $(\alpha_1, \alpha_2)$  chosen in Algorithm~\ref{alg:weights-1}~Step~2. {(Note that any choice of the form $(\alpha_1, \alpha_2) = c(1,-1)$ would lead to $\sqrt{n} \log(n) \tau_v - o(\log(n)) \propto \langle w^{\star},d(v)\rangle$.)}

Next, suppose $p_2 = q$ and recall that $p_1 \neq q$. By our choice of $y = y(p_1,q)$, we have that 
\begin{align*}
    w^\star =  (1,-y,0,0)\log \left(\frac{p}{q}\right). 
\end{align*} 
In this case, we need $\alpha_1 = \log \frac{p_1}{q} $ and $\alpha_2 = 0$, which is also the case by our choice in Algorithm~\ref{alg:weights-1}~Step~2. 

Thus, in both cases, our choices of $(\alpha_1, \alpha_2)$ yield that, with probability $1-o(1)$, $\sqrt{n} \log(n) \tau_v = \langle w^{\star}, d(v)\rangle + o(\log n)$ for each $v \in [n]$. By Proposition \ref{prop:seperation-verification}, we conclude that for some $\varepsilon > 0$, 
\begin{align*}
&\sqrt{n} \log(n) \min_{v \in V_1} \tau_v \geq \frac{1}{2} \varepsilon \log n \\
\text{and}\quad &\sqrt{n} \log(n) \max_{v \in V_2} \tau_v \leq  -\frac{1}{2}\varepsilon \log n
\end{align*}
with probability $1-o(1)$, and consequently
\begin{align*}
 \min_{v \in V_1} \tau_v \geq \frac{\varepsilon}{2\sqrt{n}} \quad\text{and}\quad   \max_{v \in V_2} \tau_v \leq  -\frac{\varepsilon}{2\sqrt{n}}.
\end{align*}
Finally, since $\frac{C}{\sqrt{n} \log \log (n)} = o \big(\frac{1}{\sqrt{n}} \big)$, we conclude with probability $1-o(1)$,
\begin{align*}
&\min_{v \in V_1}\left(s_1 a_1 u_1 + s_2 a_2 u_2  \right)_v > \frac{\varepsilon}{3\sqrt{n}}\\
\text{and}\quad & \max_{{v \in V_2}}\left(s_1 a_1 u_1 + s_2 a_2 u_2  \right)_v < - \frac{\varepsilon}{3\sqrt{n}}.    
\end{align*}

Therefore, thresholding the vector $s_1 a_1 u_1 + s_2 a_2 u_2$ at zero correctly identifies the communities with high probability. In other words, $\sgn(s_1 a_1 u_1 + s_2 a_2 u_2)$ coincides with the MAP estimator. While $s_1, s_2$ are unknown, the final step of Algorithm \ref{alg:classify} chooses the best one among the four candidate community partitions arising from the four possible sign combinations. By Theorem \ref{theorem:achievability-general}, we know that the MAP estimator is the unique maximizer of the posterior probability. Therefore, the spectral algorithm will identify the correct candidate. This completes the proof of Theorem~\ref{thm:one-matrix-two-vec}~\ref{thm:one-matrix-two-vec-1}.

To prove Theorem~\ref{thm:one-matrix-two-vec}~\ref{thm:one-matrix-two-vec-2}, let $p_1,p_2,q$ be distinct. 
Notice that \eqref{eq:tau} would hold for any $\alpha_1,\alpha_2$ and the corresponding choices of $a_1,a_2$. The particular choice of $\alpha_1,\alpha_2$ was only needed after \eqref{eq:tau} to compare it with $w^\star$. 
By Proposition~\ref{prop:seperation-verification} and \eqref{eq:tau}, in order for \specone~algorithms to be successful, we must have $w^\star = (\alpha_1, -y \alpha_1, \alpha_2, -y \alpha_2)$.
Suppose for the sake of contradiction that $w^\star = (\alpha_1, -y \alpha_1, \alpha_2, -y \alpha_2)$ for some $\alpha_1, \alpha_2$. Since all the entries of $w^{\star}$ are nonzero, we know that $\alpha_1, \alpha_2 \neq 0$. By taking coordinate ratios, we have that 
\begin{eq}\label{y-contradiction}
    &y = \frac{\log \frac{1-q}{1-p_1}}{\log \frac{p_1}{q}} = y(p_1,q)\\
    \text{and} \quad &y = \frac{\log \frac{1-q}{1-p_2}}{\log \frac{p_2}{q}} = y(p_2,q). 
\end{eq}
Now, we claim that for any fixed $q\in (0,1)$, the function $y(p,q)$ is strictly increasing. 
Indeed, 
\begin{eq}\label{y-derivative}
    \frac{\partial}{\partial p} y(p,q) &= \frac{\log \frac{p}{q} \times \frac{1}{1-p} - \log \frac{1-q}{1-p} \times \frac{1}{p} }{\log^2\frac{p}{q}}  \\
    &= \frac{1}{p\log \frac{p}{q}} \bigg(\frac{p}{1-p} - y(p,q)\bigg). 
\end{eq}
Using the fact that $1-\frac{1}{x} < \log x < x - 1$ for any $x>0$, 
\begin{align}
    \text{ for any }p>q: ~ y(p,q) &=  \frac{\log \frac{1-q}{1-p}}{\log \frac{p}{q}} < \frac{\frac{1-q}{1-p}-1}{1 - \frac{q}{p}} = \frac{p}{1-p}\label{eq:y-inequality-1}\\
    \text{ for any }p<q: ~ y(p,q) &=  \frac{\log \frac{1-q}{1-p}}{\log \frac{p}{q}} = \frac{\log \frac{1-p}{1-q}}{\log \frac{q}{p}} \nonumber \\
    &> \frac{1-\frac{1-q}{1-p}}{ \frac{q}{p} - 1} = \frac{p}{1-p}. \label{eq:y-inequality-2}
\end{align}
Therefore, $\frac{\partial}{\partial p} y(p,q) > 0$ for any $p\in (0,1)$, which proves that $y(p,q)$ is strictly increasing. However, $p_1\neq p_2$ and therefore $y(p_1,q) \neq y(p_2,q)$. Thus, \eqref{y-contradiction} leads to a contradiction. In other words,
it is not possible to choose $\alpha_1, \alpha_2$ so that $w^\star = (\alpha_1, -y \alpha_1, \alpha_2, -y \alpha_2)$. The proof then follows by applying Proposition~\ref{prop:seperation-verification}~\ref{prop:seperation-verification-2}.
\end{proof}

\begin{remark}\label{rem:arbitrary-encoding} 
Instead of using the encoding $\{1,-y,0\}$ for present, absent and censored edges, we could have instead used a more general encoding of the form $\{c_1,-yc_2,c_3\}$. In that case, the entrywise approximation would still hold.  One could go though the same steps to show that the decision rule for the spectral algorithm would again be asymptotically based on determining whether some linear expression such as \eqref{eq:tau} is above or below a certain threshold $T$. 
Thus, for $p_1,p_2,q$ which are distinct, an identical argument shows that spectral algorithms with more general encoding also do not succeed sufficiently close to~$t_c$.

\end{remark}

\subsection{Two Matrices} In this section, we will prove Theorem~\ref{thm:two-matrix-two-vec}. Let us start by describing the algorithm that always succeeds up to the information theoretic threshold in the two community case. {The design of the algorithm is analogous to the design of Algorithm \ref{alg:one-matrix}.}
\begin{breakablealgorithm}
\caption{Two-matrix community detection for two communities}\label{alg:two-matrix-a}
\begin{algorithmic}[1]
\Require{Parameters $t>0$, $\rho, p_1, p_2, q\in (0,1)$ such that $|\{p_1, p_2, q\}|\geq 2$, and $G \sim \CSBM_n^2(\bar{\rho}, P, t)$.}
\Ensure{Community classification $\hsig \in \{1,2\}^n$.} 
\vspace{0.2cm}
    \State Fix $y,\tilde{y}\notin \cY$ where $\cY$ is given by Lemma~\ref{edgeEigenvalue}~Part~\eqref{edgeEigenvalue-1}. Construct two $n \times n$ matrices $A = A(G,y)$, $\tilde{A} = A(G,\tilde{y})$ as defined in Definition~\ref{def:signed-adj}.
    
    \State Find the top  two eigenpairs of $A$ and $\tilde{A}$ {and respectively denote them} $((\lambda_l, u_l))_{l=1,2}$ and $((\tilde{\lambda}_l, \tilde{u}_l))_{l=1,2}$.
    \State Use Algorithm \ref{alg:weights-2} on input $\left(t,\rho, \begin{pmatrix}p_1 & q \\ q & p_2\end{pmatrix}, y, \tilde{y}\right)$ to compute the weights  $(c_1, c_2, \tilde{c}_1, \tilde{c}_2)$.
    \State Let $U = \{ s_1 c_1 u_1 + s_2 c_2 u_2 + \tilde{s}_1 \tilde{c}_1 \tilde{u}_1 + \tilde{s}_2 \tilde{c}_2 \tilde{u}_2 : s_1, s_2, \tilde{s}_1, \tilde{s}_2 \in \{\pm1\}\}.$ For each $u \in U$, let $\hat{\sigma}( \cdot ;u) = 1+(1+\sgn(u))/2$.
    \State Return $\hsig = \argmax_{u\in U} \PR(G \mid \hsig(\cdot;u))$.
\end{algorithmic}
\end{breakablealgorithm}

\begin{breakablealgorithm}
\caption{Find weights (two matrices, two communities)}\label{alg:weights-2}
\begin{algorithmic}[1]
\Require{Parameters $t>0$, $\rho, p_1, p_2, q\in (0,1)$ such that $|\{p_1, p_2, q\}|\geq 2$, and $y,\tilde{y}\notin \cY$, $y\neq \tilde{y}$ where $\cY$ is given by Lemma~\ref{edgeEigenvalue}~Part~\eqref{edgeEigenvalue-1}.}
\Ensure{Weights $(c_1, c_2, \tilde{c}_1, \tilde{c}_2)$}
\vspace{0.2cm}
 \State Let $\cV_1:= \{i:i\leq \rho n\}$ and define $B$ to be the symmetric block matrix where $B_{ij}$ is $ \frac{t \log n}{n}[p_1 - y(p_1,q)(1-p_1)]$ if $i,j\in \cV_1$, $\frac{t \log n}{n}[p_2 - y(p_2,q)(1-p_2)]$ if $i,j\notin \cV_1$, and $\frac{t\log n}{n}[q - y(p_1,q)(1-q)]$ if $i\in \cV_1,j\notin \cV_1$ or $i\notin \cV_1,j\in \cV_1$. 
Define $\tilde{B}$ similarly by replacing $y$ by $\tilde{y}$. Let the eigenpairs of $B$ and $\tilde{B}$ be $((\gamma_l, v_l))_{l=1,1}$ and $((\tilde{\gamma}_l, \tilde{v}_l))_{l=1,2}$, respectively. 
\item Solve the following system for $\alpha_1, \alpha_2, \tilde{\alpha}_1, \tilde{\alpha}_2$:
\begin{eq}\label{coeffs-2matrix}
    &\alpha_1 + \tilde{\alpha}_1 = \log \frac{p_1}{q}, \quad -y \alpha_1 - \tilde{y} \tilde{\alpha}_1 = \log \frac{1-p_1}{1-q}, \\
    &\alpha_2 + \tilde{\alpha}_2 = \log \frac{q}{p_2} \quad -y \alpha_2 - \tilde{y} \tilde{\alpha}_2 = \log \frac{1-q}{1-p_2}.
\end{eq}
Let $z$ be a block vector with $z_i = \alpha_1$ for $i \in \cV_1$ and $z_i = \alpha_2$ for $i \not \in \cV_1$.
Define $\tilde{z}$ similarly by replacing $(\alpha_1,\alpha_2)$ by $(\tilde{\alpha}_1,\tilde{\alpha}_2)$. 
\State Return $(c_1, c_2, \tilde{c}_1, \tilde{c}_2)$ satisfying 
\begin{eq}\label{two-z-choices}
&\sqrt{n} \log n \left(c_1 \frac{v_1}{\gamma_1} + c_2 \frac{v_2}{\gamma_2}\right) = z \\
\text{and} \quad
&\sqrt{n} \log n \left(\tilde{c}_1 \frac{\tilde{v}_1}{\tilde{\gamma}_1} + \tilde{c}_2 \frac{\tilde{v}_2}{\tilde{\gamma}_2}\right) = \tilde{z}.
\end{eq}
\end{algorithmic}
\end{breakablealgorithm}

\begin{proof}[Proof of Theorem \ref{thm:two-matrix-two-vec}]
As in the proof of Theorem~\ref{thm:one-matrix-two-vec}, we condition on $\sigma_0$ satisfying $\left| n_i - \rho_i n \right| \leq n^{2/3}$. 
Fix  $y,\tilde{y}\notin \cY$, $y\neq \tilde{y}$ where $\cY$ is given by Lemma~\ref{edgeEigenvalue}~Part~\eqref{edgeEigenvalue-1}. Recall all the notation in Algorithms~\ref{alg:two-matrix-a},~\ref{alg:weights-2}. Let $A^\star:= \E[A]$ and $\tilde{A}^\star:= \E[\tilde{A}]$, and let $((\lambda_l^\star, u_l^\star))_{l=1,1}$, and $(\tilde{\lambda}_l^\star, \tilde{u}_l^\star))_{l=1,2}$ be the top eigenpairs of the corresponding matrices. Applying Corollary~\ref{corollary:entrywise}, we have that with probability $1-o(1)$
\begin{align*}
&\Bigg\Vert s_1 c_1 u_1 + s_2 c_2 u_2 + \tilde{s}_1 \tilde{c}_1 \tilde{u}_1 + \tilde{s}_2 \tilde{c}_2 \tilde{u}_2  \\
&\qquad - A \bigg(c_1 \frac{u_1^{\star}}{\lambda_1^{\star}} + c_2 \frac{u_2^{\star}}{\lambda_2^{\star}}  \bigg) -  \tilde{A}\bigg(\tilde{c}_1 \frac{\tilde{u}_1^{\star}}{\tilde{\lambda}_1^{\star}} + \tilde{c}_2 \frac{\tilde{u}_2^{\star}}{\tilde{\lambda}_2^{\star}}    \bigg) \Bigg\Vert_{\infty} \\
&\leq \frac{C\left(|c_1| + |c_2| + |\tilde{c}_1| + |\tilde{c}_2| \right)}{\sqrt{n} \log \log n }, 
\end{align*}
for some $s_1, s_2, \tilde{s}_1, \tilde{s}_2 \in \{\pm1\}$. 
Let \[\tau = A \bigg(c_1 \frac{u_1^{\star}}{\lambda_1^{\star}} + c_2 \frac{u_2^{\star}}{\lambda_2^{\star}}  \bigg) +  \tilde{A}\bigg(\tilde{c}_1 \frac{\tilde{u}_1^{\star}}{\tilde{\lambda}_1^{\star}} + \tilde{c}_2 \frac{\tilde{u}_2^{\star}}{\tilde{\lambda}_2^{\star}}    \bigg).\]
Using \eqref{two-z-choices}, we can repeat the arguments above \eqref{eq:tau}, to show that 
$\sqrt{n} \log(n) \tau = (1+o(1)) z^\star$, where $z_{i}^\star$ is $\alpha_1+\tilde{\alpha}_1$ on $V_1:= \{u:\true(u) = +1\}$ and $\alpha_2+\tilde{\alpha}_2$ on $V_2:= \{u:\true(u) = -1\}$. Consequently, with probability $1-o(1)$, for each $v \in [n]$, 
\begin{align*}
\sqrt{n} \log(n) \tau_v &= d_{+1}(v) \left(\alpha_1 + \tilde{\alpha}_1\right) -  d_{-1}(v) \left(y\alpha_1 + \tilde{y}\tilde{\alpha}_1 \right) \\
&+ d_{+2}(v)  \left(\alpha_2 + \tilde{\alpha}_2 \right) -  d_{-2}(v)  \left(y\alpha_2 + \tilde{y}\tilde{\alpha}_2 \right) \\
&+ o(\log n).
\end{align*}
{The choice of constants in \eqref{coeffs-2matrix} is such that $\sqrt{n} \log(n) \tau_v = \langle w^\star,d(v)\rangle + o(\log n)$, where $w^\star$ is given by~\eqref{defn:genie-vec}. Thus,} by Proposition \ref{prop:seperation-verification} {Part 1}, there exists some $\varepsilon > 0$ such that with probability $1 - o(1)$ 
\begin{align*}
&\sqrt{n} \log(n) \min_{v \in V_1} \tau_v \geq \frac{1}{2} \varepsilon \log n \\
\text{ and } &\sqrt{n} \log(n) \max_{v \in V_2} \tau_v \leq  -\frac{1}{2}\varepsilon \log n
\end{align*}
with probability $1-o(1)$. 
The rest of the proof is identical to the final part of the argument in the proof of Theorem~\ref{thm:one-matrix-two-vec}~\ref{thm:one-matrix-two-vec-1}. 
\end{proof}

\section{More than two communities}
In this section, we will prove Theorem~\ref{thm:two-matrix-k-communities}. Similar to Lemma~\ref{edgeEigenvalue}, we need the following, whose proof is provided in Appendix~\ref{sec:eigenvalue-properties}.

 \begin{lemma} \label{edgeEigenvalue-general-k}
Let $\rho\in (0,1)^k$, and $P\in (0,1)^{k\times k}$ be a symmetric matrix. For any $y>0$, let $P^{\sss (y)}$ be the matrix such that $P^{\sss (y)}_{ij}=\rho_j(P_{ij}-y(1-P_{ij}))$ for all $i,j$. Then, either (1) $P^{\sss (y)}$ has a zero eigenvalue for all $y$ or (2) $P^{\sss (y)}$ has repeated eigenvalues for all $y$ or (3) there is a finite set $\cY$ such that $P^{\sss (y)}$  has distinct nonzero eigenvalues for all $y\not\in\cY$.

Consequently, if $P^{\sss (0)}:= P\cdot\diag(\rho)$ has $k$ distinct, non-zero eigenvalues, then~(3) holds.
\end{lemma}

Let us describe the algorithm that always succeeds up to the information theoretic threshold in the $k$-community case. {Again, the design of the algorithm is analogous to the single matrix case.}
\begin{breakablealgorithm}
\caption{Two-matrix community detection for general $k\geq 3$ communities}\label{alg:two-matrix-b}
\begin{algorithmic}[1]
\Require{Parameters $t>0$, $\rho \in (0,1)^k$ such that $\sum_i\rho_i = 1$, a symmetric matrix $P \in (0,1)^{k \times k}$, 
and also $G \sim \CSBM_n^k(\rho, P, t)$. }
\Ensure{Community classification $\hsig\in [k]^n$.}
\vspace{0.2cm}
\State Fix $y,\tilde{y}\notin \cY$ where $\cY$ is given by Lemma~\ref{edgeEigenvalue-general-k}. Construct two $n \times n$ matrices $A = A(G,y)$, $\tilde{A} = A(G,\tilde{y})$ as defined in Definition~\ref{def:signed-adj}.

\State Find the top $k$ eigenpairs of $A$ and $\tilde{A}$, respectively denoting them $((\lambda_l, u_l))_{l\in [k]}$ and $((\tilde{\lambda}_l, \tilde{u}_l))_{l \in [k]}$. Let $U$ (respectively $\tilde{U}$) be the $n \times k$ matrix whose $i$-th column is $u_i$ (respectively $\tilde{u}_i$). 
\State Use Algorithm \ref{alg:weights-3} on input $\left(t,\rho, P, y, \tilde{y}\right)$ to compute the weight vectors $\left(c_i, \tilde{c}_i \right)_{i\in [k]}$.
\State For $s \in \{\pm 1\}^{k}$, let $D^{\sss (s)}:= \diag(s)$. For any $s,\tilde{s}\in \{\pm 1\}^{k}$, construct the estimator 
\begin{eq}\label{estimator-k-com}
    &\hat{\sigma}(v; s, \tilde{s}) = \argmax_{i \in [k]} \Big\{\big(U  D^{\sss (s)}  c_i\big)_v + \big(\tilde{U} D^{\sss (\tilde{s})} \tilde{c}_i\big)_v\Big\} \\
    &\text{ for each }v\in [n].
\end{eq}
\State Return $\hsig = \argmax_{s, \tilde{s}\in \{\pm 1\}^k} \PR(G \mid \hat{\sigma}(\cdot; s, \tilde{s}))$. \label{step:find-signs}

\end{algorithmic}
\end{breakablealgorithm}

\begin{breakablealgorithm}
\caption{Find weights (Two matrices, $k \geq 3$ communities)}\label{alg:weights-3}
\begin{algorithmic}[1]
\Require{Parameters $t>0$, $\rho \in (0,1)^k$ such that $\sum_i\rho_i = 1$, a symmetric matrix $P \in (0,1)^{k \times k}$, and $y,\tilde{y}\notin \cY$ where $\cY$ is given by Lemma~\ref{edgeEigenvalue-general-k}.
}
\Ensure{Weight vectors $\left(c_i, \tilde{c}_i \right)_{i=1}^k \subset \R^k$.}
\vspace{0.2cm}
\State For $k\geq 1$, let $\cV_k:= \{i: n \sum_{j=0}^{k-1} \rho_j \leq i\leq  n \sum_{j=1}^{k} \rho_j \}$ with $\rho_0 =0$. 
Define $B$ to be the symmetric block matrix where $B_{uv} =  \frac{t \log n}{n}[P_{ij} - y(1-P_{ij})]$ if $u\in \cV_i$ and $v\in \cV_j$. Define $\tilde{B}$ similarly by replacing $y$ by $\tilde{y}$. Let the top $k$ eigenpairs of $B$ and $\tilde{B}$ be $((\gamma_i, v_i))_{i\in [k]}$, and $((\tilde{\gamma}_i, \tilde{v}_i))_{i\in [k]}$. 
Let $V$ (respectively $\tilde{V}$) be the $n \times k$ matrix whose $i$th column is $\nicefrac{v_i}{\gamma_i}$ (respectively $\nicefrac{\tilde{v}_i}{\tilde{\gamma}_i}$). 
\State Solve the following system for $\{\alpha_{ri}\}_{r,i \in [k]}$, $\{\tilde{\alpha}_{ri}\}_{r,i \in [k]}$:
\begin{eq}\label{z-star-gen-k}
    \alpha_{ri} + \tilde{\alpha}_{ri} &= \log \left(P_{ri} \right),\\
    -y\alpha_{ri} -\tilde{y} \tilde{\alpha}_{ri} &= \log \left(1-P_{ri} \right), \quad \forall r, i \in [k].
\end{eq}
For $i \in [k]$, 
let $z_i$ (respectively $\tilde{z}_i$) be the block vector with $z_{iv} = \alpha_{ri}$ (respectively $\tilde{z}_{iv} = \tilde{\alpha}_{ri}$) when $v\in \cV_r$. 
    \State Return $\left(c_i, \tilde{c}_i\right)_{i=1}^k$ solving 
    \begin{eq} \label{z-star-gen-k-solution}
        &\sqrt{n} \log(n) V c_i = z_i \\
        \text{and} \quad &\sqrt{n} \log(n) \tilde{V} \tilde{c}_i = \tilde{z}_i \quad \text{for all }i\in [k].
    \end{eq}

\end{algorithmic}
\end{breakablealgorithm}

\begin{proof}[Proof of Theorem \ref{thm:two-matrix-k-communities}]
The argument is identical to the proof of Theorem~\ref{thm:two-matrix-two-vec}. 
We skip redoing all the details for general $k\geq 3$ and instead give an overview of the steps.

Indeed, since $P \cdot \diag(\rho)$ has $k$ distinct, non-zero eigenvalues by our assumption, Lemma~\ref{edgeEigenvalue-general-k} implies that the eigenvalues of $\E[A(G,y)]$ are also distinct for sufficiently large $n$. 
Applying the entrywise bounds for the eigenvectors in Corollary~\ref{corollary:entrywise},
holds for general $k$. 
{The parameters \eqref{z-star-gen-k} and \eqref{z-star-gen-k-solution} are chosen in such a way so that for some $s, \tilde{s} \in \{\pm 1\}^k$, for each community $i \in [k]$, the associated approximating vector $\tau^{(i)}$ satisfies
\begin{align*}
&\sqrt{n} \log(n) \left(UD^{(s)}c_i + \tilde{U}D^{(\tilde{s})}\tilde{c}_i\right)_v =: \sqrt{n} \log(n) \tau^{(i)}_v \\
&= \big(\log(P_{ri}), \log(1-P_{ri}) \big)_{r\in [k]}\cdot d(v) + o(\log n)   
\end{align*}
with probability $1-o(1)$, for all $v\in [n]$.}
The estimator described by \eqref{estimator-k-com} is constructed so that for some $s, \tilde{s} \in \{\pm 1\}^k$, we have
\[\hat{\sigma}(v; s, \tilde{s}) = \argmax_{i \in [k]} \tau_v^{\sss {(i)}}.\] 
Corollary~\ref{corollary:max} implies that for this pair $(s, \tilde{s})$, we have $\hat{\sigma}(v; s, \tilde{s}) = \sigma_0(v)$ for all $v$ with high probability. Finally, the correct pair $s, \tilde{s}$ is chosen in Step \ref{step:find-signs}, by again appealing to statistical achievability (Theorem \ref{theorem:achievability-general}).
\end{proof}

\begin{remark}
 We can simplify the algorithms by taking $A$ and $\tilde{A}$ without any ternary encoding if both $P\cdot \diag(\rho)$ and $(J-P)\cdot \diag(\rho)$ have $k$ distinct, non-zero eigenvalues.
Indeed, define $A,\tilde{A}$ 
\begin{align*}
    &A_{ij} = 
    \begin{cases}
    1 & \quad \text{if }\{i,j\}\text{ is present}\\
    0 & \quad \text{if }\{i,j\}\text{ is absent or censored}
    \end{cases}\\ 
    \text{and}\quad& 
    \tilde{A}_{ij} = 
    \begin{cases}
    1 & \quad \text{if }\{i,j\}\text{ is absent}\\
    0 & \quad \text{if }\{i,j\}\text{ is present or censored.}
    \end{cases}
\end{align*}
We can simply set $\alpha_{ri} = \log(P_{ri})$ and $\tilde{\alpha}_{ri} = \log(1-P_{ri})$, and choose $c_i,\tilde{c}_i$ according to \eqref{z-star-gen-k-solution}. With this choice, the estimator in \eqref{estimator-k-com} (optimized over the signs as in Algorithm~\ref{alg:two-matrix-b}~Step 5) achieves exact recovery up to the information theoretic threshold.

Of course, such a simplification might not be possible for many possible choices of parameters. 
For example, in the two community case, we can take $\rho = 1/2$, and $p_1,p_2,q$ such that $p_1p_2- q^2 \neq 0$ but $(1-p_1)(1-p_2)- (1-q)^2 = 0$. 
One such choice is $p_1 = \frac{23}{25}, p_2 = \frac{17}{25}, q= \frac{3}{5}$.  
\end{remark}


%

\bibliography{references}
\bibliographystyle{abbrv}

\appendix

\section{Proof of Poisson approximation}
\label{sec:appendix-poisson}

\begin{proof}\emph{(Proof of Lemma \ref{fact:stirling}).}
Observe that $(D_{a,r},D_{b,r})$ are independent over $r$ as they depend on disjoint sets of independent random variables. Thus,
\begin{eq}\label{exact-deg=prob}
&\PR(D = d) = \prod_{r\in [n] \setminus \{i\} } \frac{|S_r|!}{d_{1,r}!d_{2,r}!(|S_r| -d_{1,r}-d_{2,r})!}\\
& \quad \times (\alpha \psi_r)^{d_{1,r}}  (\alpha (1-\psi_r))^{d_{2,r}} (1-\alpha)^{|S_r| -d_{1,r}-d_{2,r}} \\
& \quad \times  \frac{(|S_i| - |V|)!}{d_{1,i}!d_{2,i}!(|S_i| - |V| -d_{1,i}-d_{2,i})!}\\
& \quad \times (\alpha \psi_i)^{d_{1,i}}  (\alpha (1-\psi_i))^{d_{2,i}} (1-\alpha)^{|S_i| - |V| -d_{1,i}-d_{2,i}}. 
\end{eq}
We use Stirling's approximation and the fact that $1-\e^{-x} \asymp x$ as $x\to 0$. 
Thus, using the assumptions on $d$, for each $r\neq i$, we have 
\begin{align*}
    &\frac{|S_r|!}{(|S_r|-d_{1,r} - d_{2,r})!} \\
    &\asymp \frac{\e^{-|S_r|} |S_r|^{|S_r|+\frac{1}{2}}}{\e^{-|S_r|+d_{1,r}+d_{2,r}} (|S_r|-d_{1,r}-d_{2,r})^{|S_r|-d_{1,r}-d_{2,r}+\frac{1}{2}}} \\
    &= \e^{-d_{1,r}-d_{2,r}} \frac{(|S_r|-d_{1,r}-d_{2,r})^{d_{1,r}+d_{2,r}}}{(1-\frac{d_{1,r}+d_{2,r}}{|S_r|})^{|S_r|+\frac{1}{2}}}\\
    &\asymp  \frac{(|S_r|-d_{1,r}-d_{2,r})^{d_{1,r}+d_{2,r}}}{(1-\frac{d_{1,r}+d_{2,r}}{|S_r|})^{\frac{1}{2}}}\\
    &\asymp \big(n\rho_r\big(1+O(\log ^{-2}n )\big) \big) ^{d_{1,r}+d_{2,r}} 
    \asymp (n\rho_r)^{d_{1,r}+d_{2,r}}, 
\end{align*}
where in the last step we have used 
$d_{1,r},d_{2,r} = o(\log ^{3/2} n)$.
Also, 
\[(1-\alpha)^{|S_r|-d_{1,r} - d_{2,r}} \asymp \e^{-\alpha(|S_r|-d_{1,r} - d_{2,r})} \asymp \e^{-t\rho_r\log(n)/2}. \]
Thus, the $r$-th product term in \eqref{exact-deg=prob} is asymptotically equal to 
\begin{align*}
    \e^{-t\rho_r\log n} \frac{(\rho_r\psi_r t\log n)^{d_{1,r}} (\rho_r(1-\psi_r) t\log n)^{d_{2,r}}}{d_{1,r}!d_{2,r}!}.
\end{align*}
The identical approximation holds for $r = i$ as well using the fact that $|S_i\setminus V| = n \rho_i (1+O(\log^{-2} n ))$. Thus the proof follows from \eqref{exact-deg=prob}. 
\end{proof}

\section{Proof of Lemma \ref{lemma:Schudy-Sviridenko}}\label{sec:appendix-sum-concentration}
We first state a special case of \cite[Theorem 1.3]{Schudy2012}.
\begin{lemma}\label{lemma:SS-special-case}
Let $X_1, \dots, X_n$ be independent random variables, and let $X = \sum_{i=1}^n X_i$. Let $L > 0$. Suppose that for each $i \in [n]$ and $k \in \mathbb{N}$,
\begin{equation}
\mathbb{E}\left[|X_i|^k\right] \leq k \cdot L \cdot \mathbb{E}\left[|X_i|^{k-1} \right]. \label{eq:L-bounded}
\end{equation}
Let $\mu_0 = \sum_{i=1}^n \mathbb{E}[|X_i|]$. Then, for any $\lambda>0$,
\begin{eq}
&\mathbb{P}\left(\left|X - \mathbb{E}[X] \right| \geq \lambda \right) \\
&\leq \e^2 \max\left\{\exp \left(-\frac{\lambda^2}{\mu_0 L R} \right), \exp \left(-\frac{\lambda}{L R} \right) \right\},\label{eq:SS-special-case}
\end{eq}
where $R$ is an absolute constant. 
\end{lemma}

\begin{proof}[Proof of Lemma \ref{lemma:Schudy-Sviridenko}]
We first verify \eqref{eq:L-bounded}, for $L = \max \{|x| : x \in S\}$. For $k \in \mathbb{N}$,
\begin{align*}
\mathbb{E}\left[|X_i|^k \right] &= \sum_{x \in S} |x|^k \mathbb{P}(X_i = x)\\ &\leq L\sum_{x \in S} |x|^{k-1} \mathbb{P}(X_i = x)\\
&= L \mathbb{E}\left[|X_i|^{k-1} \right]\\
&\leq kL \mathbb{E}\left[|X_i|^{k-1} \right].
\end{align*}
We apply Lemma \ref{lemma:SS-special-case}, noting that $\mu_0 = \mathbb{E}[Y]$. Set $\lambda = \delta |\mathbb{E}[X]|$. Then $\lambda < \mu_0$, so that the first bound in \eqref{eq:SS-special-case} applies, giving the claim.
\end{proof}

\section{Proof of eigenvalue properties}\label{sec:eigenvalue-properties}
\begin{proof}[Proof of Lemma \ref{edgeEigenvalue}]
The only $2\times 2$ matrices whose eigenvalues are not distinct are the multiples of the identity matrix. 
Indeed, 
\begin{align*}
    \det \begin{pmatrix}
        a&b\\
        b&c
    \end{pmatrix} = 0 
    \implies \lambda^2 - (a+c)\lambda + ac - b^2= 0,
\end{align*}
which has same roots in $\lambda$ if and only if $(a-c)^2+ b^2 = 0$. Therefore, $A'$ in \eqref{defn:2-2-matrix} has identical eigenvalues if and only if $y = q/(1-q)$. 

Next, note that $\det(A'(y))$ is a quadratic function in $y$, which has at most two roots unless $\det(A'(y))$ is the zero polynomial. 
To rule out the latter possibility, note that 
\begin{align*}
&\det\bigg(A'\Big(\frac{q}{1-q}\Big)\bigg) \\
&= \rho (1-\rho) \Big(p_1-\frac{q}{1-q}(1-p_1)\Big)\Big(p_2-\frac{q}{1-q}(1-p_2)\Big),    
\end{align*}
which is nonzero if $p_1$ and $p_2$ are both different from $q$. 
When $p_2 = q$, then $\det(A'(0)) = \rho (1-\rho) (p_1q - q^2)$, which cannot be zero due to our assumption that $0<p_1.p_2,q<1$ cannot be all be the same. 
Hence, $\det(A'(y))$ cannot be a zero polynomial, and thus there are at most two values of $y$ such that $\det(A'(y)) = 0$. Combined with the condition for having distinct roots,
there is a set $\cY$ with $|\cY| \leq 3$ such that, for $y\notin \cY$, $A'(y)$ has two distinct and nonzero eigenvalues. This proves Lemma~\ref{edgeEigenvalue}~Part~\eqref{edgeEigenvalue-1}.

To prove Lemma~\ref{edgeEigenvalue}~Parts~\eqref{edgeEigenvalue-2},\eqref{edgeEigenvalue-3}, consider the case where \eqref{spec-one-success-regime} holds. 
Also, take $y=y(p,q)=\log\big(\frac{1-q}{1-p}\big)/\log \big(\frac{p}{q}\big)$. Due to symmetry of $y(p,q)$, \eqref{eq:y-inequality-1} \text{and} \eqref{eq:y-inequality-2} imply 
\begin{align*}
&y(p,q) > \frac{q}{1-q} \text{ if }p>q \\
\text{and} \quad &y(p,q)< \frac{q}{1-q} \text{ if }p<q    
\end{align*}
Also, $y(p,q)>0$ for all $0<p,q<1$. Thus, if  \eqref{spec-one-success-regime}, then $y(p,q) \ne\frac{q}{1-q}$. 
Hence, the off diagonal entries of $A'(y(p,q))$ are nonzero and the eigenvalues of $A'(y(p,q))$ are distinct. 
If $p_2=q$ then $\det(A'(y(p,q)))= \rho(1-\rho)(p-q)(1+y)(q-y(1-q))\ne 0$, so its eigenvalues are both nonzero. This proves Part~\eqref{edgeEigenvalue-2}. 

To prove Part~\eqref{edgeEigenvalue-3}, take $p_1= p_2 = p$. Note that  $\det(A'(y)) = 0$ is $0$ if and only if $|p-y(1-p)|=|q-y(1-q)|$.
Since $p\ne q$, the latter holds if and only if  $y=\frac{p+q}{2-p-q}$ or $p+q = 1$. 
Let $x=(p+q)/2$ and $\epsilon=(p-q)/2$, and observe that
\begin{align*}
&y(p,q)=\ln\Big(\frac{1-q}{1-p_1}\Big)/\ln\Big(\frac{p}{q}\Big)\\
&=\ln\left(\frac{1-x+\epsilon}{1-x-\epsilon}\right)/\ln\left(\frac{x+\epsilon}{x-\epsilon}\right)\\
&=\frac{\ln(1+\epsilon/(1-x))-\ln(1-\epsilon/(1-x))}{\ln(1+\epsilon/x)-\ln(1-\epsilon/x)}\\
&=\frac{\sum_{i=0}^\infty 2(\epsilon/(1-x))^{2i+1}/(2i+1)}{\sum_{i=0}^\infty 2(\epsilon/x)^{2i+1}/(2i+1)}\\
&=\frac{x}{1-x}\cdot \frac{\sum_{i=0}^\infty (\epsilon/(1-x))^{2i}/(2i+1)}{\sum_{i=0}^\infty (\epsilon/x)^{2i}/(2i+1)}. 
\end{align*}
The ratio of infinite sums is greater than $1$ if $x<1/2$ and less than $1$ if $x>1/2$, so $y=\frac{p+q}{2-p-q}\iff 2x = p+q=1$. This proves Part~\eqref{edgeEigenvalue-3}.
\end{proof}

\begin{proof}[Proof of Lemma~\ref{edgeEigenvalue-general-k}]
First, observe that $P^{\sss (y)}$  has a zero eigenvalue if and only if its determinant is zero. Since $\det(P^{\sss (y)})$ is a polynomial in $y$ of degree at most $k$, either it is identically zero (i.e., (1) holds) or there exists a subset $\cY_1\subset \R$ with $|\cY_1| \leq k$ such that $\det(P^{\sss (y)})\neq 0$ for all $y\notin \cY_1$. 

Next, observe that for any given $y$, the eigenvalues of $P^{\sss (y)}$ are the roots of the characteristic polynomial $\chi^{\sss (y)}(\lambda):=\det(P^{(y)}-\lambda I)$.
Let $f^{\sss (y)}(\lambda)$ be the polynomial with leading coefficient $1$ in $\lambda$ that is the greatest common divisor of $\chi^{\sss (y)}(\lambda)$ and $(\chi^{\sss (y)})'(\lambda) =  \frac{\dif}{\dif \lambda}\chi^{\sss (y)}(\lambda)$. 
Then $\chi^{\sss (y)}(\lambda)$ has repeated roots in $\lambda$ if and only if $f^{\sss (y)} (\lambda)$ is not a constant function in $\lambda$. 
Now, consider $\chi^{\sss (y)}(\lambda)$ and $\frac{\dif}{\dif \lambda}\chi^{\sss (y)}(\lambda)$ as elements of $\mathbb{R}^{\sss(y)}[\lambda]$, the ring of polynomials in $\lambda$ with coefficients that are rational functions of $y$. 
Then, there exist $f^{\star \sss (y)}, g_1^{\sss (y)}, g^{\sss (y)}_2, h_1^{\sss (y)},h_2^{\sss (y)} \in \mathbb{R}^{\sss (y)}[\lambda]$ such that the leading coefficient of $f^{\star \sss (y)}$ is $1$, and
\begin{align*}
    &f^{\star \sss (y)}=g_1^{\sss (y)}\chi^{\sss (y)} + g_2^{\sss (y)} (\chi^{\sss (y)} )', \\
    &\chi^{\sss (y)} = h_1^{\sss (y)} f^{\star \sss (y)}, \quad (\chi^{\sss (y)})' = h_2^{\sss (y)} f^{\star \sss (y)}.
\end{align*}
Thus, for any $y$, $f^{\star \sss (y)}$ will evaluate to $f^{\sss (y)}$, unless the denominator of at least one coefficient of $f^{\star \sss (y)}$, $g_1^{\sss (y)}$, $g_2^{\sss (y)}$, $h_1^{\sss (y)}$, or $h_2^{\sss (y)}$ evaluates to $0$. 
Since the coefficients are rational functions in $y$, this can only happen for $y\in \cY_2$, where $\cY_2$ is a finite set. Therefore, if $f^{\star \sss (y)}(\lambda)$ is a constant in $\lambda$, then for all $y\notin \cY_2$, the eigenvalues of $P^{\sss (y)}$ are all distinct. Taking $\cY = \cY_1\cup \cY_2$, we have shown the (3) holds in this case.

Next, suppose that $f^{\star \sss (y)}(\lambda)$ is not  constant in $\lambda$. Then, $P^{\sss (y)}$ must have a repeated eigenvalue for all $y \notin \cY_2$. The eigenvalues of $P^{\sss (y)}$ change continuously as functions of $y$. Thus, if there was any $y$ for which its eigenvalues were all distinct they would have to be distinct for all values of $y$ sufficiently close to that one. Therefore, $P^{\sss (y)}$ must have repeated eigenvalues for all values of $y$ and (2) holds in this case. This completes the proof of Lemma~\ref{edgeEigenvalue-general-k}.
\end{proof}

\end{document}